\documentclass[a4paper,reqno]{amsart}
\usepackage[paperheight=11in,paperwidth=8.5in]{geometry}
\usepackage{amsthm,amssymb,amsmath,hyperref}
\usepackage[utf8]{inputenc}
\usepackage{dsfont}
\usepackage{color, mdframed}
\usepackage{tikz}
\usetikzlibrary{decorations.pathreplacing}
\usepackage[utf8]{inputenc}
\usepackage{graphicx,enumerate}
\usepackage{mathtools,bm}
\usepackage{enumitem}
\usepackage{stmaryrd}
\allowdisplaybreaks

\makeatletter
\newcommand{\setword}[2]{%
  \phantomsection
  #1\def\@currentlabel{\unexpanded{#1}}\label{#2}%
}
\def\l@section{\@tocline{1}{12pt plus2pt}{0pt}{}{\bfseries}}
\def\l@subsection{\@tocline{2}{0pt}{2pc}{2pc}{}}
\makeatother
\makeatletter
\def\subsection{\@startsection{subsection}{2}{\z@}%
	{-3.25ex\@plus -1ex \@minus -.2ex}%
	{1.5ex \@plus .2ex}%
	{\normalfont\bfseries\boldmath}}
\def\subsubsection{\@startsection{subsubsection}{3}%
	\z@{.5\linespacing\@plus.7\linespacing}{-.5em}%
	{\normalfont\bfseries\boldmath}}
\renewcommand\paragraph{\@startsection{paragraph}{4}{\z@}%
	{3.25ex \@plus1ex \@minus.2ex}%
	{-1em}%
	{\normalfont\normalsize\bfseries}}
\makeatother

\theoremstyle{plain}
\newtheorem{thm}{Theorem}[section]

\newtheorem{mthm}[thm]{Main Theorem}

\newtheorem{lem}[thm]{Lemma}
\newtheorem{prop}[thm]{Proposition}

\theoremstyle{definition}

\theoremstyle{remark}
\newtheorem{rem}[thm]{Remark}
\newtheorem{obs}[thm]{Observation}

\theoremstyle{plain}

\numberwithin{equation}{section}

\theoremstyle{plain} 
\newcommand{\thistheoremname}{}
\newtheorem{genericthm}[thm]{\thistheoremname}

  \newtheorem*{genericthm*}{\thistheoremname}
\newenvironment{namedthm*}[1]
  {\renewcommand{\thistheoremname}{#1}%
   \begin{genericthm*}}
  {\end{genericthm*}}

\newcommand{\ep}{\epsilon}

\newcommand{\del}{\delta}

\newcommand{\G}{\Gamma}

\newcommand{\R}{{\mathbb R}}

\newcommand{\N}{{\mathbb N}}

\newcommand{\Z}{{\mathbb Z}}

\newcommand{\dist}{\hbox{ \rm dist}}

\newcommand{\calA}{{\mathcal A}}
\newcommand{\calB}{{\mathcal B}}
\newcommand{\calC}{{\mathcal C}}
\newcommand{\calD}{{\mathcal D}}

\newcommand{\calF}{{\mathcal F}}

\newcommand{\calH}{{\mathcal H}}
\newcommand{\calI}{{\mathcal I}}

\newcommand{\calL}{{\mathcal L}}
\newcommand{\calM}{{\mathcal M }}

\newcommand{\calS}{{\mathcal S}}

\newcommand{\calU}{{\mathcal U}}
\newcommand{\calV}{{\mathcal V}}

\newcommand{\supp}{{\textrm{supp}}}

\def\g{\gamma}

\makeatletter
\newcommand{\vast}{\bBigg@{4}}
\newcommand{\Vast}{\bBigg@{5}}
\makeatother

\newcommand{\frakA}{{\mathfrak A}}
\newcommand{\frakB}{{\mathfrak B}}

\newcommand{\frakm}{{\mathfrak m}}
\newcommand{\frakM}{{\mathfrak M}}

\newcommand{\frakL}{{\mathfrak L}}

\newcommand{\TFC}{\textnormal{{\bf TFC}}}
\newcommand{\sTFC}{\textnormal{{$\mathsf{TFC}$}}}

\def\udot#1{\ifmmode\oalign{$#1$\crcr\hidewidth.\hidewidth
    }\else\oalign{#1\crcr\hidewidth.\hidewidth}\fi}
    
\def\a{\alpha}
\def\R{\mathbb{R}}
\def\Z{\mathbb{Z}}

\def\beq{\begin{equation}}
\def\eeq{\end{equation}}
\makeatletter
\newcommand{\doublewidetilde}[1]{{%
  \mathpalette\double@widetilde{#1}%
}}
\newcommand{\double@widetilde}[2]{%
  \sbox\z@{$\m@th#1\widetilde{#2}$}%
  \ht\z@=.9\ht\z@
  \widetilde{\box\z@}%
}

\makeatother

\def\one{\mbox{1\hspace{-4.25pt}\fontsize{12}{14.4}\selectfont\textrm{1}}}

\makeatletter
\def\@makefnmark{%
  \leavevmode
  \raise.9ex\hbox{\fontsize\sf@size\z@\normalfont\tiny\@thefnmark}}
\makeatother

\begin{document}
	
\title[]{\Large{The Bilinear Hilbert--Carleson operator along curves}\\\normalsize{The purely non-zero curvature case}}
\author{\'Arp\'ad B\'enyi, Bingyang Hu and Victor Lie}

\address{\'Arp\'ad B\'enyi: Department of Mathematics, Western Washington University, 516 High Street, Bellingham, WA 98225,  U.S.A.}%
\email{benyia@wwu.edu}

\address{Bingyang Hu: Department of Mathematics, Auburn University, 221 Parker Hall, Auburn, AL 36849, U.S.A.}%
\email{bzh0108@auburn.edu}

\address{Victor D. Lie: Department of Mathematics, Purdue University, 150 N. University St, W. Lafayette, IN 47907, U.S.A.  and
The ``Simion Stoilow" Institute of Mathematics of the Romanian Academy, Bucharest, RO 70700, P.O. Box 1-764, Romania}%
\email{vlie@purdue.edu}

\begin{abstract} In this paper, we provide the maximal boundedness range (up to end-points) for the Bilinear Hilbert--Carleson operator along curves in the (purely) non-zero curvature setting. More precisely, we show that the operator 
$$  BHC_{[\vec{a},\vec{\a}]}(f_1,f_2)(x):=\sup_{\lambda\in\R}\left|\,\textnormal{p.v.}\int_{\R} f_1(x- a_1 t^{\a_1})\,f_2(x- a_2 t^{\a_2})\,e^{i\,\lambda\,a_3 \,t^{\a_3}}\,\frac{dt}{t}\right|$$
obeys the bounds  
$$\|BHC_{[\vec{a},\vec{\a}]}(f_1,f_2)\|_{L^r}\lesssim_{\vec{a}\,\vec{\a},r,p_1,p_2}\|f_1\|_{L^{p_1}}\,\|f_2\|_{L^{p_2}}$$ 
whenever  $\vec{a}=(a_1,a_2,a_3),\,\vec{\a}=(\a_1,\a_2,\a_3)\in (\R\setminus\{0\})^3$ with $\vec{\a}$ having pairwise distinct coordinates and 
for any H\"older range $\frac{1}{p_1}+\frac{1}{p_2}=\frac{1}{r}$ with $1<p_1,p_2<\infty$ and $\frac{1}{2}<r<\infty$.

This result is achieved via the Rank II LGC method introduced in \cite{HL23}.
\end{abstract}
\date{\today}


\thanks{}

\maketitle

\setcounter{tocdepth}{2}
\tableofcontents

\section{Introduction}

In this paper we focus on the study of singular integral operators that encompass the following features:
\begin{itemize}
\item a \textsf{multi(sub)linear} nature;

\item a \textsf{commutation} relation relative to the standard symmetries, \textit{i.e.},  translations and dilations;

\item a \textsf{Carleson-type} behavior, \textit{i.e.}, the defining object may be expressed as a supremum over a family of multilinear integral operators with (generalized) modulated singular kernel.
\end{itemize}

Indeed, our goal in this paper is to understand the behavior of the \emph{purely curved (non-resonant)} case of the generic \emph{Bilinear Hilbert--Carleson operator along $\vec{\g}$} defined as\footnote{Throughout the paper, all the singular integral expression are apriori only defined for Schwartz functions.}
\begin{equation}\label{BHCcurv}
        BHC_{[\vec{\g}]}(f_1,f_2)(x):=\sup_{\lambda\in\R}\left|\: \textrm{p.v.} \int_{\R} f_1(x-\g_1(t))\,f_2(x-\g_2(t))\,e^{i\,\lambda\,\g_3(t)}\,\frac{dt}{t}\right|\,,
\end{equation}
where here $\vec{\g}(t)=(\g_1(t), \g_2(t), \g_3(t))$ is a piecewise differentiable curve in $\R^3$.

While the deeper motivation behind the study of \eqref{BHCcurv} will be explained in detail in Section \ref{MotivHistory}, see in particular Remark \ref{ConIII2} therein, for the moment being, we summarize its key underlying justification as follows: the behavior of $BHC_{[\vec{\g}]}$ is \emph{connected in a fundamental, conceptual way} to the behavior of the \emph{trilinear Hilbert transform along $\vec{\g}$} given by 
\begin{equation}\label{ThTcurv}
        THT_{[\vec{\g}]}(f_1,f_2,f_3)(x):= \,\textrm{p.v.}\int_{\R} f_1(x-\g_1(t))\,f_2(x-\g_2(t))\,f_3(x-\g_3(t))\,\frac{dt}{t}\,,
\end{equation}
with the latter being a central object of interest whose general understanding constitutes an important open problem in harmonic analysis.

With these being said, we pass now to

\subsection{Statement of the main result}\label{Mainres}

\begin{mthm} \label{mainresult} \textit{Let \footnote{In the case $\a\notin\Z$ one has to properly define the meaning of $t^{\a}$ as either $|t|^{\a}$ or $\textrm{sgn}{t}\,|t|^{\a}$.} $\vec{\g}(t)=(a_1 t^{\a_1}, a_2 t^{\a_2}, a_3 t^{\a_3})$ with $\vec{a}=(a_1,a_2,a_3),\,\vec{\a}=(\a_1,\a_2,\a_3)\in (\R\setminus\{0\})^3$  and set
\begin{equation}\label{BHCcurval}
        BHC_{[\vec{a},\vec{\a}]}(f_1,f_2)(x):=\sup_{\lambda\in\R}\left|\,\emph{p.v.}\int_{\R} f_1(x- a_1 t^{\a_1})\,f_2(x- a_2 t^{\a_2})\,e^{i\,\lambda\,a_3 \,t^{\a_3}}\,\frac{dt}{t}\right|\,.
\end{equation}
If $\vec{\a}$ is a vector with \emph{pairwise distinct} coordinates then
the purely nonresonant\footnote{\textit{I.e.}, the operator under discussion has no modulation invariance structure.} Bilinear Hilbert--Carleson operator $BHC_{[\vec{a},\vec{\a}]}(f_1,f_2)$ obeys\footnote{Here and henceforth, for $a, b \in \R$, $a \lesssim b$ ($a \gtrsim b$, respectively) means there exists an absolute $C>0$, which is independent of $a$ and $b$, such that $a \le Cb$ ($a \ge Cb$, respectively). Moreover, if both $a\lesssim b$ and $a \gtrsim b$ hold, then we say $a \sim b$.}
\begin{equation}\label{BHCcurval1}
        \|BHC_{[\vec{a},\vec{\a}]}(f_1,f_2)\|_{L^r}\lesssim_{\vec{a}\,\vec{\a},r,p_1,p_2}\|f_1\|_{L^{p_1}}\,\|f_2\|_{L^{p_2}}
\end{equation}
for any H\"older range $\frac{1}{p_1}+\frac{1}{p_2}=\frac{1}{r}$ with $1<p_1,p_2<\infty$ and $\frac{1}{2}<r<\infty$.}
\end{mthm}

\begin{rem}\label{mainresultconse} As a confirmation of the natural companionship and already alluded deep interrelationship between \eqref{BHCcurv} and \eqref{ThTcurv}, we have that many elements in the proof of the above main theorem \emph{both rely and are relied upon}\footnote{For more on this please see Remark \ref{ConIII2}, Remark \ref{twoway} and, more generally, Section \ref{Planmainideas}.} when treating the trilinear counterpart \eqref{ThTcurv}, as witnessed by the work in \cite{HL23}:

 \textit{Under the same (purely) nonresonant hypothesis on $\vec{\a}$ as in Main Theorem \ref{mainresult}, we have that the curved trilinear Hilbert transform 
\begin{equation}\label{ThTcurval}
        THT_{[\vec{a},\vec{\a}]}(f_1,f_2,f_3)(x):= \,\emph{p.v.}\int_{\R} f_1(x- a_1 t^{\a_1})\,f_2(x- a_2 t^{\a_2})\,f_3(x-a_3 t^{\a_3})\,\frac{dt}{t}\,.
\end{equation}
obeys
\begin{equation}\label{ThTcurval1}
        \|THT_{[\vec{a},\vec{\a}]}(f_1,f_2,f_3)\|_{L^r}\lesssim_{\vec{a}\,\vec{\a},r,p_1,p_2,p_3}\|f_1\|_{L^{p_1}}\,\|f_2\|_{L^{p_2}}\,\|f_3\|_{L^{p_3}}
\end{equation}
for any Banach H\"older range $\frac{1}{p_1}+\frac{1}{p_2}+\frac{1}{p_3}=\frac{1}{r}$ with $1<p_1,p_2,p_3<\infty$ and $1\leq r<\infty$.}
\end{rem}

\begin{rem}\label{Gen1} Both proofs of the boundedness of \eqref{BHCcurval} and \eqref{ThTcurval} in the purely non-zero curvature setting as stated in Main Theorem \ref{mainresult} and Remark \ref{mainresultconse}, respectively, employ the Rank II LGC--method introduced in \cite{HL23}. Indeed, in contrast with the reasonings in \cite{Lie2024}, \cite{GL25} and \cite{HsL24} but similar with those in \cite{BBLV24} and \cite{HL23}, the linearized discrretized model of the operator in the present paper---obtained via the Rank I LGC--method---is \emph{not absolutely summable}. We stress here that the lack of absolute summability appearing in our context is a phenomenon \emph{within each single scale and not in-between scales}. The latter situation has been encountered before, see \emph{e.g.} \cite{Th02}, and, comparatively, is easier to address via partial integration, Plancherel and telescoping sum arguments. The former phenomenon however, is much more subtle as it captures hidden cancellations within the \emph{same-scale} local Fourier coefficients of the input functions: handling this type of challenge is the main strength and novelty as well as the very reason for designing the Rank II LGC--method.  Indeed, at the heart of this method lies a \emph{correlative} time-frequency discretized model that exploits the \emph{interaction} of the Fourier modes of the input functions and whose control is achieved via the following key ingredients:
\begin{enumerate} 
\item  a \emph{sparse-uniform decomposition} of the input functions adapted to a foliation of the phase-space;
  
\item  a structural analysis of some appropriate \emph{maximal joint Fourier coefficients};

\item  a level set analysis with respect to the \emph{time-frequency correlation set}.
\end{enumerate}
\end{rem}

\begin{rem}\label{Gen2A} The Bilinear Hilbert--Carleson operator defined by \eqref{BHCcurval1} and corresponding to the \emph{hybrid quadratic non-resonant} case  $\a_1=\a_2=1$ and $\a_3\notin\{1,2\}$ was introduced and studied in \cite{BBLV24}. Overall, this hybrid case is strictly harder than the purely non-resonant case studied in the present paper since in the former case one has to confront the presence of the modulation invariance structure expressed in \eqref{hybnresI}. However, a careful comparative analysis of the two proofs reveals the following deeper antithesis:
\begin{itemize}  
\item on the one hand, the \emph{single scale ($L^2$--base) smoothing inequality}\footnote{See the Main Proposition (informal version) in  Section \ref{Planmainideas} versus the corresponding Proposition 4.1. in \cite{BBLV24}.}  is significantly more difficult in the purely non-resonant case as opposed to the hybrid case. Indeed the crucial distinction is made by the un-isotropic (in the purely non-resonant case) versus isotropic (in the hybrid case) character of the $t-$argument corresponding to the input functions. This is turn requires for the  non-resonant case the introduction of a novel approach relying on the so-called constancy propagation of the linearizing function---see the discussion in Section \ref{Planmainideas};

\item on the other hand, the \emph{multi-scale analysis} is significantly more difficult in the hybrid case as opposed to the purely non-resonant case due to the extra modulation invariance structure of the former case which requires a BHT/Carleson type analysis involving tree selection algorithms based on the concepts of size, energy and mass---see Section 6 in \cite{BBLV24}.  
\end{itemize}
\end{rem}

\begin{rem}\label{Gen3}  Building on the context offered by Main Theorem \ref{mainresult}, Remark \ref{mainresultconse} and Remark \ref{Gen2A} it becomes apparent that the studies of the Bilinear Hilbert--Carleson operator $ BHC_{[\vec{a},\vec{\a}]} $ defined by \eqref{BHCcurval} and of the Trilinear Hilbert Transform $THT_{[\vec{a},\vec{\a}]}$ defined by \eqref{ThTcurval} are not only intimately connected to each-other but also on the specific choice of the vector $\vec{\a}=(\a_1,\a_2,\a_3)\in \R_{+}^3$. Consequently, we are naturally led to consider the following \emph{hierarchy}:
\begin{itemize}
\medskip
\item \textsf{The purely \underline{\emph{non-resonant}} case: $\{\a_j\}_{j=1}^3$ covers \underline{\emph{three}} distinct values}. In this situation neither $ BHC_{[\vec{a},\vec{\a}]} $ nor $THT_{[\vec{a},\vec{\a}]}$  have any modulation invariance symmetries. The full treatment of the former object is provided in the present paper while the treatment of the latter object is (essentially) covered in \cite{HL23}.
\medskip
\item \textsf{The \underline{\emph{hybrid}} case: $\{\a_j\}_{j=1}^3$ covers \underline{\emph{two}} distinct values}. In this situation both objects under consideration have suitable modulation invariance properties. However, their analysis splits in two major subcases:
\begin{itemize} 
 \item \textsf{\emph{hybrid quadratic non-resonant}} case: 
$$\qquad\qquad\exists\,i\not=k\in\{1,2,3\}\:\:\textrm{with}\:\:\a_i=\a_k=\min\{\a_j\}_{j=1}^3\notin \left\{\frac{1}{2}\max\{\a_j\}_{j=1}^3,\,\max\{\a_j\}_{j=1}^3\right\}\,.$$
\quad In this situation we further split our discussion according to the nature of our operator:
\begin{itemize}
\item for the Trilinear Hilbert Transform $THT_{[\vec{a},\vec{\a}]}$, due to the input symmetry and the presence of the dilation invariant Haar measure $\frac{dt}{t}$, we may assume wlog that $\a_1=\a_2=1$ and $\a_3\notin\{1,2\}$. As a consequence, we have the linear modulation\footnote{Throughout the paper, we denote the $j$-order modulation with parameter $a\in\R$ by $M_{j,a} f(x):=e^{i a x^j}\,f(x)$ with $j\in\N$.} invariance
    \begin{equation}\label{hybnres}
    \left| THT_{[\vec{a},\vec{\a}]} \left(M_{1,\frac{c}{a_1}}f_1, M_{1,-\frac{c}{a_2}}f_2, f_3\right)\right|=\left| THT_{[\vec{a},\vec{\a}]} \left(f_1, f_2, f_3\right)\right|\qquad \forall\:c\in\R\,.
    \end{equation}

\item for the Bilinear Hilbert--Carleson operator $ BHC_{[\vec{a},\vec{\a}]}$ we still have the dilation invariance of the measure but we lack the input symmetry and hence there is not a unique case to consider. Nevertheless, a natural prototype for the hybrid quadratic non-resonant case which is also the analogue of \eqref{hybnres}, is the situation $\a_1=\a_2=1$ and $\a_3\notin\{1,2\}$ given by
   \begin{equation}\label{hybnresI}
     BHC_{[\vec{a},\vec{\a}]} \left(M_{1,\frac{c}{a_1}}f_1, M_{1,-\frac{c}{a_2}}f_2\right)=BHC_{[\vec{a},\vec{\a}]} \left(f_1, f_2\right)\qquad \forall\:c\in\R\,.
    \end{equation}   
This last operator has been addressed in \cite{BBLV24}.
   \end{itemize}
\item  \textsf{\emph{hybrid quadratic resonant}} case:
$$\qquad\qquad\exists\,i\not=k\in\{1,2,3\}\:\:\textrm{with}\:\:\a_i=\a_k=\min\{\a_j\}_{j=1}^3=\frac{1}{2}\max\{\a_j\}_{j=1}^3\,.$$
\quad As before, we further split our discussion according to the nature of our operator:
\begin{itemize}
\item for the Trilinear Hilbert Transform $THT_{[\vec{a},\vec{\a}]}$, appealing to the input symmetry and dilation invariance, we may assume wlog that $\a_1=\a_2=1$ and $\a_3=2$. As a consequence, on top of \eqref{hybnres}, we now also have the quadratic modulation invariance
    \begin{equation}\label{hybres}
    \left| THT_{[\vec{a},\vec{\a}]} \left(M_{2,\frac{c}{a_1}}f_1, M_{2,-\frac{c}{a_2}}f_2, M_{1, \frac{c}{\a_3}(\a_1-\a_2)} f_3\right)\right|=\left| THT_{[\vec{a},\vec{\a}]} \left(f_1, f_2, f_3\right)\right|\qquad \forall\:c\in\R\,.
    \end{equation}

\item for the Bilinear Hilbert--Carleson operator $ BHC_{[\vec{a},\vec{\a}]}$, again, there are more options, with the natural analogue of \eqref{hybres} given by the case $\a_1=\a_2=1$ and $\a_3=2$. On top of \eqref{hybnresI} we now have 
   \begin{equation}\label{hybresI}
     BHC_{[\vec{a},\vec{\a}]} \left(M_{2,\frac{c}{a_1}}f_1, M_{2,-\frac{c}{a_2}}f_2\right)=BHC_{[\vec{a},\vec{\a}]} \left(f_1, f_2\right)\qquad \forall\:c\in\R\,.
    \end{equation}   
\end{itemize}
The hybrid quadratic resonant case is widely open. Any significant advancement would require some fundamental new ideas. 
\end{itemize}
\medskip
\item \textsf{The purely \underline{\emph{resonant}} case: $\{\a_j\}_{j=1}^3$ covers \underline{\emph{one}} single value}.  This is of course the most difficult case from all of the ones listed above. By applying a standard dilation invariance argument one can reduce matters to the situation $\a_1=\a_2=\a_3=1$. Also, for the simplicity of the exposition we set $\vec{a}=(a_1,a_2,a_3)=(1,2,3)$ and with these choices for $\vec{a},\vec{\a}$ we simply refer to  $THT_{[\vec{a},\vec{\a}]}$ as $\textsf{THT}$ and, with the obvious adaptations, to $BHC_{[\vec{a},\vec{\a}]}$ as $\textsf{BHC}$. With these, we have
\begin{itemize}   
 \item the (zero-curvature) Trilinear Hilbert Transform $\textsf{THT}$ obeys the modulation invariance relations
    \begin{equation}\label{thtpres}
    \left| \textsf{THT} \left(M_{1,3b_1}f_1, M_{1,3b_2}f_2, M_{1, -b_1-2b_2} f_3\right)\right|=\left| \textsf{THT}\left(f_1, f_2, f_3\right)\right|\qquad \forall\:b_1,\,b_2\in\R\,,
    \end{equation}
and
    \begin{equation}\label{thtpresquad}
    \left| \textsf{THT} \left(M_{2,3b}f_1, M_{2,-3b}f_2, M_{2,b} f_3\right)\right|=\left| \textsf{THT}\left(f_1, f_2, f_3\right)\right|\qquad \forall\:b\in\R\,.
    \end{equation}
 The infamous question addressing the boundedness of \textsf{THT} is widely open and represents one of the most difficult problems in harmonic analysis.\footnote{For more on this, the interested reader is invited to consult the introductory sections in \cite{Lie20} and \cite{HL23}.} 
    
\item the (zero-curvature) Bilinear Hilbert--Carleson operator $\textsf{BHC}$ obeys the modulation invariance relations
\begin{equation}\label{BHCpres1}
     \textsf{BHC} \left(M_{1,b_1}f_1, M_{1,b_2}f_2\right)=\textsf{BHC} \left(f_1, f_2\right)\qquad \forall\:b_1,\,b_2\in\R\,,
    \end{equation}   
and
 \begin{equation}\label{BHCpres2}
     \textsf{BHC}\left(M_{2,4b} f_1, M_{2,-b}f_2\right)=\textsf{BHC}\left(f_1, f_2\right)\qquad \forall\:b\in\R\,.
 \end{equation}   
 The question addressing the boundedness of \textsf{BHC} is also widely open and is directly related to another fundamental problem in classical Fourier analysis that of the pointwise convergence of the bilinear Fourier series.\footnote{For more on this one is invited to consult the Introduction in \cite{BBLV24}.}    
\end{itemize}

\end{itemize}

\end{rem}

\subsection{Motivation and historical background}\label{MotivHistory}

The perspective and the overall structure of the presentation in this section is inspired by a survey paper \cite{Liesurv} in which the author therein discusses several deep connections between multilinear singular operators and maximal multi(sub)linear operators of Carleson type.
\medskip

Fix $n\in\N$. For expository reasons, in what follows, we introduce two generic classes of operators\footnote{The case of foremost importance for our present paper is $n=3$.}:
\begin{itemize}

\item \underline{$1^{st}$ class}: \textsf{the $(n-1)$-linear $\vec{\G}$-Carleson operator} $C_{n-1,\vec{\G}}$, defined as
\beq\label{nhilBHCarl}
C_{n-1,\vec{\G}}(f_1,\ldots,f_{n-1})(x):= \textrm{p.v.}\int_{\R}\prod_{j=1}^{n-1} f_j(x-\underline{\g}_j(t))\,\frac{e^{i\,\g(x,t)}}{t}\,dt\,,
\eeq
where here $\vec{\underline{\g}}:=(\underline{\g}_1 ,\ldots ,\underline{\g}_{n-1}):\mathbb R \rightarrow \mathbb R^{n-1}$ is an $(n-1)-$variable tuple with each $\underline{\g}_j$ a suitable piecewise smooth planar curve, $\g(x,t)$ is a piecewise smooth function in $t$ and only measurable in $x$, and, finally $\vec{\G}(x,t):=(\vec{\underline{\g}}(t), \g(x,t))$.
\medskip

 \item \underline{$2^{nd}$ class}: \textsf{the $n$-linear \emph{$\vec{\g}$-Hilbert transform}} $H_{n,\vec{\g}}$, defined as
\beq\label{nhilb}
H_{n,\vec{\g}}(f_1,\ldots,f_n)(x):= \textrm{p.v.}\int_{\R}\prod_{j=1}^n f_j(x-\g_j(t))\,\frac{dt}{t}\,,
\eeq
 where here $\vec{\g}:=(\g_1 ,\ldots ,\g_n):\mathbb R \rightarrow \mathbb R^n$ is an $n-$variable tuple with each $\g_j$ being a suitable piecewise smooth planar curve.
 \medskip
\end{itemize}

Next, in the increasing order of complexity, we build a hierarchical structure within which we elaborate on the deep and subtle interplay between the two classes introduced above. For now, in preparation for the discussion below, we mention that at the heuristic level this interplay involves a ``\emph{degree reduction}", by trading the $n^{th}$ input function in \eqref{nhilb} for a maximal ($x-$dependent) oscillatory exponential in \eqref{nhilBHCarl}, thus connecting
$H_{n,\vec{\g}}$ to $C_{n-1,\vec{\G}}$. We end these preparatives by stating that our antithetical discussion follows in parallel the above two sets of objects depending on 
\begin{itemize}
\item the value of $n$;

\item the curvature paradigm:  zero, non-zero and hybrid curvature regimes.
\end{itemize}

\noindent\underline{\textbf{I) The case $n=1$.}} $\quad$ In this context, using the convention $\prod_{j=1}^{0} f_j\equiv1$ one may represent \eqref{nhilBHCarl} as a maximal oscillatory singular integral. With this interpretation, one can portray the work of E. Stein and S. Wainger in \cite{swmul} as providing uniform in $x-$bounds to \eqref{nhilBHCarl} whenever $\g(x,t)=\sum_{j=1}^d a_j(x)\,t^{\a_j}$ with $d\in\N$, $\{\a_j(\cdot)\}_j$ real measurable functions and $\{\a_j\}_j\subset\R_{+}$.

Moving now our focus on \eqref{nhilb}, it becomes transparent that this expression corresponds to the classical Hilbert transform along curves $H_{1,\g}$ which is a central object within harmonic analysis, and which, among others, serves---in the zero curvature setting---as a main prototype for the seminal Calder\'on-Zygmund theory (see e.g. \cite{CZ1} and \cite{CZ2}). 

Along the way, we notice that the $L^2$-boundedness of \eqref{nhilb} is equivalent with the $L^{\infty}$-boundedness (as a function of $x$) of \eqref{nhilBHCarl} in the setting $\g(x,t):=x\,\g_1(t)$.
\smallskip

\noindent\underline{\textbf{II.1) The case $n=2$, zero curvature.}} $\quad$ We start our examination with the first class \eqref{nhilBHCarl}: the zero curvature case corresponds to the situation $\underline{\g}_1(t)=t$ and $\g(x,t)=a(x)\,t$ with $a$ measurable function, and the operator defined by \eqref{nhilBHCarl} represents the classical Carleson operator. Historically, the study of this operator was motivated by the central question---laying at the foundation of the harmonic analysis area---regarding the pointwise convergence of Fourier Series for square integrable functions. This problem offered the first study case of an operator that is invariant under linear modulation. In 1966, L. Carleson proved in \cite{Car66} his celebrated result on the $L^2$ boundedness of $C_{1,\G}$. Later, R. Hunt, \cite{hu}, extended this result to any $L^p$, $1<p<\infty$.

The second class \eqref{nhilb} corresponds to $\vec{\g}(t)=(a t, b t)$, $a\not=b\in\R\setminus\{0\}$  and was motivated by A. Calder\'on's study of the Cauchy transform on Lipschitz curves (\cite{Cal}, \cite{CMM}). After more than two decades, the first key insight into this problem was due to M. Lacey, who in \cite{la1} noticed that both the Carleson operator $C_{1,\G}$ and the bilinear Hilbert transform $H_{2,\vec{\g}}$ share a similar modulation invariance structure thus hinting that the approach of $H_{2,\vec{\g}}$ must rely on time-frequency techniques and in particular wave-packet analysis. This intuition was indeed confirmed in the breakthrough papers \cite{lt1}, \cite{lt2}, where, by introducing a novel time-frequency framework,  M. Lacey and C. Thiele proved the $L^{p_1}\times L^{p_2}\,\longrightarrow\,L^r$ boundedness of $H_{2,\vec{\g}}$ within the range $\frac{1}{p_1}+\frac{1}{p_2}=\frac{1}{r}$, $1<p,q\leq\infty$ and $\frac{2}{3}<r<\infty$.

\begin{rem}
We end the overview of case II.1 with two commentaries:
\begin{itemize}
 \item this zero-curvature case confirms the strong connection between the two classes of operators introduced in \eqref{nhilBHCarl} and \eqref{nhilb}: indeed, the methods pioneered by L. Carleson, \cite{Car66} and C. Fefferman, \cite{f} in relation to $C_{1,\G}$ served both as an inspiration and proved quintessential in the successful approach of $H_{2,\vec{\g}}$ via the concepts of mass, energy, tiles and tree organization/selection; 
 \item due to the lack of absolute summability of the Gabor coefficients within the discretized model operator representing $H_{2,\vec{\g}}$, the question regarding its maximal boundedness range -- i.e., whether or not one can cover the regime $\frac{1}{2}<r\leq\frac{2}{3}$ -- remains an important open problem in the time-frequency area. For more on the latter, we invite the reader to examine Section 3.3. in \cite{HL23} and in particular Observation 3.3. therein.
\end{itemize}     
\end{rem}

\noindent\underline{\textbf{II.2) The case $n=2$, non-zero curvature.}} $\quad$ For the first class \eqref{nhilBHCarl} this line of research was motivated both by the maximal singular oscillatory integral study revealed at Case $n=1$ and, more importantly, by the study of singular/maximal oscillatory integral operators on Heisenberg groups (where one has an unisotropic behavior). E. Stein, \cite{s2}, proved the $L^2$ boundedness of \eqref{nhilBHCarl} in the regime $\underline{\g}_1(t)=t$ and $\g(x,t)=a(x) t^2$ while E. Stein and S. Wainger (\cite{sw}) covered the full $L^p$, $1<p<\infty$ range (and also its higher dimensional analogue) for $\underline{\g}_1(t)=t$ and $\g(x,t)=\sum_{j=2}^d a_j(x) t^j$ with $d\geq 2$. Their approach relied on the $TT^{*}-$method and Van der Corput techniques.

For the second class \eqref{nhilb} the motivation is twofold: as a curved analogue of the classical (zero curvature) Bilinear Hilbert transform and as a continuous analogue of the ergodic theoretical problem regarding the norm and pointwise convergence of Furstenberg non-conventional bilinear averages.\footnote{For more on the connections and related bibliography we invite the interested reader to consult the Introduction in \cite{Lie2024}.} In this non-zero curvature realm, the complete boundedness range for $H_{2,\vec{\g}}$ is now fully settled depending on the nature of the curves $\vec{\g}$ considered: 1) in \cite{Li13}, \cite{LX16}, this was done for monomials and polynomials, respectively, with the method of proof relying on the $\sigma$-uniformity concept introduced in \cite{cltt} and inspired by  Gowers's work in \cite{Gow98}; 2)  In \cite{Lie15}, \cite{Lie18}, \cite{GL20}, the analysis was performed on a larger class of ``non-flat" curves that included any Laurent or even generalized polynomial with no monomial of degree one with the method of proof combining two types of discretization: a Taylor approximation of second order for the phase of the multiplier that decouples the frequency variables and a suitable (non-linear) wave-packet analysis applied to the input functions.

\smallskip
\begin{rem} Based on the Rank I LGC method introduced by the third author, in \cite{Lie2024} and \cite{GL25}, a unified approach to \eqref{nhilBHCarl} and \eqref{nhilb} in the non-zero case was offered thus creating the link between these two cases. Another result establishing an implicit relation between \eqref{nhilBHCarl} and \eqref{nhilb} in the particular setting $\g_1(t)=\underline{\g}_1(t)=t$, $\g_2(t)=t^2$ and $\g(x.t)=a(x) t^2$  was presented in \cite{DR22} relying on Gowers uniformity norms.
\end{rem}

\noindent\underline{\textbf{II.3) The case $n=2$, hybrid curvature.}} \quad For the first class \eqref{nhilBHCarl} this situation arises as a natural symbiosis between the zero curvature case covering the celebrated theorem of Carleson and the non-zero curvature situation encompassing the previously mentioned results in \cite{s2} and \cite{sw}. In the latter, E. Stein conjectures that the Polynomial Carleson operator $C_{1,\G}$ with $\g(x,t)=\sum_{j=1}^d a_j(x) t^j$ is of strong type $(p,p)$ for $1<p<\infty$. The relevance of this theme goes much beyond the synthesis character mentioned above: indeed, the Polynomial Carleson operator is the first example considered in the literature of a linear operator that is invariant under higher order modulations, thus requiring a higher-order wave-packet analysis. In \cite{lv1}, \cite{Lie20} the third author answered in the affirmative this conjecture by developing  a new time-frequency framework for higher-order wave packets, a tile selection algorithm that removed the presence of exceptional sets as well as novel local analysis adapted to the mass and tree selection algorithms. It is worth noticing that this result remains to date the only treatment of an operator obeying higher order modulation invariance. Finally, further elaboration on these ideas produced higher dimensional analogues, see \cite{zk1}, \cite{lv3n} and more recently \cite{BDJST}.

For the second class \eqref{nhilb}, the motivation for considering the hybrid curvature case resides on the desire of better understanding the subtle phenomenon that governs the transition from the zero to non-zero curvature. In order to understand this theme one is expected to combine both the approaches used at items II.1) and II.2) above. Indeed, in \cite{GL25}, the authors therein employed the Rank I LGC method in order to obtain the boundedness of $H_{2,\vec{\g}}$ in the hybrid setting $\g_1(t)=t$ and $\g_2(t)=at +b t^{\alpha}$ for $a\in\R\setminus\{1\}$, $b\in\R$ and $\alpha\in(0,\infty)\setminus\{1,2\}$.

\begin{rem} A noteworthy fact in the hybrid case context of \eqref{nhilb} is that the Rank I LGC method fails in the parabola case $\g_2(t)=at +b t^{2}$, hence the necessary restriction $\a\not=2$. This is a fundamental, conceptual obstruction transcending technical issues which is closely related to the similar phenomenon that happens in the case of the Polynomial Carleson operator where using a linear wave-packets analysis only works as long as the  polynomial phase does not contain quadratic monomials in $t$. This latter aspect corresponds precisely to the situation when the operator under analysis lacks quadratic/higher order modulation invariance structure. 
\end{rem}

\begin{rem} In the same \cite{GL25} the authors therein provide a unified approach to both \eqref{nhilBHCarl} and \eqref{nhilb} in the case of hybrid curves with no quadratic (quasi-)resonances thus re-confirming the intimate connection between the two classes of operators under discussion.
\end{rem}

\begin{obs} 
The jump from $n=2$ to $n\geq 3$ poses key difficulties in all the curvature regimes and until very recently no results existed even in the (purely) non-zero curvature case. Also, from now on, the reference to the hybrid curvature situation is adjusted (relaxed) to refer to a generic situation that covers both zero and non-zero curvature features without necessarily including the full zero-curvature realm.
\end{obs}

\noindent\underline{\textbf{III.1) The case $n=3$, (purely) zero curvature.}} \quad This addresses the case when in \eqref{nhilBHCarl} we take $\underline{\g}_1(t)=-t$,  $\underline{\g}_2(t)=t$ and $\g(x,t)=a(x) t$ with $a(\cdot)$ measurable. In this situation we refer to $C_{2,\G}$ as the purely resonant Bilinear Hilbert-Carleson operator since it naturally extends both the classical Bilinear Hilbert transform and the Carleson operator (with both these in the zero-curvature setting).

The motivation for considering this operator is multifaceted: 1) it is directly connected with the problem of the pointwise convergence of Fourier series in the bilinear setting; 2) it is related with the bilinear version of Rubio de Francia operator associated with arbitrary strips; 3) it belongs to the class of (multi-linear) operators that are both linear and quadratic modulation invariant and for which---as mentioned earlier---with the notable exception of the Polynomial Carleson operator discussed in II.3), we lack a good understanding. For more details on all these items the reader is invited to consult the extended Introduction in \cite{BBLV24}.

For the class represented by \eqref{nhilb}, this case corresponds to the situation when $\g_j(t)= a_j t$ and $\{a_j\}_{j=1}^3\subset\R\setminus\{0\}$ pairwise distinct; this way, $H_{3,\vec{\g}}$ represents the celebrated classical Trilinear Hilbert transform whose $L^p$-boundedness remains a major open question in harmonic analysis. In terms of motivation, this problem has deep ramifications in an array of mathematical areas: 1) on the analytic number theory/additive combinatorics side, $H_{3,\vec{\g}}$ is connected with the theme of counting additive patterns in subsets of $\Z$ in the spirit of Roth's and Szemer\'edi's theorems (\cite{Sze69}, \cite{Sze75}) and later Gowers, \cite{Gow98}; other connections between the study of $H_{3,\vec{\g}}$ and additive combinatorics appears in \cite{Chr01}; 2) on the ergodic theory side, the study of $H_{3,\vec{\g}}$  relates to the behavior of suitable Furstenberg non-conventional averages (see \emph{e.g.} \cite{FW96}, \cite{HK05}, \cite{Zieg07}, \cite{Leib05a}) as well as to the pointwise convergence of bilinear Birkhoff averages for two commuting measure-preserving transformations (\cite{DKST19}); 3) on the time-frequency/classical harmonic analysis side, $H_{3,\vec{\g}}$ exhibits linear and quadratic modulation invariance and is part of a natural hierarchy of objects extending the bilinear Hilbert transform and also further related to the triangular Hilbert transform (\cite{KTZ}).

\begin{rem} 
The problems of studying the $L^p-$boundedness properties for both operators \eqref{nhilBHCarl} and  \eqref{nhilb}  in the zero curvature setting and $n=3$ are widely open with no positive results known to date. A fundamental difficulty in approaching either of these problems is the---yet---poor understanding of the higher order modulation invariance feature.
\end{rem}

\noindent\underline{\textbf{III.2) The case $n=3$, (purely) non-zero curvature.}} \quad For the class in \eqref{nhilBHCarl} this case corresponds to the generic situation $\underline{\g}_1(t)=a_1 t^{\a_1}$, $\underline{\g}_2(t)=a_2 t^{\a_2}$ and $\g(x,t)=a(x) t^{\a_3}$ with $\{a_j\}_{j=1}^3\subset\R\setminus\{0\}$ and $\{\a_j\}_{j=1}^3\subset\R\setminus\{0\}$ pairwise distinct. In this situation we refer to $C_{2,\G}$ as the purely non-resonant Bilinear Hilbert-Carleson operator since this object has no modulation invariance structure. 

Moving to class \eqref{nhilb}, this covers the situation $\underline{\g}_j(t)=a_j t^{\a_j}$, $1\leq j\leq 3$, with $H_{3,\vec{\g}}$ representing now the curved trilinear Hilbert transform.

The motivation for considering the purely curved situation for both \eqref{nhilBHCarl} and \eqref{nhilb} arises from multiple directions\footnote{For more on these we invite the interested reader to consult the introduction in \cite{HL23}.}: 
\begin{enumerate}
\item these problems can be considered as natural models for the very difficult but fundamental open problems corresponding to the zero curvature setting and discussed at point III.1; 
\item they have connections with number theory/additive combinatorics, specifically with the question of identifying qualitative and/or quantitative conditions on a subset $S$ of $\R$ (or $\Z$) that guarantee the existence of  polynomial progressions patterns; 
\item they have connections with ergodic theory area: in particular, the purely curved case of \eqref{nhilb} may be regarded as a singular Euclidean extension of the ergodic theoretical problem regarding the norm and pointwise convergence behavior of the Furstenberg non-conventional polynomial averages.
\end{enumerate}     
\begin{rem} 
The non-zero curvature setting is the only one that is now completely settled in the $n=3$ case. Indeed we have the following:
\begin{itemize}
\item the maximal boundedness range---up to end-points---for the purely non-resonant Bilinear Hilbert-Carleson operator  $C_{2,\G}$ is the subject of the present paper;

\item the $L^{p_1}(\R)\times L^{p_2}(\R)\times L^{p_3}(\R)$ into $L^{r}(\R)$ boundedness of $H_{3,\vec{\g}}$ within the Banach H\"older range $\frac{1}{p_1}+\frac{1}{p_2}+\frac{1}{p_3}=\frac{1}{r}$ with $1<p_1,p_3<\infty$, $1<p_2\leq \infty$ and $1\leq r <\infty$ was solved recently by the second and third authors in \cite{HL23}. The extension to the quasi-Banach range corresponding to $r>\frac{1}{2}$ is dealt with in the follow up paper \cite{HL24}. Whether or not the latter constitutes the maximal boundedness range of $H_{3,\vec{\g}}$ seems to be a subtle open problem with additive combinatorics flavor.
\end{itemize}
\end{rem}

\begin{rem}\label{lgcIImeth} The resolution of both $C_{2,\G}$ and $H_{3,\vec{\g}}$ in this non-zero curvature setting relies in a crucial manner on the newly developed Rank II LGC method introduced in \cite{HL23}. Indeed, as already noticed in Remark \ref{Gen1}, a key difficulty in approaching both these problems was the lack of absolute summability for the associated linearized discretized models (and hence the failure of Rank I LGC method).  
\end{rem}

\begin{rem}\label{ConIII2} Reinforcing the philosophy stated at the beginning of our section and shaped throughout our previous remarks, the present case offers another instance of the deep connections between the two classes of operators \eqref{nhilBHCarl} and \eqref{nhilb}, respectively. Indeed, a crucial ingredient in the treatment of $H_{3,\vec{\g}}$ is a good understanding---though only within a partial range of the spatial and frequency parameters\footnote{\emph{I.e.} $k$ and $m$ respectively - see Remark \ref{twoway} and, more generally, Section \ref{Planmainideas}.}---of the main component of the maximal operator $C_{2,\G}$. The latter is achieved via a novel procedure that we refer to as ``\emph{constancy propagation}" of the linearizing function $a(x)$ which relies on a bootstrap argument resembling an induction on scale reasoning and involving the structure of the phase $a(x)$ in the definition of $\g(x,t)$ as part of $C_{2,\G}$. Thus, in a nutshell, our approach of $H_{3,\vec{\g}}$ \emph{requires} in first place a good understanding of $C_{2,\G}$, with the latter becoming a natural companion to the former object. The role of the present paper is to complete this companionship initiated with the study of $H_{3,\vec{\g}}$ in \cite{HL23} and provide the full understanding of the boundedness properties of $C_{2,\G}$, in particular covering all the possible range for the spatial and frequency parameters.
\end{rem}

\begin{rem}\label{Pel} The methods developed in \cite{HL23} and here have a great potential for further extensions and/or generalizations. Indeed, one such natural direction of investigation regards the key results involving local smoothing inequalities proved in \cite{KMPW2024a} and \cite{KMPW2024b}. These results rely on the so-called PET induction scheme developed by Bergelson and Leibman in \cite{BL96}, a powerful method  which---in the context described above---however, faces two key restrictions: 1) it only applies to curves described by suitable polynomials (corresponding in our case to the restriction $\{\a_j\}_{j=1}^3\subset\N$ pairwise distinct), and, 2) it only applies to large physical scales (corresponding in our case, in either \eqref{nhilBHCarl} or \eqref{nhilb}, to enforcing $|t|>>1$).  However, the flexibility of our methodology overcomes both these obstacles supporting thus our initial claim.  
\end{rem}

\noindent\underline{\textbf{III.3) The case $n=3$, hybrid curvature.}} \quad  A typical situation addressing this case is $\underline{\g}_1(t)=a_1 t$, $\underline{\g}_2(t)=a_2 t$ and $\g(x,t)=a(x) t^{\a_3}$ or $\underline{\g}_3(t)=a_3 t^{\a_3}$ depending on whether we refer to \eqref{nhilBHCarl} or  \eqref{nhilb}, respectively. Of course, as usual. we require $\{a_j\}_{j=1}^3\subset\R\setminus\{0\}$ with $a_1,\,a_2$ pairwise distinct and $\a_3\in\R\setminus\{0,1,2\}$. In this situation both $C_{2,\G}$ and $H_{3,\vec{\g}}$ are linear modulation invariant while they also exhibit non-zero curvature features---hence the \emph{hybrid} curvature character. The condition $\a_3\not=2$ is imposed in order for these operators to have \emph{only} linear but \emph{no} quadratic modulation invariance features, with the latter instance being a next level, critical step in the major enterprise of tackling the purely zero curvature case described at item III.1).

The motivation for studying these problems is shared with the one described at  III.2). Moreover this case is a natural link between Cases III.2) and III.1): indeed, approaching Case III.3 requires the knowledge accumulated from Case III.2) and is required in order to have a better understanding of the difficulties arising when dealing with Case III.1).

\begin{rem} While not completely settled, we believe that soon we will be able to fully understand the hybrid no-quadratic resonance case. Indeed, we have the following:
\begin{itemize}
\item the hybrid Bilinear Hilbert Carleson operator $C_{2,\G}$ in the above setting was studied in \cite{BBLV24}, where the authors therein proved that  $C_{2,\G}$ is bounded from $L^{p_1}(\R)\times L^{p_2}(\R)$ into $L^{r}(\R)$ for $\frac{1}{p_1}+\frac{1}{p_2}=\frac{1}{r}$ with $1<p_1,p_3\leq \infty$ and $\frac{2}{3} < r <\infty$ thus in particular extending the work of M. Lacey and C. Thiele, \cite{lt1}, \cite{lt2}, on the classical (zero-curvature) Bilinear Hilbert transform;

\item the hybrid trilinear Hilbert transform $H_{3,\vec{\g}}$ in a closely related setting with the one above is the subject of an ongoing study in \cite{BBL25}. 
\end{itemize}
Both these works are crucially relying on Rank II LGC (in particular their discretized linear wave-packet models fail to be absolutely summable) at which---in contrast with case III.2)---one now adds modulation invariance/multiresolution analysis techniques. 
\end{rem}

\begin{rem}\label{ConIII3}  Mirroring Remark \ref{ConIII2}, the hybrid case reconfirms once again the intimate link between \eqref{nhilBHCarl} and \eqref{nhilb}: the analysis of $H_{3,\vec{\g}}$ encapsulates as a prerequisite a good understanding of the maximal operator $C_{2,\G}$ that is further conditioned by the elements within the itemized list in Remark \ref{Gen1} and in particular by the local structure of the linearizing phase function $a(x)$ in the definition of $C_{2,\G}$.
\end{rem}

\noindent\underline{\textbf{IV) The general case $n\geq 4$.}} \quad Of course, as expected, this general case presents itself with multiple challenges that are yet to be understood. As before, one can create a natural hierarchy in an increasing order of difficulty from non-zero, hybrid to zero curvature. While the latter two cases are---in full generality---beyond reach as of now, for the former case substantial progress has been made. Indeed, the purely curved case addressing  $H_{n,\vec{\g}}$ in the general setting $\g_j(t)=a_j t^{\a_j}$ with  $\{a_j\}_{j=1}^n\subset\R\setminus\{0\}$ and $\{\a_j\}_{j=1}^n\subset\R\setminus\{0\}$ pairwise distinct as well as $C_{n-1,\vec{\G}}$ with the obvious modifications is the subject of an ongoing investigation in \cite{HL25} revealing yet again the intimate connection between \eqref{nhilb} and  \eqref{nhilBHCarl}. Moreover, the Hardy-Littlewood maximal analogue of \eqref{nhilb} has already been settled in the purely curved case and general $n$ in \cite{HL24},

\subsection{Plan of the proof; main ideas}\label{Planmainideas}

In this section we provide
\begin{itemize}
\item (A) an outline of the structure of our paper;

\item (B) a brief description of some of the most relevant ideas.
\end{itemize}

We start our discussion with item (A): first the versatility\footnote{See also Remark \ref{Pel}.} of our methods allows us to treat with minimal effort any nonresonant choice of $\vec{\a}$ (see the hypothesis in Main Theorem \ref{mainresult}) once we have designed the strategy for a particular, representative choice. For this reason, it is enough to only focus on a special case of \eqref{BHCcurval1}, denoted by  
\begin{equation}\label{bhcdef}
BHC(f_1, f_2)(x):=\sup_{\lambda \in \R} \left|\, \textrm{p.v.} \int_\R f_1(x-t)f_2(x+t^2)e^{i\lambda t^3} \frac{dt}{t} \right|\,.
\end{equation}
In Section \ref{PrelDec}, we perform a preliminary discretization of $BHC$ depending on two type-parameters:
\begin{itemize}
\item a spatial parameter $k\in\Z$ accounting for the spatial scale of the kernel/multiplier;

\item a frequency parameter triple $(j,l,m)\in\Z^3$ accounting for the magnitude of the Fourier supports of the input functions $f_1$, $f_2$ and of the linearizing phase $\lambda(x)$, respectively. 
\end{itemize}
Next, depending on the suitable interdependencies between $k,j,l,m$ which govern the behaviour of the phase of the multiplier, and, in particular, the presence/absence of stationary points, we subdivide our operator into three key components:
$$BHC\leq BHC^{Lo}+BHC^{\not \Delta}+BHC^{\Delta}\,,$$
with the middle component further subdivided as
$$BHC^{\not \Delta}\leq BHC^{\not \Delta, S}+BHC^{\not \Delta, NS}\,,$$
where here $BHC^{Lo}$ stands for the low oscillatory component, the high oscillatory component corresponds to an on-diagonal ($BHC^{\Delta}$) and an
off-diagonal ($BHC^{\not \Delta}$) component, and finally, $BHC^{\not \Delta}$ encompasses a high frequency off-diagonal \emph{stationary} component $BHC^{\not \Delta, S}$ and a high frequency off-diagonal \emph{non-stationary} component $BHC^{\not \Delta, NS}$.

Withe these, our approach develops as follows\footnote{For convenience, all the analysis below---except for the one in the Appendix---is performed in the regime $k\geq0$.}:
\begin{itemize}
\item $BHC^{\not \Delta, NS}$ and $BHC^{Lo}$ are relatively easy to control, task which can be accomplished independent of the behavior of the other components: the former is treated in Section \ref{20250228sec01}---see Proposition \ref{20250317prop01}, while the latter is handled in Section \ref{20250228sec02}---see Theorem \ref{thmin};
\item as expected, the main difficulties and ideas reside in the local $L^2$ smoothing estimate for the diagonal component $BHC^{\Delta}$---see Section \ref{Mainsect} and in particular Theorem \ref{20241018mainthm01} therein; 
\item the local $L^2$ resolution of the high frequency off-diagonal \emph{stationary} component $BHC^{\not \Delta, S}$ relies on a skillful adaptation of the ideas employed for treating the main term $BHC^{\Delta}$; the required argumentation is presented in Section \ref{20250308sec01}---see Theorem \ref{generalizedmain};
\item finally, the maximal quasi-Banach boundednees range (up to end-points) of $BHC$ is completed in Section \ref{20250309sec01} while in the Appendix we provide a brief sketch of the modifications required in order to deal with the situation when $k<0$.         
\end{itemize}

We now move our focus on item (B): As expected, the crux of our entire work resides on obtaining the local $L^2$ smoothing control on the diagonal component $BHC^{\Delta}$ with the latter representing (essentially) the regime $j=l=m\in\N$ which corresponds to the main situation when the phase of the multiplier admits stationary points. The main result captured within Theorem \ref{20241018mainthm01} can be essentially included within the following:
$\newline$

\noindent\textbf{Main Proposition (informal).} [\textsf{$L^2$-smoothing control on the maximal joint Fourier coefficient}]\quad \emph{Let $r\in\Z$, $k,\,m\in\N$ and set 
$I_r^k=[r 2^{-k}, (r+1) 2^{-k}]$ and $I^k\equiv I_1^k$. Assume $f_1,\,f_2\in L^2(\R)$ such that $\textrm{supp}\, \widehat{f}_1\subseteq [2^{m+k}, 2^{m+k+1}]$ and  $\textrm{supp}\, \widehat{f}_2\subseteq [2^{m+2k}, 2^{m+2k+1}]$. Define now the $(m,k)$--\emph{maximal joint Fourier coefficient} of the pair $(f_1,f_2)$ along the moment curve $(t,t^2,t^3)$ as 
\begin{equation} \label{maxjointFourcoeff}
 \mathcal{J}_{m, k}(f_1, f_2)(x):=\sup_{\lambda\in\R}\,\frac{1}{|I^k|^{\frac{1}{2}}}\,\left|\int_{I^k} f_1(x-t)\,f_2(x+t^2)\,e^{i \lambda t^3} dt \right|\,.
\end{equation} 
Then, there exists $\ep>0$ absolute constant such that for any $r\in\Z$, one has
\begin{equation} \label{maxjointFourcoeffcontr}
\left\|\mathcal{J}_{m, k}(f_1, f_2)(\cdot)\right\|_{L^2(I_r^k)}\lesssim   2^{-\epsilon \min \{2 k, m\}}\,\|f_1\|_{L^2(3I_r^k)}\, \|f_2\|_{L^2(3 I_r^k)}\,.
\end{equation}}

\smallskip
We next elaborate on some of the main ideas involved in the proof of the above statement. The first simple observation based on some standard reasonings is that we only need to show \eqref{maxjointFourcoeffcontr} for $r=1$ and hence we may assume from now on that $x\in I^k$. Next, by exploiting the hypothesis in the Main Proposition, a linearization procedure reduce matters to the following situation: 
\smallskip

\begin{equation} \label{maxjointFourcoeff1}
 \mathcal{J}_{m, k}(f_1, f_2)(x)\approx \frac{1}{|I^k|^{\frac{1}{2}}}\,\left|\int_{I^k} f_1(x-t)\,f_2(x+t^2)\,e^{i 2^{\frac{m}{2}+3k} \widetilde{\lambda}(x) t^3} dt \right|\,,
\end{equation} 
with 
\begin{equation} \label{constlam}
\widetilde{\lambda}: I^k \to \left[2^{\frac{m}{2}}, 2^{\frac{m}{2}+1} \right]\quad\textsf{measurable, constant on intervals of length}\:\:2^{-m-2k}\,.
\end{equation} 
\smallskip

Building on the work in \cite{HL23}, the key contribution  of this paper is to unravel the subtle relationship between the \emph{structure} of the linearizing phase function $\widetilde{\lambda}$ and the \emph{nature} of the method employed in the resolution of \eqref{maxjointFourcoeffcontr}. More precisely, we have the following philosophical dichotomy:
\begin{itemize}
\item on the one hand, if $\widetilde{\lambda}$ is ``slowly" varying, \emph{e.g.} $\widetilde{\lambda}$ is constant on $ I^k$, then one can apply the Rank I LGC method having as a result a wave-packet discretizated model of \eqref{maxjointFourcoeff1} which can be controlled via absolute values in order to obtain the desired estimate in \eqref{maxjointFourcoeffcontr}.

\item on the other hand, if  $\widetilde{\lambda}$ is ``widely" varying within the context of \eqref{constlam}, then Rank I LGC method fails\footnote{\emph{I.e.}, the resulting discretized wave-packet model does not produce the required $ 2^{-\epsilon \min \{2 k, m\}}$-decay if the absolute values are applied over the local Fourier coefficients corresponding to $f_1$ and $f_2$.} and one has to appeal to the \emph{correlative} nature of the Rank II LGC method.
\end{itemize}

The main ingredient which reveals and allows us to exploit the above dichotomy is the \emph{constancy propagation} procedure introduced in \cite{HL23} and further refined here. As the name suggests, this procedure resembling an induction on scale mechanism allows us to move from a smaller constancy scale to a  larger one by exploiting different cancellation phenomena within the structure of our operator.

With these being said, we are ready to provide more content to our story: the constancy propagation procedure can not be performed \emph{pointwise}, directly on the expression \eqref{maxjointFourcoeff1}, but rather on a majorant of its $L^2$ $x$-average; more precisely, employing a spatial sparse-uniform dichotomy corresponding to the first stage of item (1) in Remark \ref{Gen1}, one reduces matters to the situation when both $f_1$ and $f_2$ are \emph{spatially} uniformly distributed. Once here, we apply a partition of the $t$-integration domain at the level of the \emph{linearizing} scale, \emph{i.e.} we write $I^k=\bigcup_{q\sim 2^{\frac{m}{2}}}I_q^{m, k} $  where here  $I_q^{m, k}:=\left[ \frac{q}{2^{\frac{m}{2}+k}}, \frac{q+1}{2^{\frac{m}{2}+k}}   \right]$ noticing that for $t\in I_q^{m, k}$ we have $2^{\frac{m}{2}+3k}\,\widetilde{\lambda}(x)\,t^3= -2 \cdot \frac{q^3 \lambda(x)}{2^{\frac{3m}{2}+3k}}\,+\,3 \cdot \frac{q^2 \lambda(x)}{2^{m+2k}}\,t\,+\,O(1)$, and, after some standard reasonings, we obtain 
\begin{eqnarray} \label{dom}
&& \:\:\:\:\:\: \qquad\qquad\qquad\qquad\qquad\left\|\mathcal{J}_{m, k}(f_1, f_2)(\cdot)\right\|_{L^2(I^k)}^2 \lesssim \nonumber  \\
&&\left[{\calL}_{m, k}(f_1, f_2) \right]^2 \ :=  2^{\frac{m}{2}+k} \cdot  \int_{I^k} \left( \sum_{q \sim 2^{\frac{m}{2}}} \left| \int_{I_q^{m, k}} f_1(x-t)f_2(x+t^2) e^{3i2^k \cdot \frac{q^2 \widetilde{\lambda}(x)}{2^{\frac{m}{2}}}t} dt\right|^2 \right)dx.
\end{eqnarray} 

It is on this latter expression that our constancy propagation applies.  Specifically, our analysis splits in two major cases:

\medskip 

\noindent \underline{\textsf{Case I: $k \ge \frac{m}{2}$.}} \quad In this case, the corner stone is represented by the proof of the following:
\smallskip

\begin{prop} (\cite[Theorem 4.3]{HL23}, \textsf{Informal statement})\label{maxjointctrl} If $k\geq \frac{m}{2}$, $f_1$, $f_2$  spatially uniformly distributed, and 
\begin{equation} \label{constlam1}
\widetilde{\lambda}: I^k \to \left[2^{\frac{m}{2}}, 2^{\frac{m}{2}+1} \right]\quad\textsf{measurable, constant on intervals of length}\:\:2^{-3k}\,,
\end{equation} 
then there exists $\ep>0$ absolute constant such that the expression ${\calL}_{m, k}(f_1, f_2)$ defined in \eqref{dom} obeys
\begin{equation} \label{maxjointFourcoeffcontr1}
{\calL}_{m, k}(f_1, f_2) \lesssim   2^{-\epsilon m}\,\|f_1\|_{L^2(3I^k)}\, \|f_2\|_{L^2(3 I^k)}\,.
\end{equation}
\end{prop}
\smallskip

\begin{rem}\label{twoway}  Proposition \ref{maxjointctrl} is a concrete embodiment of the deep connection between \eqref{BHCcurval} and \eqref{ThTcurval} claimed at the beginning of our Introduction. Indeed
\begin{itemize}
\item on the one hand,  Proposition \ref{maxjointctrl}---in the form of Theorem 4.3---is a key ingredient in \cite{HL23} for obtaining the $L^p$-boundedness of \eqref{ThTcurval1}; 

\item on the other hand, as revealed by \eqref{maxjointFourcoeffcontr}, \eqref{dom} and \eqref{maxjointFourcoeffcontr1},  Proposition \ref{maxjointctrl} captures exactly the desired smoothing estimate for the main component of \eqref{BHCcurval1} in the regime $k\geq m$.  
\end{itemize}
\end{rem}

Once here, our analysis further splits in two regimes:
\begin{enumerate}
    \item [$\bullet$]  \underline{\textsf{Case I.1: $k \ge m$.}} \quad In this situation, we simply notice that $2^{-m-2k}\geq 2^{-3k}$ and hence \eqref{constlam} implies \eqref{constlam1}. Consequently our Main Proposition is a direct corollary of Proposition \ref{maxjointctrl}.
\smallskip        
    
    \item [$\bullet$]  \underline{\textsf{Case I.2: $\frac{m}{2} \le k \le m$.}}\quad  In this situation, we follow a two-step approach: 
\medskip    
 \begin{itemize}   
    \item \underline{\textsf{Step I.2.1:}} We first implement the constancy propagation procedure in order to extend the constancy of $\widetilde{\lambda}$ from intervals of length $2^{-m-2k}$ to intervals of length $2^{-m-k}$ (see Section \ref{20241028subsec01}). 
\smallskip    
    \item \underline{\textsf{Step I.2.2:}} Once we reach the scale $2^{-m-k}$, we notice that in our regime $2^{-m-k} \ge 2^{-3k}$, and hence,  Proposition \ref{maxjointctrl} applies once again.
 \end{itemize}   
\end{enumerate}

\medskip

\noindent \underline{\textsf{Case II: $0 \le k \le \frac{m}{2}$.}} In this case there is no stratagem that allows us to reduce matters to  Proposition \ref{maxjointctrl}; in particular, we notice that in the present setting, in contrast with \eqref{maxjointFourcoeffcontr1},  the decaying factor in \eqref{maxjointFourcoeffcontr} takes the form $2^{-\epsilon k}$. Thus, our strategy develops now as follows:
\smallskip
\begin{itemize}
\item \underline{\textsf{Step II.1:}} As at the Step I.2.1, requiring though a slightly different point of view, we apply the constancy propagation procedure in order to extend the constancy of $\widetilde{\lambda}$ from intervals of length $2^{-m-2k}$ to intervals of length $2^{-m-k}$ (see Proposition \ref{20241018prop01}, Section \ref{20241217subsec01}).
\smallskip

\item \underline{\textsf{Step II.2:}} Once we reach the $2^{-m-k}-$constancy scale for  $\widetilde{\lambda}$, we return to the original form \eqref{maxjointFourcoeff1} and apply to it the Rank I LGC method in order to obtain \eqref{maxjointFourcoeffcontr} (see Section \ref{20250228sec03}). 
\end{itemize}
\smallskip

\noindent We end this section with the following 

\begin{rem}\label{twoway1} While the constancy propagation procedure is a key ingredient in the resolution of  
\begin{itemize}
\item \emph{i)} on the one hand, Proposition \ref{maxjointctrl}, and, 

\item \emph{ii)} on the other hand, Step I.2.1 and Step II.1, 
\end{itemize}
there are some relevant distinctions in the way in which this procedure is implemented: Indeed, in the former situation the analysis is performed at the level of each $x-$interval having the $x-$relevant scale $2^{-2k}$ (see Proposition 4.7 in \cite{HL23}) while in the latter situation a successful analysis requires a grouping/concatenation of $2^{m-k}$ $x-$intervals having the $x-$relevant scale $2^{-m-k}$ (see Sections \ref{20250228sec03} and \ref{20241028subsec01}).  
\end{rem}

\subsection*{Acknowledgments}
The first author was supported by an AMS-Simons Research Enhancement Grant for PUI Faculty. The second author was supported by the Simons Travel grant MPS-TSM-00007213. The third author was supported by the NSF grant DMS-2400238 and by the Simons Travel grant MPS-TSM-00008072.

\section{Set-up for the problem: preliminaries}\label{PrelDec}

For notational simplicity and clarity of the exposition, throughout the reminder of the paper we set $\vec{a}=(1,-1,1)$, $\vec{\a}=(1,2,3)$ and focus only on the corresponding $BHC_{[\vec{a},\vec{\a}]}$ which for convenience will be referred from now on as $BHC$. With these, the goal of our paper is to analyze the boundedness behavior of the \emph{(purely) non-resonant bilinear Hilbert Carleson}:
$$
BHC(f_1, f_2)(x):=\sup_{\lambda \in \R} \left|\, \textrm{p.v.} \int_\R f_1(x-t)f_2(x+t^2)e^{i\lambda t^3} \frac{dt}{t} \right|, 
$$
whose multiplier form is given by
$$
BHC(f_1, f_2)(x)= \sup_{\lambda \in \R} \left| \int_{\R} \widehat{f_1}(\xi) \widehat{f_2}(\eta) \frakm_\lambda(\xi, \eta) e^{i\xi x} e^{i \eta x} d\xi d\eta \right|,
$$
with
$$
\frakm_{\lambda}(\xi, \eta):=\, \textrm{p.v.} \int_{\R} e^{-i\xi t} e^{i \eta t^2} e^{i \lambda t^3} \frac{dt}{t}. 
$$
\subsection{Discretization of the multiplier}
As usual in such circumstances, we proceed with a two-layer decomposition, as follows:

\subsubsection{A physical discretization} Let $\rho \in C_0^\infty(\R)$ be an odd function with $\textrm{supp}\,\rho \subseteq \left[-4, 4 \right] \backslash \left[-\frac{1}{4}, \frac{1}{4} \right]$ which, for $\rho_k(t):=2^k \rho(2^k t)$, satisfies 
$$
\frac{1}{t}=\sum_{k \in \Z} \rho_k(t)\,\qquad\textrm{for all}\:\:t\in\R\setminus\{0\}\,.
$$
This yields a first decomposition 
$$
BHC(f_1, f_2)(x)=\sup_{\lambda \in \R} \left| \sum_{k \in \Z} BHC_{\lambda}^k(f_1, f_2)(x) \right|, 
$$
where for each $\lambda \in \R$, 
$$
BHC_{\lambda}^k(f_1, f_2)(x):=\int_{\R^2} \widehat{f_1}(\xi) \widehat{f_2}(\eta) \frakm^k_{\lambda}(\xi, \eta) e^{i\xi x} e^{i \eta x} d\xi d\eta, 
$$
with
$$
\frakm_{\lambda}^k(\xi, \eta):=\int_{\R} e^{-i \frac{\xi}{2^k} t} e^{i \frac{\eta}{2^{2k}} t^2} e^{i \frac{\lambda}{2^{3k}}t^3} \rho(t)dt.  
$$

\subsubsection{A frequency discretization} 
Since the derivative of the phase function plays a fundamental role in the behavior of such integral operators, it is natural to start this section with a basic \emph{phase analysis}:
\begin{itemize}
\item for $k\in\Z$, we denote by
$$
\varphi_{\xi, \eta, \lambda}^k(t):=-\frac{\xi}{2^k}t+\frac{\eta}{2^{2k}}t^2+\frac{\lambda}{2^{3k}}t^3
$$ 
and notice that
\begin{equation} \label{20240407eq01}
\frac{d}{dt} \varphi_{\xi, \eta, \lambda}^k(t)=-\frac{\xi}{2^k}+\frac{2\eta}{2^{2k}}t+\frac{3\lambda}{2^{3k}}t^2, 
\end{equation} 
\begin{equation} \label{20240407eq02}
\frac{d^2}{dt^2} \varphi_{\xi, \eta, \lambda}^k(t)=\frac{2\eta}{2^{2k}}+\frac{6\lambda}{2^{3k}}t,
\end{equation} 
and
\begin{equation} \label{20240407eq03}
\frac{d^3}{dt^3} \varphi_{\xi, \eta, \lambda}^k(t)=\frac{6\lambda}{2^{3k}}\,;
\end{equation} 

\item next, setting $t_0$ the unique root of $\frac{d^2}{dt^2} \varphi_{\xi, \eta, \lambda}^k$, we have 
\begin{equation} \label{20240407eq05}
t_0:=-\frac{2^k \eta}{3 \lambda} \qquad \textrm{and} \qquad \frac{d^2}{dt^2} \varphi_{\xi, \eta, \lambda}^k (t_0)=0;
\end{equation} 

\item denting the discriminant associated with \eqref{20240407eq01} by 
\begin{equation} \label{20240407eq04}
\Delta:=\frac{\eta^2+3\lambda\xi}{2^{4k}}\,,
\end{equation} 
and using \eqref{20240407eq04} and \eqref{20240407eq05}, one can rewrite $\frac{d}{dt} \varphi_{\xi, \eta, \lambda}^k$ as
$$
\frac{d}{dt} \varphi_{\xi, \eta, \lambda}^k(t)=\frac{3 \lambda}{2^{3k}} \left[(t-t_0)^2-\left(\frac{2^{3k}}{3\lambda} \right)^2 \Delta \right], 
$$
and hence
$$
\frac{d}{dt} \varphi_{\xi, \eta, \lambda}^k(t_0)=-\frac{2^{3k}}{3\lambda} \Delta\,;
$$ 

\item finally, when $\Delta \ge 0$ we denote by $t_1, t_2$ the roots of the equations $\frac{d}{dt} \varphi_{\xi, \eta, \lambda}^k(t)=0$, and notice that
$$
t_{1, 2}=-\frac{2^k \eta}{3\lambda} \pm \frac{2^{3k} \sqrt{\Delta}}{3 \lambda}
$$
with
$$
\left|t_1-t_0 \right|= \left|t_2-t_0 \right|=\frac{2^{3k} \sqrt{\Delta}}{3 \left| \lambda \right|}. 
$$
\end{itemize}

With these done, having in mind the size of the $t-$coefficients in \eqref{20240407eq01}, we let $\phi_1$ and $\phi_2$ be some even and positive functions with $\supp \ \phi_{1, 2} \subseteq \left [-4, 4 \right] \backslash  \left[-\frac{1}{4}, \frac{1}{4} \right]$, which satisfy the partition of unity 
$$
1=\sum_{j, l \in \Z} \phi_1 \left(\frac{\xi}{2^{j+k}} \right) \phi_1 \left(\frac{\eta}{2^{l+2k}} \right) \qquad \textrm{and} \qquad 1=\sum_{m \in \Z} \phi_2 \left(\frac{\lambda}{2^{m+3k}} \right),
$$
respectively. Finally, for every $j, k, l, m \in \Z$, we let
$$
\frakm_{j, l, m, \lambda}^k \left(\xi, \eta \right):=\left(\int_{\R} e^{i\varphi_{\xi, \eta, \lambda}^k(t)} \rho(t) dt\right)  \phi_1 \left(\frac{\xi}{2^{j+k}} \right) \phi_1 \left(\frac{\eta}{2^{l+2k}} \right) \phi_2 \left(\frac{\lambda}{2^{m+3k}} \right),
$$
thus obtaining the decomposition 
$$
BHC(f_1, f_2)(x)=\sup_{\lambda \in \R} \left| \sum_{k \in \Z} \sum_{j, l, m \in \Z} BHC_{j, l, m, \lambda}^k(f_1, f_2)(x) \right|\,,
$$
where $BC_{j, l, m, \lambda}^k$ is the operator induced by the multiplier $m_{j, l, m, \lambda}^k$.

\subsection{Discretization of our $BHC$ operator and basic reductions}

In this section we provide:

\begin{enumerate}
\item  The road-map for the discretization of our operator $BHC$ into three main components:
\begin{itemize}
\item a \emph{low oscillatory} component denoted by $BHC^{Lo}$;

\item a \emph{high frequency diagonal} component denoted by $BHC^{\Delta}$;

\item a \emph{high frequency off-diagonal} component denoted by $BHC^{\not \Delta}$ and further subdivided into
\begin{itemize}
\item a high frequency off-diagonal \emph{stationary} component $BHC^{\not \Delta, S}$;

\item a high frequency off-diagonal \emph{non-stationary} component $BHC^{\not \Delta, NS}$.
\end{itemize}
\end{itemize}

\item A simple and concise treatment of the latter $BHC^{\not \Delta, NS}$ component.
\end{enumerate}

\subsubsection{A preliminary discretization}
We start our decomposition as follows:
\begin{eqnarray*}
BHC(f_1, f_2)%
&\le & BHC^{Hi}(f_1, f_2)\,+\,BHC^{Lo}(f_1, f_2)  \\
& \le &  BHC_{+}^{Hi}(f_1, f_2)\,+\,BHC_{-}^{Hi}(f_1, f_2)\,+\, BHC^{Lo}(f_1, f_2)\,, 
\end{eqnarray*}
where
\begin{enumerate}
    \item [$\bullet$] the \emph{low oscillatory} component is defined by
    $$
    BHC^{Lo}(f_1, f_2)(x):=\sup_{\lambda \in \R} \left|BHC_{ \lambda}^{Lo}(f_1, f_2)(x)\right|\qquad\textrm{with}\qquad BHC_{ \lambda}^{Lo}(f_1, f_2):=\sum_{k \in \Z} \sum_{(j, l, m) \in \Z^3 \backslash \N^3}BHC_{j, l, m, \lambda}^k(f_1, f_2);
    $$
    \item [$\bullet$] the \emph{high oscillatory} component is defined by 
    $$
    BHC^{Hi}(f_1, f_2)(x):=\sup_{\lambda \in \R} \left|BHC_{\lambda}^{Hi}(f_1, f_2)(x)\right|\qquad\textrm{with}\qquad BHC_{\lambda}^{Hi}(f_1, f_2):=\sum_{k \in \Z} \sum_{(j, l, m) \in \N^3} BHC^k_{j, l, m, \lambda}(f_1, f_2), 
    $$
    and 
    $$
       BHC_{\pm}^{Hi}(f_1, f_2)(x):=\sup_{\lambda \in \R} \left|BHC_{\pm,\lambda}^{Hi}(f_1, f_2)(x)\right|\qquad\textrm{with}\qquad BHC_{\pm,\lambda}^{Hi}(f_1, f_2):=\sum_{k \in \Z_{\pm}} \sum_{(j, l, m) \in \N^3} BHC^k_{j, l, m, \lambda}(f_1, f_2).
    $$
\end{enumerate}

\subsubsection{Decomposition of the high oscillatory term into a diagonal and an off-diagonal component}
We further divide our analysis of the high oscillatory component into two different cases:
\begin{equation}\label{hifreq}
BHC^{Hi}(f_1, f_2) \le BHC^{\Delta}(f_1, f_2)\,+\,BHC^{\not \Delta}(f_1, f_2) \qquad \textrm{with}
\end{equation}
\begin{enumerate}
    \item [$\bullet$] $BHC^{\Delta}(f_1, f_2)(x):=\sup_{\lambda \in \R} \left|BHC_\lambda^{\Delta}(f_1, f_2)(x) \right|$ the \emph{diagonal} component having the frequency index range
    $$
    \max\{j, l, m\}-\min\{j, l, m\} \le 300; 
    $$
    \item [$\bullet$] $BHC^{\not \Delta}(f_1, f_2)(x):=\sup_{\lambda \in \R} \left| BHC_{\lambda}^{\not \Delta}(f_1, f_2)(x) \right|$ the \emph{diagonal} component having the frequency index range
    $$
    \max\{j, l, m\}-\min\{j, l, m\}> 300. 
    $$
\end{enumerate}
One can also define $BHC_{\pm}^{\Delta}(f_1, f_2)$ and $BHC_{\pm}^{\not \Delta}(f_1, f_2)$ in an obvious way. Moreover, with the obvious correspondences, we further subdivide the off-diagonal term as 
\begin{equation}\label{offdiag}
BHC^{\not \Delta}(f_1, f_2)\le BHC^{\not \Delta, S}(f_1, f_2)\,+\,BHC^{\not \Delta, NS}(f_1, f_2) \qquad \textrm{where}
\end{equation}
\begin{enumerate}
    \item [$\bullet$] $BHC^{\not \Delta, S}(f_1, f_2)$ stands for the \emph{stationary} off-diagonal component corresponding to the stationary phase regime, \textit{i.e.} when at least one of the roots $\{t_1, t_2\}$ has the absolute value of size $\sim 1$, or equivalently, one of the following situations is satisfied\footnote{For any $j, l \in \Z$, $j \cong l$ means $|j-l|<20$.}
    \begin{enumerate}
        \item [(a)] $j \cong l$ and $j>m+100$;
        \item [(b)] $l \cong m$ and $l>j+100$;
        \item [(c)] $m \cong j$ and $m>l+100$. 
    \end{enumerate}
    \item [$\bullet$] $BHC^{\not \Delta,NS}(f_1, f_2)$ stands for the \emph{non-stationary} off-diagonal component corresponding to the non-stationary phase regime, \textit{i.e.}, when the first item above is not satisfied. 
\end{enumerate}

\vspace{-0.1in}

\subsection{Treatment of the non-stationary off-diagonal component $BHC^{\not \Delta,NS}$}  \label{20250228sec01}

This part is routine and follows a similar reasoning with the one in \cite[Section 2.4]{HL23}. To begin with, we let $\lambda: \R \to \R$ be a measurable function such that 
 $$
BHC^{\not \Delta, NS}(f_1, f_2)(x):= \sup\limits_{\lambda \in \R} \left| BHC^{\not \Delta, NS}_{\lambda} (f_1, f_2)(x)\right| \le 2 \left| BHC^{\not \Delta, NS}_{\lambda(x)} (f_1, f_2)(x)\right|. 
 $$
We then make use of the following two results.

\begin{prop}{\cite[Lemma 2.2]{HL23}} \label{20241113prop01}
For any $(j, l, m) \in \N^3$ within the non-stationary off-diagonal regime defined above, one has 
$$
\left| \frac{d}{dt} \varphi_{\xi, \eta, \lambda(x)}^k(t) \right| \gtrsim 2^j+2^l+2^m\,\qquad \textrm{for all}\:\:\:|t| \sim 1\,.
$$
\end{prop}

\begin{prop} \label{20250317prop01}
Let $j, l, m$ and $\lambda(\cdot)$ be defined as above. Then the following estimate holds:
$$
BHC^{\not \Delta, NS}(f_1, f_2)(x) \lesssim  \sum_{j, l \ge 0} \frac{1}{2^{\max\{j, l\}}} \cdot \int_{\frac{1}{4} \le |t| \le 4} \calS_{2^j t}^{\phi_1}(x) \calS_{-2^l t^2}^{\phi_1}(x) dt,
$$
where here $\calS_t^\phi f(x):= \left( \sum\limits_{k \in \Z} \left| \left(f* \left(2^k \check{\phi} \left(2^k \cdot \right) \right) \right) \left(x-\frac{t}{2^k} \right) \right|^2 \right)^{\frac{1}{2}}$ is the \emph{$t$-shifted square function}.

Consequently, we have that 
\begin{equation}\label{nonstathifoffdiag}
        \|BHC_{\lambda(\cdot)}^{\not \Delta, NS}(f_1, f_2)(\cdot)\|_{L^r}\lesssim_{r,p_1,p_2}\|f_1\|_{L^{p_1}}\,\|f_2\|_{L^{p_2}}
\end{equation}
for any H\"older range $\frac{1}{p_1}+\frac{1}{p_2}=\frac{1}{r}$ with $1<p_1,p_2\leq \infty$ and $\frac{1}{2}<r<\infty$.

\end{prop}

\begin{proof}
Fix $k \in \Z$ and $(j, l, m) \in \N^3$ in the non-stationary off-diagonal regime. We then see that using integration by parts, we have
$$
\frakm_{j, l, m, \lambda(x)}^k(\xi, \eta)=\frakA_{j, l, m, \lambda(x)}^k(\xi, \eta)+\frakB_{j, l, m, \lambda(x)}^k(\xi, \eta),
$$
where
$$
\frakA_{j, l, m, \lambda(x)}^k(\xi, \eta):=i \left( \int_{\R} e^{i \varphi_{\xi, \eta, \lambda(x)}^k(t)} \frac{\rho'(t)}{\frac{d}{dt} \varphi_{\xi, \eta, \lambda(x)}^k(t)}dt \right) \phi_1 \left(\frac{\xi}{2^{j+k}} \right) \phi_1 \left(\frac{\eta}{2^{l+2k}}\right) \phi_2 \left(\frac{\lambda(x)}{2^{{m+3k}}} \right)
$$
and
$$
\frakB_{j, l, m, \lambda(x)}^k(\xi, \eta):=-i \left( \int_{\R} e^{i \varphi_{\xi, \eta, \lambda(x)}^k(t)} \frac{\rho(t) \frac{d^2}{dt^2} \varphi_{\xi, \eta, \lambda(x)}^k(t)}{\left(\frac{d}{dt} \varphi_{\xi, \eta, \lambda(x)}^k(t) \right)^2}dt \right) \phi_1 \left(\frac{\xi}{2^{j+k}} \right) \phi_1 \left(\frac{\eta}{2^{l+2k}}\right) \phi_2 \left(\frac{\lambda(x)}{2^{{m+3k}}} \right)
$$
This further allows us to write. 
$$
BHC_{j, l, m, \lambda(x)}^k(f_1, f_2)=BHC_{j, l, m, \lambda(x)}^{k, \frakA}(f_1, f_2)+BHC_{j, l, m, \lambda(x)}^{k, \frakB}(f_1, f_2), 
$$
where $BHC_{j, l, m,\lambda(x)}^{k, \frakA}(f_1, f_2)$ and $ BHC_{j, l, m, \lambda(x)}^{k, \frakB}(f_1, f_2)$ have the multipliers $\frakA_{j, l, m, \lambda(x)}^k(\xi, \eta)$ and $\frakB_{j, l, m, \lambda(x)}^k(\xi, \eta)$, respectively. 

Denote $n:=\max\{j, l, m\}$. Observe that by Proposition \ref{20241113prop01}, for $|t| \sim 1$, 
$$
\left| \frac{d^2}{dt^2} \varphi_{\xi, \eta, \lambda(x)}^k(t) \right| =\left|\frac{6 \lambda(x)}{2^{3k}}\right| \cdot \left|t-t_0 \right| \lesssim 2^m \cdot \left\{2^{l-m}, 1\right\} \le 2^n \lesssim \left| \frac{d}{dt} \varphi_{\xi, \eta, \lambda(x)}^k(t) \right|, 
$$
which further gives
$$
\left| \frac{\rho'(t)}{\frac{d}{dt} \varphi_{\xi, \eta, \lambda(x)}^k(t)} \right|, \quad \left|  \frac{\rho(t) \frac{d^2}{dt^2} \varphi_{\xi, \eta, \lambda(x)}^k(t)}{\left(\frac{d}{dt} \varphi_{\xi, \eta, \lambda(x)}^k(t) \right)^2}\right| \lesssim \frac{1}{2^n}.
$$
Therefore, it suffices for us to analyze only one of the terms above, say $BHC_{j, l, m, \lambda(\cdot)}^{k, \frakA}(f_1, f_2)$. Next, via a Taylor series argument, we may consider, for the sake of analysis, that
\begin{eqnarray} \label{20241115eq02}
&& \left| BHC_{j, l, m, \lambda(x)}^{k, \frakA}(f_1, f_2)(x)\right|  \\
&& \lesssim \frac{1}{2^n} \int_{\frac{1}{4} \le |t| \le 4} \left| \left(\widehat{f_1}(\cdot) \phi_1 \left(\frac{\cdot}{2^{j+k}} \right) \right)^{\vee} \left(x-\frac{t}{2^k} \right) \right| \left| \left(\widehat{f_2}(\cdot) \phi_1 \left(\frac{\cdot}{2^{l+2k}} \right) \right)^{\vee} \left(x+\frac{t^2}{2^{2k}} \right) \right| \phi_2 \left(\frac{\lambda(x)}{2^{m+3k}} \right) dt. \nonumber 
\end{eqnarray}

With these, we then have 
\begin{eqnarray*}
&& \left| \sum_{k \in \Z} \sum_{j, l, m \ge 0} BHC_{j, l, m, \lambda(x)}^{k, \frakA}(f_1, f_2)(x) \right| \\
&& \lesssim \sum_{j, l \ge 0} \frac{1}{2^{\max\{j, l\}}} 
\Bigg [\int_{\frac{1}{4} \le |t| \le 4} \sum_{k \in \Z} \left| \left(\widehat{f_1}(\cdot) \phi_1 \left(\frac{\cdot}{2^{j+k}} \right) \right)^{\vee} \left(x-\frac{2^j t}{2^{j+k}} \right) \right| \\
&& \qquad \qquad \qquad  \cdot \left| \left(\widehat{f_2}(\cdot) \phi_1 \left(\frac{\cdot}{2^{l+2k}} \right) \right)^{\vee} \left(x+\frac{2^l t^2}{2^{l+2k}} \right) \right| \cdot  \left( \sum_{m \in \N} \phi_2 \left(\frac{\lambda(x)}{2^{m+3k}} \right) \right)  dt \Bigg ] \\
&& \lesssim \sum_{j, l \ge 0} \frac{1}{2^{\max\{j, l\}}} \cdot \int_{\frac{1}{4} \le |t| \le 4} \calS_{2^j t}^{\phi_1}(x) \calS_{-2^l t^2}^{\phi_1}(x) dt,
\end{eqnarray*}

 Finally, the desired $L^r$ bounds in \eqref{nonstathifoffdiag} follow from the above estimate and the corresponding $L^p$ bounds on $\calS_t^\phi f$ (see, e.g., \cite[Proposition 42]{Lie18} or \cite[Lemma 3.3]{GL20}). 
\end{proof}

\section{Treatment of the high-frequency diagonal component $BHC^{\Delta}$}\label{Mainsect}
Recall that 
$$
BHC^{\Delta}(f_1, f_2):=\sup\limits_{\lambda \in \R} \left|BHC_{\lambda}^{\Delta}(f_1, f_2) \right| \approx \sup_{\lambda \in \R} \left| \sum_{k \in \Z} \sum_{m \in \N} BHC_{m, \lambda}^k(f_1, f_2) \right|
$$
where $BHC_{m, \lambda}^k(f_1, f_2):=BHC_{m, m, m, \lambda}^k(f_1, f_2)$.  We start by splitting.
$$
BHC^{\Delta}(f_1, f_2)  \lesssim  BHC_{+}^{\Delta}(f_1, f_2)\,+\,BHC_{-}^{\Delta}(f_1, f_2), 
$$
with
$$
 BHC_{\pm}^{\Delta}(f_1, f_2):=\sup_{\lambda \in \R} |BHC_{\pm, \lambda}^{\Delta}(f_1, f_2)|:=\sup_{\lambda \in \R} \left|\sum_{k \in \Z_{\pm}} \sum_{m \in \N} BHC_{m, \lambda}^k(f_1, f_2)\right|.
$$
We shall first focus on the term $BHC_{+}^{\Delta}(f_1, f_2)$.

\subsection{The case $k \ge 0$: Preparatives}
We first notice that
$$
 BHC_{+}^{\Delta}(f_1, f_2)=\sup\limits_{\lambda \in \R} \left|BHC_{+, \lambda}^{\Delta}(f_1, f_2)\right| \le \sum_{k \in \Z_{+}} \sum_{m \in \N} \sup_{\lambda} \left|BHC_{m, \lambda}^k(f_1, f_2) \right|. 
$$
From now on, we will fix the choice of $k, m \in \N$. Recall that in this case $|t| \sim 2^{-k}\leq 1$ and hence $|t|^3 \le t^2 \le |t|$. 

We now have that 
$$
BHC_{m, \lambda}^k(f_1, f_2)(x) \approx 2^k \phi_2 \left(\frac{\lambda}{2^{m+3k}} \right) \int_{\R} f_{1, m+k}(x-t) f_{2, m+2k}(x+t^2)e^{i \lambda t^3} \rho(2^k t) dt,
$$
and as a consequence we further divide our operator as 
$$
 \sup_{\lambda \in \R} \left|BHC_{m, \lambda}^k(f_1, f_2) \right| \le \sum_{n \in \Z} \sup_{\lambda \in \R} \left| BHC_{m, \lambda}^{k, I_n^{0, k}}(f_1, f_2)\right|,
$$
with
$$
BHC_{m, \lambda}^{k, I_n^{0, k}}(f_1, f_2)(x):=\one_{I_n^{0, k}}(x) \cdot 2^k \phi_2 \left(\frac{\lambda}{2^{m+3k}} \right) \int_{\R} f_{1, m+k}(x-t) f_{2, m+2k}(x+t^2)e^{i \lambda t^3} \rho(2^k t) dt,
$$
where, throughout the paper, $\one_{[a, b]}$ denotes the characteristic function of $[a, b]$, and for $m \in \N$, $j\in\{1,2\}$ and $k, q \in \Z$,   $f_{j, m}$ stands for $\calF^{-1} \left(\widehat{f}_j (\cdot) \phi_1 \left( \frac{\cdot}{2^{m}} \right) \right)$ and
$$
I_q^{m, k}:=\left[ \frac{q}{2^{\frac{m}{2}+k}}, \frac{q+1}{2^{\frac{m}{2}+k}}   \right].
$$

Notice that in order to recover the corresponding global estimates for $\sup\limits_{\lambda \in \R} \left|BC_{m, \lambda}^k\right|$, it suffices to prove local estimates for each of the terms $\sup\limits_{\lambda \in \R} \left| BC_{m, \lambda}^{k, I_n^{0, k}}\right|$, and then ``glue" them together by using the fact that $\left\{3I_n^{0,k} \right\}_{n \in \Z}$ has the finite intersection property. Therefore, from now on, we may assume without loss of generality that $n=1$. Moreover, for notational simplicity, in the sequel we will replace $I_1^{0, k}$ by $I^k$ and refer to 
$f_{j, m+j k}$ as simply $f_j$, $j\in\{1,2\}$. 

\bigskip 

\noindent \setword{\textbf{Key Heuristic 1.}}{KH1} \emph{Fix $k\geq 0$. Applying a standard linearization procedure we know that there exists a measurable function $\lambda (\cdot): \R \to \R$ such that 
$$
 \sup_{\lambda \in \R} \left| BHC_{m, \lambda}^{k, I^k}(f_1, f_2)(x) \right| \le 2 \left| BHC_{m, \lambda(x)}^{k, I^k}(f_1, f_2)(x) \right|. 
$$
Moreover, importantly, we claim that one can make this choice such that (essentially)}
\begin{center}
\texttt{$\lambda(x)$ takes constant values on intervals of length $2^{-m-2k}$}. 
\end{center}
\bigskip

Indeed, to see this, we notice that 
\begin{eqnarray*}
&& \sup_{\lambda \in \R} \left| BHC_{m, \lambda}^{k, I^k}(f_1, f_2)(x) \right| \sim \sup_{\lambda \in \R} \left| \sum_{p' \sim 2^{m+k}} \one_{I_{p'}^{0, m+2k}} (x) \cdot2^k \phi_2 \left(\frac{\lambda}{2^{m+3k}} \right) \int_{I^k} f_1(x-t) f_2(x+t^2)e^{i \lambda t^3} dt \right|  \\
&& \quad \lesssim \sum_{p' \sim 2^{m+k}} \sup_{\lambda \in \R} \left| \one_{I_{p'}^{0, m+2k}}(x)\cdot2^k \phi_2 \left(\frac{\lambda}{2^{m+3k}} \right) \int_{I^k} f_1(x-t) f_2(x+t^2)e^{i \lambda t^3} dt \right|
\end{eqnarray*}
and, recalling our notational convention $f_j\equiv f_{j, m+jk}$, we immediately deduce that
\begin{enumerate}
\item [$\bullet$] $f_1$ is morally a constant on intervals of length $2^{-m-k}$;
\item [$\bullet$] $f_2$ is morally a constant on intervals of length $2^{-m-2k}$.
\end{enumerate}
Therefore, for each $\lambda \in \R$, one has 
$$
\one_{I_{p'}^{0, m+2k}}(x)  \int_{I^k} f_1(x-t) f_2(x+t^2)e^{i \lambda t^3} dt \approx \one_{I_{p'}^{0, m+2k}}(x)  \int_{I^k} f_1\left(\frac{p'}{2^{m+2k}}-t \right) f_2\left(\frac{p'}{2^{m+2k}}+t^2\right) e^{i \lambda t^3} dt
$$
which justifies our claim. \hfill \qedsymbol{}

\bigskip  

Our goal is to prove the following key result:
\medskip 

\begin{thm} \label{20241018mainthm01}
Let $k \ge 0$ and $m\in \N$. Let further, 
\begin{equation} \label{20240411eq01}
\Lambda_m^k(f_1, f_2, f_3):=2^k\int_{I^k} \int_{I^k} f_1(x-t)f_2(x+t^2)f_3(x)e^{i \lambda(x) t^3} dtdx,
\end{equation}
be the dual form of $BC_{m, \lambda(\cdot)}^{k, I^k}$ defined in {\bf Key Heuristic 1}, where $\lambda (\cdot): I^k \to [2^{m+3k}, 2^{m+3k+1}]$ is any measurable function taking constant values on intervals of length $2^{-m-2k}$. 

Then, there exists some $\epsilon>0$, such that the following estimate holds:
\begin{equation} \label{20240410eq02}
\left|\Lambda_m^k(f_1, f_2, f_3) \right| \lesssim  2^{-\epsilon \min \{2 k, m\}} \prod_{j=1}^3 \left\|f_j \right\|_{L^{p_j} \left(3I^k \right)}, 
\end{equation}
 whenever $(p_1, p_2, p_3)=(2, \infty, 2), (\infty, 2, 2)$ or $(2, 2, \infty)$
\end{thm}
\medskip

We start our analysis of \eqref{20240411eq01} with a preliminary observation: applying Cauchy-Schwarz we have 
\begin{eqnarray*}
\left| \Lambda_m^k(f_1, f_2, f_3) \right|^2%
&\le&  2^{2k} \int_{I^k} \left| \sum_{q \sim 2^{\frac{m}{2}}} \int_{I_q^{m, k}} f_1(x-t)f_2(x+t^2)e^{i\lambda(x)t^3} dt \right|^2 dx \cdot \left\|f_3 \right\|_{L^2 \left(3I^k \right)}^2
\end{eqnarray*}
and using the information that for $t \in I_q^{m, k}$ one has
$$
\lambda(x) t^3=\frac{3q^2t\lambda(x)}{2^{m+2k}}-\frac{2q^3\lambda(x)}{2^{\frac{3m}{2}+3k}}+O(1),
$$
we conclude that
\begin{eqnarray}  \label{20240429eq01}
\left| \Lambda_m^k(f_1, f_2, f_3) \right|^2%
&\lesssim& 2^k \left\|f_3 \right\|_{L^2\left(3I^k \right)}^2 \left[ {\bf L}_{m, k}(f_1, f_2) \right]^2,   
\end{eqnarray}
where
\begin{equation} \label{20240426eq01}
\left[ {\bf L}_{m, k}(f_1, f_2) \right]^2:=2^{k} \int_{I^k} \left| \sum_{q \sim 2^{\frac{m}{2}}} \int_{I_q^{m, k}} f_1(x-t)f_2(x+t^2) e^{-2i \cdot \frac{q^3 \lambda(x)}{2^{\frac{3m}{2}+3k}}}  \cdot e^{3i \cdot \frac{q^2 \lambda(x)}{2^{m+2k}} t} dt \right|^2 dx\,.
\end{equation} 

Thus, we reduce matters to the analysis of ${\bf L}_{m, k}(f_1, f_2)$.

\subsection{A sparse-uniform dichotomy} \label{20241216subsec01}
We analyze the term \eqref{20240426eq01} via two different scenarios, namely, whether the $L^2$ information of the two input functions $f_1$ and $f_2$ are spatially concentrated or uniformly distributed. More precisely, let $\mu \in (0, 1)$ be sufficiently small and, for $i \in \{1, 2\}$ and $f_i$ define the \emph{sparse and uniform index sets} as follows: 
$$
\calS_\mu(f_i):=\left\{r \sim 2^{m+(i-1)k}: \frac{1}{\left|I_r^{ik+m} \right|} \int_{3I_r^{0, ik+m}}\left|Mf_i \right|^2 \gtrsim \frac{2^{\mu\min \left\{m, 2 k \right\}}}{|I^k|} \int_{3I^k} |f_i|^2 \right\}
$$
and
$$
\calU_\mu(f_i):=\left\{r \sim 2^{m+(i-1)k}: \frac{1}{\left|I_r^{ik+m} \right|} \int_{3I_r^{0, ik+m}}\left|Mf_i \right|^2 \lesssim \frac{2^{\mu\min \left\{m, 2 k \right\}}}{|I^k|} \int_{3I^k} |f_i|^2 \right\}\,,
$$
respectively, where $M$ is the Hardy-Littlewood maximal operator.

We shall now decompose our input functions accordingly:
$$
f_i=f_i^{\calS_\mu}+f_i^{\calU_\mu}, \qquad i \in \{1, 2\}, 
$$
where 
$$
f_i^{\calS_\mu}:=\sum_{r \in \calS_\mu(f_i)} \one_{I_r^{0, ik+m}} f_i, \quad \textrm{and} \quad f_i^{\calU_\mu}:=\sum_{r \in \calU_\mu(f_i)} \one_{I_r^{0, ik+m}} f_i, \qquad i \in \{1, 2\}. 
$$
Let us now decompose the term ${\bf L}_{m, k}(f_1, f_2)$ into two parts:
$$
{\bf L}_{m, k}(f_1, f_2) \le {\bf L}_{m, k}^{\calU_\mu}(f_1, f_2)+{\bf L}_{m, k}^{\calS_\mu}(f_1, f_2), 
$$
where 
\begin{enumerate}
    \item [$\bullet$] the \emph{uniform} component ${\bf L}_{m, k}^{\calU_\mu}(f_1, f_2)$ is given by $$
    {\bf L}_{m, k}^{\calU_\mu}(f_1, f_2):={\bf L}_{m, k}\left(f_1^{\calU_\mu}, f_2^{\calU_\mu}\right).
    $$
    \item [$\bullet$] the \emph{sparse} component ${\bf L}_{m, k}^{\calS_\mu}(f_1, f_2)$ is given by
    $$
    {\bf L}_{m, k}^{\calS_\mu}(f_1, f_2):={\bf L}_{m, k}\left(f_1^{\calS_\mu}, f_2^{\calU_\mu}\right)+{\bf L}_{m, k}\left(f_1^{\calU_\mu}, f_2^{\calS_\mu}\right)+{\bf L}_{m, k}\left(f_1^{\calS_\mu}, f_2^{\calS_\mu}\right).
    $$
\end{enumerate}
Note that the uniform/sparse parts enjoy the following properties:
\begin{enumerate}
\item [$\bullet$] For $f_i^{\calU_\mu}$, one has $f_i^{\calU_\mu} \in L^\infty(\R)$ with
\begin{equation} \label{uniformest} 
\left\|f_i^{\calU_\mu} \right\|_{L^\infty(3I^k)} \lesssim 2^{\frac{\mu \min \left\{m, 2k \right\}}{2}} 2^{\frac{k}{2}} \left\|f_i \right\|_{L^2(3I^k)};
\end{equation} 
\item [$\bullet$] For $f_i^{\calS_\mu}$, one has
\begin{equation} \label{20240429eq02}
\max_{i \in \{1, 2\}} \left| \supp \ f_i^{\calS_\mu} \right| \lesssim 2^{-\mu  \min\left\{m, 2k \right\}}|I^k|. 
\end{equation} 
\end{enumerate}

\subsection{Treatment of the sparse component ${\bf L}_{m, k}^{\calS_\mu}(f_1, f_2)$}
We consider two different cases.

\medskip 

\noindent\textsf{\underline{Case I: $f_1$ sparse.}} \quad In this situation we need to estimate a (generic) term of the form ${\bf L}_{m, k}\left(f_1^{\calS_\mu}, f_2\right)$. To begin with, for each $I \subseteq I^k$ interval, we set $\widetilde{\calS_\mu^{3I}}(f_1):=\bigcup\limits_{\substack{r \in \calS_\mu(f_1) \\ I_r^{0, k+m} \subseteq 3I}} I_r^{0, k+m}$. Therefore, by Cauchy-Schwarz and \eqref{20240429eq02}, we have 
\begin{eqnarray*}
\left[{\bf L}_{m, k}\left(f_1^{\calS_\mu}, f_2 \right) \right]^2%
&\le& 2^k \int_{I^k} \left( \int_{I^k} \left|f_1^{\calS_\mu}(x-t) \right| \left|f_2(x+t^2) \right| dt \right)^2 dx \\
&\le& 2^k \left| \widetilde{\calS_\mu^{3I^k}}(f_1) \right| \cdot \int_{I^k} \int_{I^k} \left|f^{\calS_\mu}_1(x-t) \right|^2 |f_2(x+t^2)|^2 dtdx \\
&\le& 2^k \cdot 2^{-\mu \min\{m, 2k\}} |I^k| \cdot \left\|f_1 \right\|^2_{L^2(3I^k)}\left\|f_2 \right\|^2_{L^2(3I^k)}.
\end{eqnarray*}
Therefore, combining the above estimate with \eqref{20240429eq01}, we have 
\begin{eqnarray*}
\left[\Lambda_m^k(f^{\calS^\mu}_1, f_2, f_3) \right]^2%
&\lesssim&  2^k \left\|f_3 \right\|_{L^2(3I^k)}^2 \cdot 2^k \cdot 2^{-\mu \min\{m, 2 k\}} |I^k| \cdot \left\|f_1 \right\|^2_{L^2(3I^k)}\left\|f_2 \right\|^2_{L^2(3I^k)} \\
&\lesssim& 2^{-\mu \min\{m, 2k\}}  \prod_{j=1}^3 \left\|f_j \right\|^2_{L^{p_j}\left(3I^k \right)}, 
\end{eqnarray*}
for $(p_1, p_2, p_3)=(2, \infty, 2), (\infty, 2, 2)$ and $(2, 2, \infty)$, which gives the desired estimate \eqref{20240410eq02}. 

\medskip 

\noindent\textsf{\underline{Case II: $f_2$ sparse.}} \quad In this case, we divide our analysis into two scenarios. We first prove the estimate \eqref{20240410eq02} with $(p_1, p_2, p_3)=(2, \infty, 2)$ or $(\infty, 2, 2)$. For this purpose, we focus our attention on the term ${\bf L}_{m, k} \left(f_1, f_2^{\calS_\mu} \right)$. As before, for any interval $I \subseteq I^k$, we set $\widetilde{\calS_{\mu}^{3I}} \left(f_2 \right):=\bigcup\limits_{\substack{r \in \calS_\mu(f_2) \\ I_r^{0, 2k+m} \subseteq 3I}} I_r^{0, 2k+m}$. With these, we have
\begin{eqnarray} \label{20240429eq10}
\left[{\bf L}_{m, k} \left(f_1, f_2^{\calS_\mu} \right) \right]^2 %
&\lesssim&  2^k \int_{I^k} \left( \int_{I^k} \left|f_1 (x-t) \right| \left|f_2^{\calS_\mu}(x+t^2) \right| dt \right)^2 dx \nonumber \\
&=& 2^k \sum_{\ell \sim 2^k} \int_{I_\ell^{0, 2k}} \left( \int_{I^k} \left|f_1 (x-t) \right| \left|f_2^{\calS_\mu}(x+t^2) \right| dt \right)^2 dx. 
\end{eqnarray}
Note that for each $\ ell\sim 2^k$, then by Cauchy-Schwarz and a change of variables, we have 
$$
\int_{I_\ell^{0, 2k}} \left( \int_{I^k} \left|f_1 (x-t) \right| \left|f_2^{\calS_\mu}(x+t^2) \right| dt \right)^2 dx \lesssim 2^k \cdot \left| \widetilde{\calS_\mu^{3I_\ell^{0, 2k}}} (f_2) \right|  \int_{I_\ell^{0, 2k}} \int_{I^k} \left|f_1 (x-t) \right|^2 \left|f_2^{\calS_\mu}(x+t^2) \right|^2 dt  dx.
$$
Thus, combining the above estimate with \eqref{20240429eq01} and \eqref{20240429eq10}, we have 
\begin{eqnarray*} 
\left[\Lambda_m^k(f_1, f^{\calS^\mu}_2, f_3) \right]^2%
&\lesssim& 2^{3k} \left\|f_3 \right\|_{L^2(3I^k)}^2 \cdot \sum_{\ell \sim 2^k}  \left| \widetilde{\calS_\mu^{3I_\ell^{0, 2k}}} (f_2)\right|  \int_{I_\ell^{0, 2k}} \int_{I^k} \left|f_1 (x-t) \right|^2 \left|f_2^{\calS_\mu}(x+t^2) \right|^2 dtdx \nonumber  \\
&\lesssim& 2^{3k} \left\|f_2 \right\|_{L^\infty(3I^k)}^2 \left\|f_3 \right\|_{L^2(3I^k)}^2 \cdot \sum_{\ell \sim 2^k} \left| \widetilde{\calS_\mu^{3I_\ell^{0, 2k}}} (f_2) \right| \int_{I_\ell^{0, 2k}} \int_{I^k} \left|f_1 (x-t) \right|^2  dtdx \\
&\lesssim& 2^k \left\|f_1 \right\|^2_{L^2(3I^k)}\left\|f_2 \right\|_{L^\infty(3I^k)}^2 \left\|f_3 \right\|_{L^2(3I^k)}^2 \cdot \sum_{\ell \sim 2^k} \left| \widetilde{\calS_\mu^{3I_\ell^{0, 2k}}}  (f_2) \right| \nonumber \\
&\lesssim& 2^{-\mu \min\{m, 2k\}} \left\|f_1 \right\|^2_{L^2 \left(3I^k \right)} \left\|f_2 \right\|^2_{L^\infty \left(3I^k \right)} \left\|f_3 \right\|^2_{L^2 \left(3I^k \right)}. \nonumber
\end{eqnarray*}
Similarly, one can also show that
$$
\left[\Lambda_m^k(f_1, f^{\calS^\mu}_2, f_3) \right]^2 \lesssim 2^{-\mu \min\{m, 2k\}} \left\|f_1 \right\|^2_{L^\infty \left(3I^k \right)} \left\|f_2 \right\|^2_{L^2 \left(3I^k \right)} \left\|f_3 \right\|^2_{L^2 \left(3I^k \right)}. 
$$
Next, we establish estimate \eqref{20240410eq02} for $(p_1, p_2, p_3)=(2, 2, \infty)$. For this, we first note that  
\begin{align*}
\left|\Lambda_m^k\left(f_1, f^{\calS_\mu}_2, f_3 \right) \right|^2 
& \le 2^{2k} \left( \int_{I^k} \int_{I^k} |f_1(x-t)|\left|f^{\calS_\mu}_2(x+t^2)\right||f_3(x)|dxdt \right)^2 \\
& \le 2^{2k} \left( \int_{I^k} \int_{2I^k} |f_3(x+t)|\left|f^{\calS_\mu}_2(x+t+t^2)\right||f_1(x)| dxdt \right)^2 \\
& \le 2^{2k} \int_{2I^k} \left(\int_{I^k} |f_3(x+t)| \left| f_2^{\calS_\mu}(x+t+t^2) \right| dt \right)^2 dx \cdot \left\|f_1 \right\|_{L^2(3I^k)}^2 \\
& \le 2^{2k} \left|\widetilde{\calS_{\mu}^{3I^k}} \left(f_2 \right) \right| \int_{2I^k} \int_{I^k}|f_3(x+t)|^2 \left| f_2^{\calS_\mu}(x+t+t^2) \right|^2 dtdx \cdot \left\|f_1 \right\|_{L^2(3I^k)}^2 \\
& \le 2^{2k} \cdot 2^{-\mu \min \left\{m, 2k\right\}} |I^k| \cdot  \int_{2I^k} \int_{I^k} |f_2(x+t+t^2)|^2dtdx \cdot \left\|f_1 \right\|_{L^2(3I^k)}^2 \left\|f_3 \right\|_{L^\infty(3I^k)}^2 \\
& \le 2^{-\mu \min \left\{m, 2k\right\}} \left\|f_1 \right\|_{L^2(3I^k)}^2 \left\|f_2 \right\|_{L^2(3I^k)}^2 \left\|f_3 \right\|_{L^\infty(3I^k)}^2,
\end{align*}
which gives the desired estimate.

\subsection{Treatment of the uniform component ${\bf L}_{m, k}^{\calU_\mu}(f_1, f_2)$: An overview} \label{20241023overview01}

In this section, for reader's convenience, we offer a very brief recap of the strategy employed for treating the uniform component  ${\bf L}_{m, k}^{\calU_\mu}(f_1, f_2)$ as displayed in the last part of the introductory Section \ref{Planmainideas}.

To begin with, let us write 
$$
\widetilde{\lambda}(x):=\frac{\lambda (x)}{2^{\frac{m}{2}+3k}},
$$
and deduce that as a consequence of our assumptions we have $\widetilde{\lambda}: I^k \to \left[2^{\frac{m}{2}}, 2^{\frac{m}{2}+1} \right]$ is a measurable function which takes constant values on intervals of length $2^{-m-2k}$. Consequently, our main expression now becomes
\begin{eqnarray} \label{20241004eq01}
&& \:\:\:\:\:\:\left[ {\bf L}^{\calU_\mu}_{m, k}(f_1, f_2) \right]^2:= 2^{k} \int_{I^k} \left| \sum_{q \sim 2^{\frac{m}{2}}} \int_{I_q^{m, k}} f_1^{\calU_\mu}(x-t)f_2^{\calU_\mu}(x+t^2) e^{-2i \cdot \frac{q^3 \widetilde{\lambda}(x)}{2^m}}  \cdot e^{3i 2^k \cdot \frac{q^2 \widetilde{\lambda}(x)}{2^{\frac{m}{2}}} t} dt \right|^2 dx \nonumber  \\
&&\lesssim \left[{\calL}_{m, k}^{\calU_\mu}(f_1, f_2) \right]^2 \ :=  2^{\frac{m}{2}+k} \cdot  \int_{I^k} \left( \sum_{q \sim 2^{\frac{m}{2}}} \left| \int_{I_q^{m, k}} f^{\calU_\mu}_1(x-t)f^{\calU_\mu}_2(x+t^2) e^{3i2^k \cdot \frac{q^2 \widetilde{\lambda}(x)}{2^{\frac{m}{2}}}t} dt\right|^2 \right)dx.
\end{eqnarray} 
In the sequel, we will divide our proof into two cases: 
\medskip 

\noindent \underline{\textsf{Case I: $k \ge \frac{m}{2}$.}} \quad Recalling now Proposition \ref{maxjointctrl} (\cite[Theorem 4.3]{HL23}) we further consider
two regimes:
\smallskip

\begin{enumerate}
    \item [$\bullet$]  \underline{\textsf{Case I.1: $k \ge m$.}} \quad In this situation, since $2^{-m-2k}\geq 2^{-3k}$, our desired estimate follows directly from Proposition \ref{maxjointctrl}.
\smallskip        
    
    \item [$\bullet$]  \underline{\textsf{Case I.2: $\frac{m}{2} \le k \le m$.}}\quad  In this situation, we follow a two-step approach: 
\medskip    
 \begin{itemize}   
    \item \underline{\textsf{Step I.2.1:}} We first implement the constancy propagation procedure in order to extend the constancy of $\widetilde{\lambda}$ from intervals of length $2^{-m-2k}$ to intervals of length $2^{-m-k}$. 
\smallskip    
    \item \underline{\textsf{Step I.2.2:}} Once we reach the scale $2^{-m-k}$, we notice that we are in the regime $2^{-m-k} \ge 2^{-3k}$, and hence  Proposition \ref{maxjointctrl} applies once again.
 \end{itemize}   
\end{enumerate}

\medskip

\noindent \underline{\textsf{Case II: $0 \le k \le \frac{m}{2}$.}} In this case, we employ a two-step approach:
\smallskip
\begin{itemize}
\item \underline{\textsf{Step II.1:}} As at the Step I.2.1, we first apply the constancy propagation procedure in order to extend the constancy of $\widetilde{\lambda}$ from intervals of length $2^{-m-2k}$ to intervals of length $2^{-m-k}$.
\smallskip

\item \underline{\textsf{Step II.2:}} Once we reach the $2^{-m-k}-$constancy scale for  $\widetilde{\lambda}$, we apply to \eqref{20240411eq01} the Rank I LGC method in order to conclude our proof. 
\end{itemize}
\smallskip
\noindent We end this section with the following:

\begin{rem}
In the sequel, in order to illustrate the new ideas in the present paper, we focus our presentation on the second part $0 \le k \le \frac{m}{2}$, and only briefly mention the treatment required for the first part $k \ge \frac{m}{2}$ (for the latter, see \textsf{Step I} in the proof of Proposition \ref{20241018prop01}). 
\end{rem}

\subsection{Treatment of the uniform component ${\bf L}_{m, k}^{\calU_\mu}(f_1, f_2)$: Case II, Step II.1 in Overview} \label{20241217subsec01} 

Recall $\calL_{m, k}^{\calU_\mu}(f_1, f_2)$ defined in \eqref{20241004eq01}. Essentially, our goal in this first part is to show that 
\begin{equation} \label{20241005eq01} 
\textnormal{\texttt{the length of constancy of the measurable function $\widetilde{\lambda}$ can be extended to $2^{-m-k}$.}}  \tag{{\bf H1}}
\end{equation}
To make the above heuristic rigorously, for each $p \sim 2^m$, we define\footnote{For notational simplicity, throughout this section, we set $f_1=f_1^{\calU_\mu}$ and $f_2=f_2^{\calU_\mu}$.}
\begin{equation} \label{20241004eq21}
V_{m, k}^p(f_1, f_2):=2^{\frac{3m}{2}+k} \sum_{q \sim 2^{\frac{m}{2}}} \int_{I_p^{0, k+m}} \left| \int_{I_q^{m, k}} f_1(x-t)f_2(x+t^2) e^{3i 2^k \cdot \frac{q^2 \widetilde{\lambda}(x)}{2^{\frac{m}{2}}} t} dt \right|^2  dx\,,
\end{equation} 
and remark that
$$
\left[\calL_{m, k}^{\calU_\mu} (f_1, f_2) \right]^2=\frac{1}{2^m} \sum_{p \sim 2^m} V_{m, k}^p(f_1, f_2).
$$
Next, we claim that one can essentially replace the term $f_1(x-t)$ in \eqref{20241004eq21} by $f_1\left(\frac{p}{2^{k+m}}-t \right)$. Indeed, this is achieved by the following approximation argument, whose proof resembles the one in \cite[Proposition 4.5 and Corollary 4.6]{HL23}, and therefore, is omitted.

\begin{prop} \label{20241006prop01}
For $k, m \in \N$ with $0 \le k \le \frac{m}{2}$, $p \sim 2^m$, $q \sim 2^{\frac{m}{2}}$ and $\del>4\mu$ sufficiently small, one has
\begin{eqnarray*}
&& \int_{I_p^{0,k+m}} \left| \int_{I_q^{m, k}} f_1(x-t) f_2(x+t^2) e^{3i 2^k \frac{q^2 \widetilde{\lambda(x)}}{2^{\frac{m}{2}}} t} dt \right|^2 dx \\
&& \lesssim \frac{2^{\del k}}{\left|I_0^{0, k+m}\right|} \int_{3I_0^{0, k+m}} \int_{I_0^{0, k+m}} \left| \int_{I_q^{m, k}} f_1 \left(\alpha+\frac{p}{2^{k+m}}-t \right) f_2 \left(x+\frac{p}{2^{k+m}}+t^2 \right)  e^{3i 2^k \frac{q^2 \widetilde{\lambda}\left(x+\frac{p}{2^{k+m}} \right)}{2^{\frac{m}{2}}} t} dt \right|^2 dx d\alpha \\
&& \quad +O \left(2^{-(\del-2\mu)k} \cdot 2^{-2m-k} \left\| f_1 \right\|_{L^2 \left(3I^k \right)}^2 \left\| f_2 \right\|_{L^2 \left(3I^k \right)}^2 \right).
\end{eqnarray*}
\end{prop}

Once at this point, we group $2^{m-k}$ many consecutive $p$'s: specifically, we decompose 
$$
\left[ 2^m, 2^{m+1} \right]=\bigcup_{\widetilde{p} \sim 2^k} \left[\widetilde{p} 2^{m-k}, (\widetilde{p}+1)2^{m-k} \right]
$$
and define
\begin{equation} \label{20241217eq10}
V_{m, k}^{\widetilde{p}}(f_1, f_2):=\frac{1}{2^{m-k}} \sum_{p \in \left[\widetilde{p} 2^{m-k}, (\widetilde{p}+1)2^{m-k} \right]} V_{m, k}^p (f_1, f_2)\,.
\end{equation}
With these done we are ready to state the following
\medskip

\noindent \setword{\textbf{Key Heuristic 2}}{KH2}. One of the following statements is \emph{essentially true}: for each $\widetilde{p} \sim 2^k$, 
\begin{enumerate}
    \item [$\bullet$] \emph{either} $\lambda (\cdot)$ is morally constant on each interval $I_p^{0, k+m}$ with $p \in \left[ \widetilde{p} 2^{m-k}, (\widetilde{p}+1) 2^{m-k} \right]$;
    \item [$\bullet$] \emph{or} there exists some $\epsilon>0$, such that $\left| V_{m, k}^{\widetilde{p}} (f_1, f_2) \right| \lesssim 2^{-\epsilon k} \left\|f_1 \right\|_{L^2(3I^k)}^2 \left\|f_2 \right\|_{L^2(3I^k)}^2$.
\end{enumerate}
In order to make the above heuristic precise, we start by further decomposing the range of $\widetilde{\lambda}$ as follows: fix $\epsilon_0>0$ sufficiently small, and then write
$$
\textrm{Range of} \ \widetilde{\lambda}= \bigcup_{\ell \sim 2^{\frac{m}{2}-\epsilon_0 k}} \left[ \ell 2^{\epsilon_0 k}, (\ell+1) 2^{\epsilon_0 k} \right]. 
$$
Next, we partition each $I_p^{0, k+m}$ into the preimages corresponding to the above decomposition: for each $\ell \sim 2^{\frac{m}{2}-\epsilon_0 k}$, we set
$$
\frakM_\ell \left(I_p^{0, k+m} \right):=\left\{x \in I_p^{0, k+m}: \widetilde{\lambda}(x) \in \left[ \ell 2^{\epsilon_0 k}, (\ell+1)2^{\epsilon_0 k}  \right] \right\}, 
$$
define
\begin{enumerate}
    \item [$\bullet$] \emph{the heavy set of $\ell$-indices}:
    \begin{equation} \label{20241011eq32}
    \calH_p^k:=\left\{\ell \sim 2^{\frac{m}{2}-\epsilon_0 k}: \left| \frakM_\ell \left(I_p^{0, k+m} \right) \right| \gtrsim \left|I_p^{0, k+m} \right| 2^{-\epsilon_0 k} \right\}
    \end{equation} 
    \item [$\bullet$] \emph{the light set of $\ell$-indices}:
    \begin{equation} \label{20241011eq40}
      \calL_p^k:=\left\{\ell \sim 2^{\frac{m}{2}-\epsilon_0 k}: \left| \frakM_\ell \left(I_p^{0, k+m} \right) \right| \lesssim \left|I_p^{0, k+m} \right| 2^{-\epsilon_0 k} \right\},
    \end{equation}
\end{enumerate}
and set 
$$
\widetilde{\calH_p^k}:=\bigcup_{\ell \in \calH_p^k} \frakM_\ell \left(I_p^{0, k+m} \right) \quad \textrm{and} \quad \widetilde{\calL_p^k}:=\bigcup_{\ell \in \calL_p^k} \frakM_\ell \left(I_p^{0, k+m} \right).
$$
Observe now that
$$
I_p^{0, k+m}=\widetilde{\calH_p^k} \cup \widetilde{\calL_p^k} \quad \textrm{with} \quad \# \calH_p^k \lesssim 2^{\epsilon_0 k}. 
$$
Using the above decomposition, we further write $V_{m, k}^{\widetilde{p}}$ into two parts as follows: 
$$
V_{m, k}^{\widetilde{p}}(f_1, f_2)=V_{m, k}^{\widetilde{p}, \calH}(f_1, f_2)+V_{m, k}^{\widetilde{p}, \calL}(f_1, f_2), 
$$
where
$$
V_{m, k}^{\widetilde{p}, \calH}(f_1, f_2):=2^{\frac{m}{2}+2k} \sum_{p \in \left[\widetilde{p} 2^{m-k}, (\widetilde{p}+1)2^{m-k} \right]} \sum_{q \sim 2^{\frac{m}{2}}} \int_{\widetilde{\calH_p^k}} \left| \int_{I_q^{m, k}} f_1(x-t)f_2(x+t^2) e^{3i 2^k \cdot \frac{q^2 \widetilde{\lambda}(x)}{2^{\frac{m}{2}}} t} dt \right|^2  dx.
$$
and 
$$
V_{m, k}^{\widetilde{p}, \calL}(f_1, f_2):=2^{\frac{m}{2}+2k} \sum_{p \in \left[\widetilde{p} 2^{m-k}, (\widetilde{p}+1)2^{m-k} \right]} \sum_{q \sim 2^{\frac{m}{2}}} \int_{\widetilde{\calL_p^k}} \left| \int_{I_q^{m, k}} f_1(x-t)f_2(x+t^2) e^{3i 2^k \cdot \frac{q^2 \widetilde{\lambda}(x)}{2^{\frac{m}{2}}} t} dt \right|^2  dx.
$$
In what follows we first address the light component $V_{m, k}^{\widetilde{p}, \calL}$ whose treatment requires a completely different insight relative to the work performed in \cite{HL23}:

\begin{prop} \label{20241018prop01}
There exists some $\epsilon_1>0$, such that 
$$
\left|V_{m, k}^{\widetilde{p}, \calL}(f_1, f_2) \right| \lesssim 2^{-\epsilon_1 k} \left\|f_1 \right\|_{L^2(3I^k)}^2 \left\|f_2 \right\|^2_{L^2(3I^k)}. 
$$
\end{prop}

\begin{proof} \textsf{Step I: Initial reduction.} Without loss of generality, we may assume that $0 \le k \le \frac{m}{2}-\epsilon_2 m$ with $\epsilon_2>0$ being sufficiently small. This follows from the decay estimate for the case when $k=\frac{m}{2}$ (see, Subsection  \ref{20241028subsec01} below) and the continuity of the $\Lambda_{j, l, m}^k$ bounds (see, \cite[Appendix 13.3]{HL23}). 

Moreover, we may also assume that for each $p \in \left[(\widetilde{p}-20)2^{m-k}, (\widetilde{p}+20)2^{m-k}\right]$, $\widetilde{\calL_p^k}=I_p^{0, k+m}$. Otherwise, one can apply an extension argument as in the proof in \cite[\textsf{Step I}, Proposition 4.7]{HL23}.

Therefore, by the above observations and Proposition \ref{20241006prop01}, it suffices to consider the term
\begin{eqnarray*}
&& \frac{2^{\frac{m}{2}+2k+\del k}}{\left|I_0^{0, k+m} \right|} \sum_{p \in \left[\widetilde{p}2^{m-k}, (\widetilde{p}+1) 2^{m-k}  \right]} \sum_{q \sim 2^\frac{m}{2}} \int_{3I_0^{0, k+m}} \int_{I_0^{0, k+m}} \bigg| \int_{I_q^{m, k}} f_1 \left(\alpha+\frac{p}{2^{k+m}}-t \right) \\
&& \qquad \qquad \qquad \qquad  \cdot f_2 \left(x+\frac{p}{2^{k+m}}+t^2 \right)  e^{3i 2^k \frac{q^2 \widetilde{\lambda}\left(x+\frac{p}{2^{k+m}} \right)}{2^{\frac{m}{2}}} t} dt \bigg|^2 dx d\alpha, 
\end{eqnarray*}
which, by assuming wlog $\alpha=0$, further reduces to
\begin{eqnarray*}
&& {\bf V}_{m, k}^{\widetilde{p}, \calL}(f_1, f_2):=2^{\frac{m}{2}+2k+\del k} \sum_{p \in \left[\widetilde{p}2^{m-k}, (\widetilde{p}+1) 2^{m-k}  \right]} \sum_{q \sim 2^\frac{m}{2}}  \int_{I_0^{0, k+m}} \bigg| \int_{I_q^{m, k}} f_1 \left(\frac{p}{2^{k+m}}-t \right) \\
&& \qquad \qquad \qquad \qquad  \cdot f_2 \left(x+\frac{p}{2^{k+m}}+t^2 \right)  e^{3i 2^k \frac{q^2 \widetilde{\lambda}\left(x+\frac{p}{2^{k+m}} \right)}{2^{\frac{m}{2}}} t} dt \bigg|^2 dx, 
\end{eqnarray*}
We end our preparatives with the following

\begin{rem}\label{Simplif} Throughout the current proof we will simplify our reasonings by making use---in numerous instances---of the fact that
\begin{itemize}
\item $f_1$ is morally constant on intervals of length $2^{-m-k}$; 
\item  $f_2$ is morally a constant on intervals of length $2^{-m-2k}$.
\end{itemize}
These simplifications can be made rigorously by either subdividing the range of $q$ and then choosing representatives within each subdivision or by simply further subdividing the $t-$integration interval---see \emph{e.g.} Step I in the proof of Proposition 4.9. in \cite{HL23}. In order to remove excessive technicalities from our exposition and thus improve its readability we will skip detailing these elements here.
\end{rem}

\medskip 

\noindent \textsf{Step II: Further reductions via $TT^*$.} We begin by applying the change of variable $t \mapsto t+\frac{q}{2^{\frac{m}{2}+k}}$, which gives
\begin{eqnarray*}
&& {\bf V}_{m, k}^{\widetilde{p}, \calL}(f_1, f_2)\approx2^{\frac{m}{2}+2k+\del k} \sum_{p \in \left[\widetilde{p}2^{m-k}, (\widetilde{p}+1) 2^{m-k}  \right]} \sum_{q \sim 2^\frac{m}{2}}  \int_{I_0^{0, k+m}} \bigg| \int_{I_0^{m, k}} f_1 \left(\frac{p}{2^{k+m}}-\frac{q}{2^{\frac{m}{2}+k}}-t \right) \\
&& \qquad \qquad \qquad \qquad  \cdot f_2 \left(x+\frac{p}{2^{k+m}}+\frac{2tq}{2^{\frac{m}{2}+k}}+\frac{q^2}{2^{m+2k}} \right)  e^{3i 2^k \frac{q^2 \widetilde{\lambda}\left(x+\frac{p}{2^{k+m}} \right)}{2^{\frac{m}{2}}} t} dt \bigg|^2 dx\,,
\end{eqnarray*}
where above we have used the relation $\left(t+\frac{q}{2^{\frac{m}{2}+k}} \right)^2=\frac{2tq}{2^{\frac{m}{2}+k}}+\frac{q^2}{2^{m+2k}}+O \left(\frac{1}{2^{m+2k}} \right)$ together with the fact that\footnote{Recall Remark \ref{Simplif}.}  $f_2$ is morally a constant on intervals of length $2^{-m-2k}$.

Next, for each $p \in \left[\widetilde{p} 2^{m-k}, (\widetilde{p}+1) 2^{m-k} \right]$, we write
\begin{equation} \label{20241024eq01}
p=\widetilde{p}2^{m-k}+\beta 2^{\frac{m}{2}-k}+\gamma, 
\end{equation} 
for some $\beta \in \{0, 1, \dots 2^{\frac{m}{2}-1} \}$ and $\gamma \in \left\{0, 1, \dots, 2^{\frac{m}{2}-k}-1 \right\}$. With this, we have\footnote{Here, for notational convenience, we slightly abuse the notation and set $f_1 (\cdot)=f_1 \left(\cdot+\frac{\widetilde{p}}{2^{2k}} \right)$, $f_2 (\cdot)=f_2 \left(\cdot+\frac{\widetilde{p}}{2^{2k}} \right)$ and $\widetilde{\lambda} (\cdot)=\widetilde{\lambda} \left(\cdot+\frac{\widetilde{p}}{2^{2k}} \right)$.}
\begin{eqnarray*}
&& {\bf V}_{m, k}^{\widetilde{p}, \calL}(f_1, f_2)=2^{\frac{m}{2}+2k+\del k} \sum_{\beta=0}^{2^{\frac{m}{2}}-1} \sum_{\gamma=0}^{2^{\frac{m}{2}-k}-1}  \sum_{q \sim 2^{\frac{m}{2}}} \int_{I_0^{0, k+m}} \bigg | \int_{I_0^{m, k}} f_1 \left(\frac{\beta}{2^{\frac{m}{2}+2k}}+\frac{\gamma}{2^{k+m}}-\frac{q}{2^{\frac{m}{2}+k}}-t \right) \\
&& \qquad \cdot f_2 \left(x+\frac{\beta}{2^{\frac{m}{2}+2k}}+\frac{\gamma}{2^{k+m}}+\frac{2tq}{2^{\frac{m}{2}+k}}+\frac{q^2}{2^{m+2k}} \right)  e^{3i 2^k \frac{q^2 \widetilde{\lambda}\left(x+\frac{\beta}{2^{\frac{m}{2}+2k}}+\frac{\gamma}{2^{k+m}}\right)}{2^{\frac{m}{2}}} t} dt \bigg|^2 dx. 
\end{eqnarray*}
We next apply the following change of variables:
\begin{enumerate}
    \item [(1)] $\beta \mapsto \beta-\left\lfloor \frac{q^2}{2^{\frac{m}{2}}} \right\rfloor$;
    \item [(2)] $\gamma \mapsto \gamma-\left\lfloor \left\{\frac{q^2}{2^{\frac{m}{2}}} \right\} 2^{\frac{m}{2}-k} \right\rfloor$;
    \item [(3)] $x \mapsto x-\frac{\left\{\left\{\frac{q^2}{2^{\frac{m}{2}}} \right\} 2^{\frac{m}{2}-k} \right\}}{2^{m+k}}$.
\end{enumerate}
Using\footnote{Recall Remark \ref{Simplif}.} now that $f_1$ is morally constant on intervals of length $2^{-m-k}$ and Cauchy-Schwarz, we deduce
\begin{eqnarray} \label{20241006eq30}
&& {\bf V}_{m, k}^{\widetilde{p}, \calL}(f_1, f_2) \lesssim 2^{\frac{m}{2}+2k+\del k} \sum_{q \sim 2^{\frac{m}{2}}} \sum_{\beta=-2^{\frac{m}{2}}}^{2^{\frac{m}{2}}} \sum_{\gamma=-2^{-\frac{m}{2}-k}}^{2^{\frac{m}{2}-k}} \int_{3I_0^{0, k+m}} \nonumber \\
&& \qquad  \Bigg | \int_{I_0^{m, k}} f_1 \left(\frac{\beta}{2^{\frac{m}{2}+2k}}+\frac{\gamma}{2^{k+m}}-\frac{q}{2^{\frac{m}{2}+k}}-t-\frac{q^2}{2^{m+2k}} \right) f_2 \left(x+\frac{\beta}{2^{\frac{m}{2}+2k}}+\frac{\gamma}{2^{k+m}}+\frac{2tq}{2^{\frac{m}{2}+k}} \right)  \nonumber \\
&& \qquad \qquad \qquad  \cdot e^{3i 2^k \frac{q^2 \widetilde{\lambda}\left(x+\frac{\beta}{2^{\frac{m}{2}+2k}}+\frac{\gamma}{2^{k+m}}-\frac{q^2}{2^{m+2k}}\right)}{2^{\frac{m}{2}}} t} dt \Bigg|^2 dx \nonumber \\
&&\lesssim 2^{\frac{m}{2}+2k+\del k} \calA \calB,
\end{eqnarray}
where 
\begin{eqnarray*}
&& \calA^2:=\sum_{q \sim 2^{\frac{m}{2}}} \sum_{\beta=-2^{\frac{m}{2}}}^{2^{\frac{m}{2}}} \sum_{\gamma=-2^{\frac{m}{2}-k}}^{2^{\frac{m}{2}-k}} \iint_{\left(I_0^{m, k} \right)^2} \left| f_1 \left(\frac{\beta}{2^{\frac{m}{2}+2k}}+\frac{\gamma}{2^{k+m}}-\frac{q}{2^{\frac{m}{2}+k}}-t-\frac{q^2}{2^{m+2k}} \right) \right|^2 \\
&& \qquad \qquad \qquad \qquad  \qquad \cdot  \left| f_1 \left(\frac{\beta}{2^{\frac{m}{2}+2k}}+\frac{\gamma}{2^{k+m}}-\frac{q}{2^{\frac{m}{2}+k}}-s-\frac{q^2}{2^{m+2k}} \right) \right|^2 dtds,
\end{eqnarray*}
and
\begin{eqnarray*}
&& \calB^2:=\sum_{q \sim 2^{\frac{m}{2}}} \sum_{\beta=-2^{\frac{m}{2}}}^{2^{\frac{m}{2}}} \sum_{\gamma=-2^{\frac{m}{2}-k}}^{2^{\frac{m}{2}-k}} \iint_{\left(I_0^{m, k} \right)^2} \Bigg | \int_{3I_0^{0, k+m}}  f_2 \left(x+\frac{\beta}{2^{\frac{m}{2}+2k}}+\frac{\gamma}{2^{k+m}}+\frac{2tq}{2^{\frac{m}{2}+k}} \right) \\
&& \qquad \cdot  f_2 \left(x+\frac{\beta}{2^{\frac{m}{2}+2k}}+\frac{\gamma}{2^{k+m}}+\frac{2sq}{2^{\frac{m}{2}+k}} \right) e^{3i 2^k \frac{q^2 \widetilde{\lambda}\left(x+\frac{\beta}{2^{\frac{m}{2}+2k}}+\frac{\gamma}{2^{k+m}}-\frac{q^2}{2^{m+2k}}\right)}{2^{\frac{m}{2}}} (t-s)} dx \Bigg |^2 dtds. 
\end{eqnarray*}
We estimate $\calA$ first. Recall that since $f_1$ is uniform, we have $\left\|f_1 \right\|_{L^\infty(3I^k)} \lesssim 2^{\mu k}2^{\frac{k}{2}} \left\|f_1 \right\|_{L^2(3I^k)}$, and hence 
$$
\calA^2 \lesssim 2^{\frac{m}{2}} \cdot 2^{m-k} \cdot 2^{-m-2k} \cdot 2^{4\mu k} \cdot 2^{2k} \cdot \left\|f_1 \right\|^4_{L^2(3I^k)}=2^{\frac{m}{2}-k} \cdot 2^{4\mu k} \left\|f_1 \right\|_{L^2(3I^k)}^4.
$$
Thus, by \eqref{20241006eq30}, one has
\begin{equation} \label{20241008eq43}
{\bf V}_{m, k}^{\widetilde{p}, \calL}(f_1, f_2)\lesssim 2^{\frac{3m}{4}+\frac{3k}{2}+(\del+2\mu)k} \left\|f_1 \right\|_{L^2(3I^k)}^2 \calB.
\end{equation}

\medskip 

\noindent\textsf{Step III: Estimate of $\calB$.} We start by applying the change of variables $t \to \frac{2^{\frac{m}{2}}}{q}t$ and $s \to \frac{2^{\frac{m}{2}}}{q}s$ and then open the square in order to deduce that (with shifting $f_1, f_2$, and $\widetilde{\lambda}$ properly)
\begin{eqnarray*}
&& \calB^2 \lesssim \sum_{q \sim 2^{\frac{m}{2}}} \sum_{\beta \sim 2^{\frac{m}{2}}} \sum_{\gamma \sim 2^{\frac{m}{2}-k}} \iint_{\left(2I_0^{m, k} \right)^2} \iint_{\left(3I_0^{0, k+m} \right)^2} \\
&& \qquad f_2 \left(x+\frac{\beta}{2^{\frac{m}{2}+2k}}+\frac{\gamma}{2^{k+m}}+\frac{2t}{2^k} \right) f_2 \left(x+\frac{\beta}{2^{\frac{m}{2}+2k}}+\frac{\gamma}{2^{k+m}}+\frac{2s}{2^k} \right) \\
&& \qquad \cdot  f_2 \left(y+\frac{\beta}{2^{\frac{m}{2}+2k}}+\frac{\gamma}{2^{k+m}}+\frac{2t}{2^k} \right)  f_2 \left(y+\frac{\beta}{2^{\frac{m}{2}+2k}}+\frac{\gamma}{2^{k+m}}+\frac{2t}{2^k} \right) \\
&& \qquad \cdot e^{3i 2^k q \left[\widetilde{\lambda}\left(x+\frac{\beta}{2^{\frac{m}{2}+2k}}+\frac{\gamma}{2^{k+m}}-\frac{q^2}{2^{m+2k}}\right)-\widetilde{\lambda}\left(y+\frac{\beta}{2^{\frac{m}{2}+2k}}+\frac{\gamma}{2^{k+m}}-\frac{q^2}{2^{m+2k}}\right)  \right](t-s)} dxdydtds. 
\end{eqnarray*}

For technical reasoning, taking $\epsilon_3>0$ sufficiently small, we decompose the regime of $\gamma$ as follows:
$$
\left[2^{\frac{m}{2}-k}, 2^{\frac{m}{2}-k+1} \right]=\bigcup_{\widetilde{\gamma} \sim 2^{\epsilon_3 k}} \left[2^{\frac{m}{2}-k-\epsilon_3 k} \widetilde{\gamma}, 2^{\frac{m}{2}-k-\epsilon_3 k}(\widetilde{\gamma}+1) \right].
$$
Using this (by properly shifting the functions again), we have 
\begin{eqnarray*}
&& \calB^2 \lesssim \sum_{q \sim 2^{\frac{m}{2}}} \sum_{\beta \sim 2^{\frac{m}{2}}} \sum_{\widetilde{\gamma} \sim 2^{\epsilon_3 k} } \sum_{\gamma' \sim 2^{\frac{m}{2}-k-\epsilon_3 k}} \iint_{\left(2I_0^{m, k} \right)^2} \iint_{\left(3I_0^{0, k+m} \right)^2} \\
&& \quad f_2 \left(x+\frac{\beta}{2^{\frac{m}{2}+2k}}+\frac{\widetilde{\gamma}}{2^{\frac{m}{2}+2k+\epsilon_3 k}}+\frac{\gamma'}{2^{m+k}}+\frac{2t}{2^k} \right) f_2 \left(x+\frac{\beta}{2^{\frac{m}{2}+2k}}+\frac{\widetilde{\gamma}}{2^{\frac{m}{2}+2k+\epsilon_3 k}}+\frac{\gamma'}{2^{m+k}}+\frac{2s}{2^k} \right)\\
&& \quad \cdot f_2 \left(y+\frac{\beta}{2^{\frac{m}{2}+2k}}+\frac{\widetilde{\gamma}}{2^{\frac{m}{2}+2k+\epsilon_3 k}}+\frac{\gamma'}{2^{m+k}}+\frac{2t}{2^k} \right)  f_2 \left(y+\frac{\beta}{2^{\frac{m}{2}+2k}}+\frac{\widetilde{\gamma}}{2^{\frac{m}{2}+2k+\epsilon_3 k}}+\frac{\gamma'}{2^{m+k}}+\frac{2s}{2^k} \right) \\
&& \quad \cdot e^{3i 2^k q \left[\widetilde{\lambda}\left(x+\frac{\beta}{2^{\frac{m}{2}+2k}}+\frac{\widetilde{\gamma}}{2^{\frac{m}{2}+2k+\epsilon_3 k}}+\frac{\gamma'}{2^{m+k}}-\frac{q^2}{2^{m+2k}}\right)-\widetilde{\lambda}\left(y+\frac{\beta}{2^{\frac{m}{2}+2k}}+\frac{\widetilde{\gamma}}{2^{\frac{m}{2}+2k+\epsilon_3 k}}+\frac{\gamma'}{2^{m+k}}-\frac{q^2}{2^{m+2k}}\right)  \right](t-s)} dxdydtds. 
\end{eqnarray*}
Next, by applying change variables $t \to t-\frac{\gamma'}{2 \cdot 2^m}-\frac{2^k y}{2}$ and $s \to s-\frac{\gamma'}{2 \cdot 2^m}-\frac{2^k y}{2}$, then up to an admissible error of the form (recall that $0 \le k \le \frac{m}{2}-\epsilon_2 m$)
$$
O\left(2^{\delta k}\,(2^{-\frac{\epsilon_3 k}{2}}+2^{-\frac{\epsilon_2 m}{2}})\left\|f_1 \right\|_{L^2(3I^k)}^2 \left\|f_2 \right\|^2_{L^2(3I^k)} \right), 
$$
one deduces that
\begin{eqnarray*}
&& \calB^2 \lesssim \sum_{q \sim 2^{\frac{m}{2}}} \sum_{\beta \sim 2^{\frac{m}{2}}} \sum_{\widetilde{\gamma} \sim 2^{\epsilon_3 k} } \sum_{\gamma' \sim 2^{\frac{m}{2}-k-\epsilon_3 k}} \iint_{\left(2I_0^{m, k} \right)^2} \iint_{\left(3I_0^{0, k+m} \right)^2} dxdydtds\\
&&  \quad f_2 \left(x-y+\frac{\beta}{2^{\frac{m}{2}+2k}}+\frac{\widetilde{\gamma}}{2^{\frac{m}{2}+2k+\epsilon_3 k}}+\frac{2t}{2^k} \right) f_2 \left(x-y+\frac{\beta}{2^{\frac{m}{2}+2k}}+\frac{\widetilde{\gamma}}{2^{\frac{m}{2}+2k+\epsilon_3 k}}+\frac{2s}{2^k} \right)\\
&& \quad \cdot f_2 \left(\frac{\beta}{2^{\frac{m}{2}+2k}}+\frac{\widetilde{\gamma}}{2^{\frac{m}{2}+2k+\epsilon_3 k}}+\frac{2t}{2^k} \right)  f_2 \left(\frac{\beta}{2^{\frac{m}{2}+2k}}+\frac{\widetilde{\gamma}}{2^{\frac{m}{2}+2k+\epsilon_3 k}}+\frac{2s}{2^k} \right) \\
&& \quad \cdot e^{3i 2^k q \left[\widetilde{\lambda}\left(x+\frac{\beta}{2^{\frac{m}{2}+2k}}+\frac{\widetilde{\gamma}}{2^{\frac{m}{2}+2k+\epsilon_3 k}}+\frac{\gamma'}{2^{m+k}}-\frac{q^2}{2^{m+2k}}\right)-\widetilde{\lambda}\left(y+\frac{\beta}{2^{\frac{m}{2}+2k}}+\frac{\widetilde{\gamma}}{2^{\frac{m}{2}+2k+\epsilon_3 k}}+\frac{\gamma'}{2^{m+k}}-\frac{q^2}{2^{m+2k}}\right)  \right](t-s)}.
\end{eqnarray*}
For technical reasoning again, we need to further decompose the range of $y$ into smaller intervals: let $\epsilon_4>0$ be sufficiently small and write
$$
3I_0^{0, k+m}=\left[0, \frac{3}{2^{k+m}}\right]=\bigcup_{\iota=0}^{3 2^{\epsilon_4 k}-1} \left[\frac{\iota}{2^{k+m+\epsilon_4 k}}, \frac{\iota+1}{2^{k+m+\epsilon_4 k}} \right].
$$
Using the above decomposition (by properly shifting the functions again) followed by the change variable $x \to x+y$, one deduces that 
\begin{eqnarray*}
&& \calB^2 \lesssim  \sum_{\beta \sim 2^{\frac{m}{2}}} \sum_{\widetilde{\gamma} \sim 2^{\epsilon_3 k} } \sum_{\iota \sim 2^{\epsilon_4 k}} \iint_{2I_0^{m, k} \times 2I_0^{m, k}}  \int_{5I_0^{0, k+m}}    \\
&& \Bigg | f_2 \left(x-\frac{\iota}{2^{k+m+\epsilon_4 k}}+\frac{\beta}{2^{\frac{m}{2}+2k}}+\frac{\widetilde{\gamma}}{2^{\frac{m}{2}+2k+\epsilon_3 k}}+\frac{2t}{2^k} \right) f_2 \left(\frac{\beta}{2^{\frac{m}{2}+2k}}+\frac{\widetilde{\gamma}}{2^{\frac{m}{2}+2k+\epsilon_3 k}}+\frac{2t}{2^k} \right)   \\
&&  \cdot   f_2 \left(x-\frac{\iota}{2^{k+m+\epsilon_4 k}}+\frac{\beta}{2^{\frac{m}{2}+2k}}+\frac{\widetilde{\gamma}}{2^{\frac{m}{2}+2k+\epsilon_3 k}}+\frac{2s}{2^k} \right) f_2 \left(\frac{\beta}{2^{\frac{m}{2}+2k}}+\frac{\widetilde{\gamma}}{2^{\frac{m}{2}+2k+\epsilon_3 k}}+\frac{2s}{2^k} \right) \Bigg | \\
&& \cdot \left| \sum_{q \sim 2^{\frac{m}{2}}} \sum_{\gamma' \sim 2^{\frac{m}{2}-k-\epsilon_3k }} \int_{I_0^{0, k+m+\epsilon_4 k}} e^{3i 2^k q \Bigg[ \substack{\widetilde{\lambda}\left(x+y+\frac{\beta}{2^{\frac{m}{2}+2k}}+\frac{\widetilde{\gamma}}{2^{\frac{m}{2}+2k+\epsilon_3 k}}+\frac{\gamma'}{2^{m+k}}-\frac{q^2}{2^{m+2k}}\right) \qquad \qquad \quad     \\
\qquad -\widetilde{\lambda}\left(y+\frac{\iota}{2^{k+m+\epsilon_4 k}}+\frac{\beta}{2^{\frac{m}{2}+2k}}+\frac{\widetilde{\gamma}}{2^{\frac{m}{2}+2k+\epsilon_3 k}}+\frac{\gamma'}{2^{m+k}}-\frac{q^2}{2^{m+2k}}\right)} \Bigg](t-s)} dy \right|  dxdtds 
\end{eqnarray*}
Now, by another change of variable $t \to t+s$ and an application of Cauchy-Schwarz, we derive that
\begin{equation} \label{20241008eq01}
\calB^2 \lesssim \calC \calD, 
\end{equation}
where
\begin{eqnarray*}
&& \calC^2:=\sum_{\beta \sim 2^{\frac{m}{2}}} \sum_{\widetilde{\gamma} \sim 2^{\epsilon_3 k} } \sum_{\iota \sim 2^{\epsilon_4 k}} \iint_{4I_0^{m, k} \times 2I_0^{m, k}}  \int_{5I_0^{0, k+m}} \\
&& \Bigg | f_2 \left(x-\frac{\iota}{2^{k+m+\epsilon_4 k}}+\frac{\beta}{2^{\frac{m}{2}+2k}}+\frac{\widetilde{\gamma}}{2^{\frac{m}{2}+2k+\epsilon_3 k}}+\frac{2(t+s)}{2^k} \right) f_2 \left(\frac{\beta}{2^{\frac{m}{2}+2k}}+\frac{\widetilde{\gamma}}{2^{\frac{m}{2}+2k+\epsilon_3 k}}+\frac{2(t+s)}{2^k} \right)   \\
&&  \cdot   f_2 \left(x-\frac{\iota}{2^{k+m+\epsilon_4 k}}+\frac{\beta}{2^{\frac{m}{2}+2k}}+\frac{\widetilde{\gamma}}{2^{\frac{m}{2}+2k+\epsilon_3 k}}+\frac{2s}{2^k} \right) f_2 \left(\frac{\beta}{2^{\frac{m}{2}+2k}}+\frac{\widetilde{\gamma}}{2^{\frac{m}{2}+2k+\epsilon_3 k}}+\frac{2s}{2^k} \right) \Bigg |^2 dx dtds
\end{eqnarray*}
and
\begin{eqnarray*}
&& \calD^2:=|I_0^{m, k}|\sum_{\beta \sim 2^{\frac{m}{2}}} \sum_{\widetilde{\gamma} \sim 2^{\epsilon_3 k} } \sum_{\iota \sim 2^{\epsilon_4 k}} \int_{4I_0^{m, k} }  \int_{5I_0^{0, k+m}} \\
&&  \left| \sum_{q \sim 2^{\frac{m}{2}}} \sum_{\gamma' \sim 2^{\frac{m}{2}-k-\epsilon_3k }} \int_{I_0^{0, k+m+\epsilon_4 k}} e^{3i 2^k q \left[ \substack{\widetilde{\lambda}\left(x+y+\frac{\beta}{2^{\frac{m}{2}+2k}}+\frac{\widetilde{\gamma}}{2^{\frac{m}{2}+2k+\epsilon_3 k}}+\frac{\gamma'}{2^{m+k}}-\frac{q^2}{2^{m+2k}}\right) \qquad \qquad \quad     \\
\qquad -\widetilde{\lambda}\left(y+\frac{\iota}{2^{k+m+\epsilon_4 k}}+\frac{\beta}{2^{\frac{m}{2}+2k}}+\frac{\widetilde{\gamma}}{2^{\frac{m}{2}+2k+\epsilon_3 k}}+\frac{\gamma'}{2^{m+k}}-\frac{q^2}{2^{m+2k}}\right)} \right]t} dy \right|^2  dxdt. 
\end{eqnarray*}
To this end, we first estimate $\calC$, which is routine since $f_2$ is uniform. More precisely, from \eqref{uniformest}, we have 
\begin{eqnarray*}
\calC^2%
& \lesssim & 2^{\frac{m}{2}} \cdot 2^{\epsilon_3 k} \cdot 2^{\epsilon_4 k} 2^{-m-2k} \cdot 2^{-m-k} \cdot 2^{8\mu k} \cdot 2^{4k} \left\|f_2 \right\|_{L^2(3I^k)}^8 \\
&=& 2^{-\frac{3m}{2}} \cdot 2^k \cdot 2^{\left(8\mu+\epsilon_3+\epsilon_4 \right)k} \left\|f_2 \right\|_{L^2(3I^k)}^8.
\end{eqnarray*}
Thus, using the above with \eqref{20241008eq43} and \eqref{20241008eq01}, we have 
\begin{eqnarray} \label{20241009eq01}
{\bf V}_{m, k}^{\widetilde{p}, \calL}(f_1, f_2)%
&\lesssim& 2^{\frac{3m}{4}+\frac{3k}{2}+(\del+2\mu)k} \left\|f_1 \right\|_{L^2(3I^k)}^2 \cdot 2^{-\frac{3m}{8}} \cdot 2^{\frac{k}{4}} \cdot 2^{\left(2\mu+\frac{\epsilon_3}{4}+\frac{\epsilon_4}{4}\right)k} \left\|f_2 \right\|_{L^2(3I^k)}^2 \calD^{\frac{1}{2}} \nonumber \\
&\lesssim& 2^{\frac{3m}{8}} \cdot 2^{\frac{7k}{4}} \cdot 2^{\left(4\mu+\del+\frac{\epsilon_3+\epsilon_4}{4} \right)k} \left\|f_1\right\|_{L^2(3I^k)}^2 \left\|f_2\right\|_{L^2(3I^k)}^2 \calD^{\frac{1}{2}}.  
\end{eqnarray}

\medskip 

\noindent\textsf{Step IV: The estimate of $\calD$.} We start by opening the square, Fubini, and integrating the $t$-integral, to deduce 
$$
\calD^2 \lesssim \left|I_0^{m, k} \right|^2 \sum_{\beta \sim 2^{\frac{m}{2}}} \sum_{\substack{\widetilde{\gamma} \sim 2^{\epsilon_3 k} \\ \iota \sim 2^{\epsilon_4 k}}} \sum_{q, q_1 \sim 2^{\frac{m}{2}}} \sum_{\gamma', \gamma_1' \sim 2^{\frac{m}{2}-k-\epsilon_3 k}} \int_{5I_0^{0, k+m}} \iint_{\left(I_0^{0, k+m+\epsilon_4 k} \right)^2} \Xi_{x, y, y_1, \beta, \iota}^{\widetilde{\gamma}, \gamma', \gamma'_1, 1}(q, q_1) dydy_1dx,
$$
where\footnote{Recall our notation $\llbracket x \rrbracket:=\frac{1}{1+|x|}$.} 
\begin{eqnarray*}
&& \Xi_{x, y, y_1, \beta, \iota}^{\widetilde{\gamma}, \gamma', \gamma'_1, 1}(q, q_1):=\vast \llbracket \left[\substack{\widetilde{\lambda}\left(x+y+\frac{\beta}{2^{\frac{m}{2}+2k}}+\frac{\widetilde{\gamma}}{2^{\frac{m}{2}+2k+\epsilon_3 k}}+\frac{\gamma'}{2^{m+k}}-\frac{q^2}{2^{m+2k}}\right) \qquad \qquad \quad     \\
\qquad -\widetilde{\lambda}\left(y+\frac{\iota}{2^{k+m+\epsilon_4 k}}+\frac{\beta}{2^{\frac{m}{2}+2k}}+\frac{\widetilde{\gamma}}{2^{\frac{m}{2}+2k+\epsilon_3 k}}+\frac{\gamma'}{2^{m+k}}-\frac{q^2}{2^{m+2k}}\right)} \right] \\
&& \qquad \qquad \qquad \qquad \qquad \qquad \qquad -\frac{q_1}{q} \left[\substack{\widetilde{\lambda}\left(x+y_1+\frac{\beta}{2^{\frac{m}{2}+2k}}+\frac{\widetilde{\gamma}}{2^{\frac{m}{2}+2k+\epsilon_3 k}}+\frac{\gamma_1'}{2^{m+k}}-\frac{q_1^2}{2^{m+2k}}\right) \qquad \qquad \quad     \\
\qquad -\widetilde{\lambda}\left(y_1+\frac{\iota}{2^{k+m+\epsilon_4 k}}+\frac{\beta}{2^{\frac{m}{2}+2k}}+\frac{\widetilde{\gamma}}{2^{\frac{m}{2}+2k+\epsilon_3 k}}+\frac{\gamma_1'}{2^{m+k}}-\frac{q_1^2}{2^{m+2k}}\right)} \right] \vast \rrbracket.
\end{eqnarray*}
Next,  we will glue the $y$-intervals with smaller length $2^{-m-k-\epsilon_3 k}$ in order to form intervals of length $2^{-\frac{m}{2}-2k}$. In order to do so we make use of $\Xi_{x, y, y_1, \beta, \iota}^{\widetilde{\gamma}, \gamma', \gamma'_1, 1}(q, q_1)>0$. With this we apply the change of variable $x \to x+\frac{\iota}{2^{k+m+\epsilon_4 k}}$ and notice that 
\begin{equation} \label{20241010eq01}
\calD^2 \lesssim \left|I_0^{m, k} \right|^2 \sum_{\beta \sim 2^{\frac{m}{2}}} \sum_{\substack{\widetilde{\gamma} \sim 2^{\epsilon_3 k} \\ \iota \sim 2^{\epsilon_4 k}}} \sum_{q, q_1 \sim 2^{\frac{m}{2}}} \sum_{\gamma', \gamma_1' \sim 2^{\frac{m}{2}-k-\epsilon_3 k}} \int_{10 I_0^{0, k+m}} \iint_{\left(I_0^{0, k+m+\epsilon_4 k} \right)^2} \Xi_{x, y, y_1, \beta, \iota}^{\widetilde{\gamma}, \gamma', \gamma'_1, 2}(q, q_1) dydy_1dx,
\end{equation} 
where 
\begin{eqnarray*}
&& \Xi_{x, y, y_1, \beta, \iota}^{\widetilde{\gamma}, \gamma', \gamma'_1, 2}(q, q_1):=\vast \llbracket \left[\substack{\widetilde{\lambda}\left(x+y+\frac{\iota}{2^{k+m+\epsilon_4 k}}+\frac{\beta}{2^{\frac{m}{2}+2k}}+\frac{\widetilde{\gamma}}{2^{\frac{m}{2}+2k+\epsilon_3 k}}+\frac{\gamma'}{2^{m+k}}-\frac{q^2}{2^{m+2k}}\right) \qquad \\
\qquad -\widetilde{\lambda}\left(y+\frac{\iota}{2^{k+m+\epsilon_4 k}}+\frac{\beta}{2^{\frac{m}{2}+2k}}+\frac{\widetilde{\gamma}}{2^{\frac{m}{2}+2k+\epsilon_3 k}}+\frac{\gamma'}{2^{m+k}}-\frac{q^2}{2^{m+2k}}\right)} \right] \\
&& \qquad \qquad \qquad \qquad \qquad \qquad -\frac{q_1}{q} \left[\substack{\widetilde{\lambda}\left(x+y_1+\frac{\iota}{2^{k+m+\epsilon_4 k}}+\frac{\beta}{2^{\frac{m}{2}+2k}}+\frac{\widetilde{\gamma}}{2^{\frac{m}{2}+2k+\epsilon_3 k}}+\frac{\gamma_1'}{2^{m+k}}-\frac{q_1^2}{2^{m+2k}}\right) \qquad \\
\qquad -\widetilde{\lambda}\left(y_1+\frac{\iota}{2^{k+m+\epsilon_4 k}}+\frac{\beta}{2^{\frac{m}{2}+2k}}+\frac{\widetilde{\gamma}}{2^{\frac{m}{2}+2k+\epsilon_3 k}}+\frac{\gamma_1'}{2^{m+k}}-\frac{q_1^2}{2^{m+2k}}\right)} \right] \vast \rrbracket.
\end{eqnarray*}
Now the summation in \eqref{20241010eq01} can be treated as a diagonal sum with respect to the variables $\widetilde{\gamma}$ and $\iota$. Using again the fact that $\Xi_{x, y, y_1, \beta, \iota}^{\widetilde{\gamma}, \gamma', \gamma'_1, 2}(q, q_1)>0$, we derive that
\begin{equation} \label{20241010eq02}
\calD^2 \lesssim \left|I_0^{m, k} \right|^2 \sum_{\beta \sim 2^{\frac{m}{2}}} \sum_{\substack{\widetilde{\gamma}, \widetilde{\gamma}_1 \sim 2^{\epsilon_3 k} \\ \iota, \iota_1 \sim 2^{\epsilon_4 k}}} \sum_{q, q_1 \sim 2^{\frac{m}{2}}} \sum_{\gamma', \gamma_1' \sim 2^{\frac{m}{2}-k-\epsilon_3 k}} \int_{10 I_0^{0, k+m}} \iint_{\left(I_0^{0, k+m+\epsilon_4 k} \right)^2} \Xi_{x, y, y_1, \beta, \iota, \iota_1}^{\widetilde{\gamma}, \widetilde{\gamma}_1, \gamma', \gamma'_1}(q, q_1) dydy_1dx,
\end{equation} 
where 
\begin{eqnarray*}
&& \Xi_{x, y, y_1, \beta, \iota, \iota_1}^{\widetilde{\gamma}, \widetilde{\gamma}_1, \gamma', \gamma'_1}(q, q_1):=\vast \llbracket \left[\substack{\widetilde{\lambda}\left(x+y+\frac{\iota}{2^{k+m+\epsilon_4 k}}+\frac{\beta}{2^{\frac{m}{2}+2k}}+\frac{\widetilde{\gamma}}{2^{\frac{m}{2}+2k+\epsilon_3 k}}+\frac{\gamma'}{2^{m+k}}-\frac{q^2}{2^{m+2k}}\right) \qquad \\
\qquad -\widetilde{\lambda}\left(y+\frac{\iota}{2^{k+m+\epsilon_4 k}}+\frac{\beta}{2^{\frac{m}{2}+2k}}+\frac{\widetilde{\gamma}}{2^{\frac{m}{2}+2k+\epsilon_3 k}}+\frac{\gamma'}{2^{m+k}}-\frac{q^2}{2^{m+2k}}\right]} \right) \\
&& \qquad \qquad \qquad \qquad \qquad \qquad -\frac{q_1}{q} \left[\substack{\widetilde{\lambda}\left(x+y_1+\frac{\iota_1}{2^{k+m+\epsilon_4 k}}+\frac{\beta}{2^{\frac{m}{2}+2k}}+\frac{\widetilde{\gamma}_1}{2^{\frac{m}{2}+2k+\epsilon_3 k}}+\frac{\gamma_1'}{2^{m+k}}-\frac{q_1^2}{2^{m+2k}}\right) \qquad \\
\qquad -\widetilde{\lambda}\left(y_1+\frac{\iota_1}{2^{k+m+\epsilon_4 k}}+\frac{\beta}{2^{\frac{m}{2}+2k}}+\frac{\widetilde{\gamma}_1}{2^{\frac{m}{2}+2k+\epsilon_3 k}}+\frac{\gamma_1'}{2^{m+k}}-\frac{q_1^2}{2^{m+2k}}\right)} \right] \vast \rrbracket.
\end{eqnarray*}
Summing up the $y, \iota, \gamma'$, and $\widetilde{\gamma}$ variables, as well as the $y_1, \iota_1, \gamma'_1$, and $\widetilde{\gamma}_1$ variables, respectively, we have
\begin{equation} \label{20241010eq20}
\calD^2 \lesssim \left|I_0^{m, k} \right|^2 \sum_{\beta \sim 2^{\frac{m}{2}}} \sum_{q, q_1 \sim 2^{\frac{m}{2}}} \int_{10 I_0^{0, k+m}} \iint_{(I_0^{m, 2k})^2} \Xi_{x, y, y_1,  \beta}(q, q_1) dydy_1dx, 
\end{equation} 
where 
\begin{eqnarray*}
&& \Xi_{x, y, y_1, \beta}(q, q_1):=\Bigg \llbracket  \left[ \widetilde{\lambda}\left(x+y+\frac{\beta}{2^{\frac{m}{2}+2k}}-\frac{q^2}{2^{m+2k}} \right)-\widetilde{\lambda}\left(y+\frac{\beta}{2^{\frac{m}{2}+2k}}-\frac{q^2}{2^{m+2k}} \right) \right] \\
&& \qquad \qquad \qquad \qquad \qquad \qquad-\frac{q_1}{q} \left[ \widetilde{\lambda}\left(x+y_1+\frac{\beta}{2^{\frac{m}{2}+2k}}-\frac{q_1^2}{2^{m+2k}} \right)-\widetilde{\lambda}\left(y_1+\frac{\beta}{2^{\frac{m}{2}+2k}}-\frac{q_1^2}{2^{m+2k}} \right) \right] \Bigg \rrbracket.
\end{eqnarray*} 
Now, in order to handle \eqref{20241010eq20}, we apply the change of variables $\beta \to \beta+\left\lfloor \frac{q^2}{2^{\frac{m}{2}}} \right \rfloor$ and then $y \to y+\frac{\left\{ \frac{q^2}{2^{\frac{m}{2}}} \right\}}{2^{\frac{m}{2}+2k}}$ and $y_1 \to y_1+\frac{\left\{ \frac{q^2}{2^{\frac{m}{2}}} \right\}}{2^{\frac{m}{2}+2k}}$ to deduce that 
\begin{eqnarray} \label{20240110eq21}
&&\calD^2 \lesssim \left|I_0^{m, k} \right|^2 \sum_{\beta \sim 2^{\frac{m}{2}}} \sum_{q, q_1 \sim 2^{\frac{m}{2}}} \int_{10 I_0^{0, m+k}} \iint_{(2I_0^{m, 2k})^2} \bigg \llbracket \left[\widetilde{\lambda} \left(x+y+\frac{\beta}{2^{\frac{m}{2}+2k}} \right)-\widetilde{\lambda} \left(y+\frac{\beta}{2^{\frac{m}{2}+2k}} \right) \right] \nonumber \\
&& \quad -\frac{q_1}{q} \left[ \widetilde{\lambda} \left(x+y_1+\frac{\beta}{2^{\frac{m}{2}+2k}}-\frac{q_1^2-q^2}{2^{m+2k}} \right)- \widetilde{\lambda} \left(y_1+\frac{\beta}{2^{\frac{m}{2}+2k}}-\frac{q_1^2-q^2}{2^{m+2k}} \right) \right] \bigg \rrbracket dydy_1 dx.
\end{eqnarray}
Then, letting $u=\left\lfloor \frac{q_1^2-q^2}{2^{\frac{m}{2}}} \right \rfloor$, $y_1 \to y_1+\frac{\left\{\frac{q_1^2-q^2}{2^{\frac{m}{2}}} \right\}}{2^{\frac{m}{2}+2k}}$, and using that $\frac{q_1}{q}=\sqrt{1+\frac{2^{\frac{m}{2}}u}{q^2}+O \left(\frac{1}{2^{\frac{m}{2}}} \right)}$, we reduce \eqref{20240110eq21} to the following estimate:
\begin{eqnarray} \label{20241011eq01}
&&\calD^2 \lesssim  \left|I_0^{m, k} \right|^2 \sum_{\beta \sim 2^{\frac{m}{2}}} \sum_{q, u \sim 2^{\frac{m}{2}}} \int_{10 I_0^{0, m+k}} \iint_{4I_0^{m , 2k} \times 2I_0^{m, 2k}} \bigg \llbracket \left[\widetilde{\lambda} \left(x+y+\frac{\beta}{2^{\frac{m}{2}+2k}} \right)-\widetilde{\lambda} \left(y+\frac{\beta}{2^{\frac{m}{2}+2k}} \right) \right] \nonumber \\
&& \quad  - \frac{\sqrt{2^{\frac{m}{2}}u+q^2}}{q} \left[ \widetilde{\lambda}\left(x+y_1+\frac{\beta}{2^{\frac{m}{2}+2k}}-\frac{u}{2^{\frac{m}{2}+2k}} \right)-\widetilde{\lambda}\left(y_1+\frac{\beta}{2^{\frac{m}{2}+2k}}-\frac{u}{2^{\frac{m}{2}+2k}} \right) \right] \Bigg \rrbracket dydy_1dx. 
\end{eqnarray}
The treatment of \eqref{20241011eq01} follows now the same reasonings as in \cite[\textsf{Steps IV} and \textsf{V}, Proposition 4.7]{HL23}. For reader's convenience we include below a sketch of the argument: firstly, for notational convenience we set $\widetilde{\lambda}_{\beta}(\cdot)=\widetilde{\lambda} \left(\cdot+\frac{\beta}{2^{\frac{m}{2}+2k}} \right)$ and reduce \eqref{20241011eq01} to
\begin{eqnarray*} 
&&\calD^2 \lesssim  \left|I_0^{m, k} \right|^2 \sum_{\beta \sim 2^{\frac{m}{2}}} \sum_{q, u \sim 2^{\frac{m}{2}}} \int_{10 I_0^{0, m+k}} \iint_{\left(4I_0^{m , 2k} \right)^2} \bigg \llbracket \left[\widetilde{\lambda}_{\beta} \left(x+y\right)-\widetilde{\lambda}_{\beta} (y) \right] \nonumber \\
&& \qquad \qquad \qquad   - \frac{\sqrt{2^{\frac{m}{2}}u+q^2}}{q} \left[ \widetilde{\lambda}_{\beta}\left(x+y_1-\frac{u}{2^{\frac{m}{2}+2k}} \right)-\widetilde{\lambda}_{\beta}\left(y_1-\frac{u}{2^{\frac{m}{2}+2k}} \right) \right] \Bigg \rrbracket dydy_1dx. 
\end{eqnarray*}
Secondly, we further decompose the $y$ and $y_1$ intervals to smaller intervals of length $2^{-k-m}$, which gives
\begin{eqnarray} \label{20241011eq30}
&&\calD^2 \lesssim \left|I_0^{m, k} \right|^2 \sum_{\beta \sim 2^{\frac{m}{2}}}  \sum_{q, u \sim 2^{\frac{m}{2}}} \sum_{\gamma, \gamma_1 \sim 2^{\frac{m}{2}-k}} \int_{10 I_0^{0, m+k}} \iint_{\left(I_0^{0 , m+k} \right)^2} \bigg \llbracket \left[\widetilde{\lambda}_{\beta,\gamma} \left(x+y\right)-\widetilde{\lambda}_{\beta,\gamma} (y) \right] \nonumber \\
&& \qquad \qquad \qquad   - \frac{\sqrt{2^{\frac{m}{2}}u+q^2}}{q} \left[ \widetilde{\lambda}_{\beta,\gamma_1}\left(x+y_1-\frac{u}{2^{\frac{m}{2}+2k}} \right)-\widetilde{\lambda}_{\beta,\gamma_1}\left(y_1-\frac{u}{2^{\frac{m}{2}+2k}} \right) \right] \Bigg \rrbracket dydy_1dx, 
\end{eqnarray}
where $\widetilde{\lambda}_{\beta,\gamma}(\cdot)=\widetilde{\lambda}_{\beta} \left(\cdot+\frac{\gamma}{2^{m+k}} \right)$. Next, we apply the change of variables $y_1 \to y_1+y$, $x \to x-y$ to derive 
\begin{eqnarray} \label{20241011eq10}
&&\calD^2 \lesssim  \left|I_0^{m, k} \right|^2 \sum_{\beta \sim 2^{\frac{m}{2}}} \sum_{q, u \sim 2^{\frac{m}{2}}} \sum_{\gamma, \gamma_1 \sim 2^{\frac{m}{2}-k}} \iint_{\left(20 I_0^{0, m+k}\right)^2} \bigg\{ \int_{2I_0^{0 , m+k}} \bigg \llbracket \left[\widetilde{\lambda}_{\beta,\gamma} \left(x\right)-\widetilde{\lambda}_{\beta,\gamma} (y) \right] \nonumber \\
&& \qquad \qquad  - \frac{\sqrt{2^{\frac{m}{2}}u+q^2}}{q} \left[ \widetilde{\lambda}_{\beta,\gamma_1}\left(x+y_1-\frac{u}{2^{\frac{m}{2}+2k}} \right)-\widetilde{\lambda}_{\beta,\gamma_1}\left(y+y_1-\frac{u}{2^{\frac{m}{2}+2k}} \right) \right] \Bigg \rrbracket dy_1 \bigg\} dydx. 
\end{eqnarray}
Fix momentarily $\beta$, $\gamma$ and recall that $\widetilde{\lambda}_{\beta,\gamma}  (\cdot)$ is \emph{light} on each interval of the length $2^{-m-k}$ due to our assumptions - see \eqref{20241011eq40}). Next, we decompose both $x$ and $y$ intervals in the above estimate as follows
$$
20I_0^{0, k+m}=\bigcup_{l \in \widetilde{\calL}_{\beta,\gamma}} \widetilde{\frakM_{l, \beta, \gamma}} \left(20I_0^{0, k+m} \right),
$$
where 
$$
\widetilde{\frakM_{l,\beta, \gamma}} \left(20I_0^{0, k+m} \right):=\left\{x \in 20I_0^{0, k+m}: \widetilde{\lambda}_{\beta,\gamma}(x) \in \left[l2^{\epsilon_0 k}, (l+1)2^{\epsilon_0 k} \right] \right\},
$$
and 
$$
\widetilde{\calL}_{\beta,\gamma}:=\left\{l \sim 2^{\frac{m}{2}-\epsilon_0 k}: \left| \widetilde{\frakM_{l,\beta, \gamma}} \left(20I_0^{0, k+m} \right) \right| \lesssim \left|I_0^{0, k+m} \right| 2^{-\epsilon_0 k} \right\}.
$$
With these, one can write \eqref{20241011eq10} as 
$$
\calD^2 \lesssim \calD_{\textnormal{diag}}^2+\calD_{\textnormal{off}}^2, 
$$
where 
\begin{eqnarray*}
&&\calD_{\textnormal{diag}}^2:= \left|I_0^{m, k} \right|^2 \sum_{\beta \sim 2^{\frac{m}{2}}} \sum_{q, u \sim 2^{\frac{m}{2}}} \sum_{\gamma, \gamma_1 \sim 2^{\frac{m}{2}-k}} \sum_{\substack{l, l_1 \in \widetilde{\calL}_{\beta,\gamma} \\ |l-l_1| \le 10}} \iint_{\widetilde{\frakM_{l,\beta,\gamma}} \left(20I_0^{0, k+m} \right) \times \widetilde{\frakM_{l_1,\beta,\gamma}} \left(20I_0^{0, k+m} \right)} \nonumber \\
&& \quad \left\{ \int_{2I_0^{0 , m+k}} \bigg \llbracket \left[\widetilde{\lambda}_{\beta,\gamma} \left(x\right)-\widetilde{\lambda}_{\beta,\gamma} (y) \right] - \frac{\sqrt{2^{\frac{m}{2}}u+q^2}}{q} \left[ \widetilde{\lambda}_{\beta,\gamma_1}\left(x+y_1-\frac{u}{2^{\frac{m}{2}+2k}} \right)-\widetilde{\lambda}_{\beta,\gamma_1}\left(y+y_1-\frac{u}{2^{\frac{m}{2}+2k}} \right) \right] \Bigg \rrbracket dy_1 \right\} dydx. 
\end{eqnarray*}
and
\begin{eqnarray*}
&&\quad\calD_{\textnormal{off}}^2:=  \left|I_0^{m, k} \right|^2 \sum_{\beta \sim 2^{\frac{m}{2}}} \sum_{q, u \sim 2^{\frac{m}{2}}} \sum_{\gamma, \gamma_1 \sim 2^{\frac{m}{2}-k}} \sum_{\substack{l, l_1 \in \widetilde{\calL}_{\beta,\gamma} \\ |l-l_1|>10}} \iint_{\widetilde{\frakM_{l,\beta, \gamma}} \left(20I_0^{0, k+m} \right) \times \widetilde{\frakM_{l_1, \beta, \gamma}} \left(20I_0^{0, k+m} \right)} \nonumber \\
&&\left\{ \int_{2I_0^{0 , m+k}} \bigg \llbracket \left[\widetilde{\lambda}_{\beta,\gamma} \left(x\right)-\widetilde{\lambda}_{\beta,\gamma} (y) \right] - \frac{\sqrt{2^{\frac{m}{2}}u+q^2}}{q} \left[ \widetilde{\lambda}_{\beta,\gamma_1}\left(x+y_1-\frac{u}{2^{\frac{m}{2}+2k}} \right)-\widetilde{\lambda}_{\beta,\gamma_1}\left(y+y_1-\frac{u}{2^{\frac{m}{2}+2k}} \right) \right] \Bigg \rrbracket dy_1 \right\} dydx. 
\end{eqnarray*}
The analysis now reduces to the estimates of the terms $\calD_{\textnormal{diag}}^2$ and $\calD_{\textnormal{off}}^2$: the estimate of $\calD_{\textnormal{diag}}^2$ yields a decay of the form $2^{\left(\del+4\mu-\frac{\epsilon_0}{4} \right)k} \left\|f_1 \right\|_{L^2(3I^k)}^2 \left\|f_2 \right\|_{L^2(3I^k)}^2$, while in estimating the second term $\calD_{\textnormal{off}}^2$ the key ingredient is the usage of the curvature within the mapping 
$$
\left[2^{\frac{m}{2}}, 2^{\frac{m}{2}+1} \right] \ni q \mapsto  \frac{\sqrt{2^{\frac{m}{2}}u+q^2}}{q} \in \left [\frac{1}{10}, 10 \right]. 
$$
(see, the argument in \cite[(4.60)-(4.66)]{HL23}). We leave the further details to the interested reader. 
\end{proof}

\subsection{Treatment of the uniform component ${\bf L}_{m, k}^{\calU_\mu}(f_1, f_2)$: Case II, Step II.2 in Overview} \label{20250228sec03}

In the previous section, we have employed the constancy propagation method -- a key idea originating in the work \cite{HL23}; specifically, we were able to extend the length of constancy for the measurable function $\widetilde{\lambda}$ from the original scale $2^{-m-2k}$ to the larger scale $2^{-m-k}$. As a consequence, with the language provided by \eqref{20241011eq32}, we are now left with treating the case when all intervals $I_p^{0, m+k}$ are \emph{heavy}. Thus, up to an admissible $2^{2\epsilon_0 k}$--loss (involving some standard extension arguments) we may assume from now on that 
\begin{equation} \label{20241012eq01}
\textnormal{\texttt{$\widetilde{\lambda} (\cdot)$ takes constant values on intervals $\left\{I_\ell^{0, m+k}\right\}_{\ell \sim 2^m}$.}} \tag{{\bf H2}}
\end{equation}
The result below represents the rigorous embodiment of the first part of \ref{KH2}:

\begin{prop}
Let $k, m \in \N$, with $0 \le k \le \frac{m}{2}$, and 
\begin{equation} \label{20241012eq19}
\Lambda_m^k(f_1, f_2, f_3):=2^k\int_{I^k} \int_{I^k} f_1(x-t)f_2(x+t^2)f_3(x)e^{i 2^{\frac{m}{2}+3k} \widetilde{\lambda}(x) t^3} dtdx, 
\end{equation} 
with $\widetilde{\lambda}$ satisfying the assumption \eqref{20241012eq01}. Then there exists an $\epsilon_5>0$, such that 
\begin{equation} \label{20241012eq20}
\left|\Lambda_m^k(f_1, f_2, f_3) \right| \lesssim 2^{-\epsilon_5 k}  \prod_{j=1}^3 \left\|f_j \right\|_{L^{p_j}(3I^k)},
\end{equation} 
whenever $(p_1, p_2, p_3)=(2, 2, \infty), (2, \infty, 2)$ or $(\infty, 2, 2)$. 
\end{prop}

We divide the proof of \eqref{20241012eq20} into four stages. 

\subsubsection{Discretization} We begin by discretizing the model \eqref{20241012eq19} via the Rank-I LGC methodology:

\medskip 

\noindent \textsf{Step I: Linearization.}
\begin{enumerate}
    \item [$\bullet$] for $f_1(x-t)$, we use discretization in spatial intervals of length $2^{-\frac{m}{2}-k}$;
    \item  [$\bullet$] for $f_2(x+t^2)$, we use discretization in spatial intervals of length $2^{-\frac{m}{2}-2k}$. 
\end{enumerate}

\medskip 

\noindent \textsf{Step II: Gabor frame decomposition.}
Guided by the spatial decomposition in \textsf{Step I}, we write
\begin{equation} \label{20240726eq03}
f_j=\sum_{\substack{p_j \sim 2^{\frac{m}{2}+(j-1)k} \\ n_j  \sim 2^{\frac{m}{2}}}} \left \langle f_j, \phi_{p_j, n_j}^{jk} \right \rangle \phi_{p_j, n_j}^{jk}, \quad j=1, 2, 
\end{equation} 
where $\phi_{r, n}^{jk}(x):=2^{\frac{m}{4}+\frac{jk}{2}} \phi \left(2^{\frac{m}{2}+jk}x-r \right) e^{i2^{\frac{m}{2}+jk}xn}$, with $\phi$ being a smooth cut-off function supported away from the origin. 

Plugging \eqref{20240726eq03} to \eqref{20241012eq19}, we see that
\begin{eqnarray*}
&& \Lambda_m^k(f_1, f_2, f_3)=2^k \sum_{\substack{p \sim 2^{\frac{m}{2}+k} \\ q \sim 2^{\frac{m}{2}}}} \sum_{\substack{p_1 \sim 2^{\frac{m}{2}} \\ n_1 \sim 2^{\frac{m}{2}}}} \sum_{\substack{p_2 \sim 2^{\frac{m}{2}+k} \\ n_2 \sim 2^{\frac{m}{2}}}} \left \langle f_1, \phi_{p_1, n_1}^{k} \right \rangle \left \langle f_2, \phi_{p_2, n_2}^{2k} \right \rangle \\
&& \qquad \qquad \cdot \int_{I_p^{m, 2k}} \int_{I_q^{m, k}} \phi_{p_1, n_1}^k(x-t)\phi_{p_2, n_2}^{2k}(x+t^2) f_3(x) e^{i 2^{\frac{m}{2}+3k} \widetilde{\lambda}(x)t^3}dtdx.
\end{eqnarray*}
Note that for $t \in I_q^{m, k}$, one has $2^{\frac{m}{2}+3k}\widetilde{\lambda}(x) t^3=\frac{3q^2t\widetilde{\lambda}(x)}{2^{\frac{m}{2}-k}}-\frac{2q^3\widetilde{\lambda}(x)}{2^m}+O(1)$. With this, we have
\begin{eqnarray*}
&& \Lambda_m^k(f_1, f_2, f_3)=2^k \sum_{\substack{p \sim 2^{\frac{m}{2}+k} \\ q \sim 2^{\frac{m}{2}}}} \sum_{\substack{p_1 \sim 2^{\frac{m}{2}} \\ n_1 \sim 2^{\frac{m}{2}}}} \sum_{\substack{p_2 \sim 2^{\frac{m}{2}+k} \\ n_2 \sim 2^{\frac{m}{2}}}} \left \langle f_1, \phi_{p_1, n_1}^{k} \right \rangle \left \langle f_2, \phi_{p_2, n_2}^{2k} \right \rangle \\
&& \qquad \quad \cdot \int_{I_p^{m, 2k}} \int_{I_q^{m, k}} \phi_{p_1, n_1}^k(x-t)\phi_{p_2, n_2}^{2k}(x+t^2) f_3(x)  e^{-2i \cdot \frac{q^3 \widetilde{\lambda}(x)}{2^m}}  \cdot e^{3i \cdot \frac{q^2 \widetilde{\lambda}(x)}{2^{\frac{m}{2}-k}} t} dtdx.
\end{eqnarray*}
Since $0 \le k \le \frac{m}{2}$, one has $\left| I_p^{m, 2k} \right|=\frac{1}{2^{\frac{m}{2}+2k}} \ge \frac{1}{2^{m+k}}$.
This yields the following decomposition: 
\begin{eqnarray} \label{20250228eq01}
I_p^{m, 2k}%
&=& \left[\frac{p}{2^{\frac{m}{2}+2k}}, \frac{p+1}{2^{\frac{m}{2}+2k}} \right]=\left[\frac{2^{\frac{m}{2}-k}p}{2^{m+k}},  \frac{2^{\frac{m}{2}-k}(p+1)}{2^{m+k}}\right] \nonumber \\
&=& \bigcup_{\widetilde{p}=0}^{2^{\frac{m}{2}-k}-1} \left[ \frac{2^{\frac{m}{2}-k}p+\widetilde{p}}{2^{m+k}}, \frac{2^{\frac{m}{2}-k}p+\widetilde{p}+1}{2^{m+k}} \right]=\bigcup_{\widetilde{p}=0}^{2^{\frac{m}{2}-k}-1} I_{2^{\frac{m}{2}-k}p+\widetilde{p}}^{0, m+k}. 
\end{eqnarray} 
Furthermore, by denoting $\psi_{r, n}^k(x):=2^{\frac{m+k}{2}} \phi\left(2^{m+k}x-r\right) e^{i2^{m+k}xn}$, we have 
\begin{eqnarray*}
&& \Lambda_m^k(f_1, f_2, f_3)=2^k \sum_{\substack{p \sim 2^{\frac{m}{2}+k} \\ q \sim 2^{\frac{m}{2}}}} \sum_{\substack{p_1 \sim 2^{\frac{m}{2}} \\ n_1 \sim 2^{\frac{m}{2}}}} \sum_{\substack{p_2 \sim 2^{\frac{m}{2}+k} \\ n_2 \sim 2^{\frac{m}{2}}}} \left \langle f_1, \phi_{p_1, n_1}^{k} \right \rangle \left \langle f_2, \phi_{p_2, n_2}^{2k} \right \rangle \\
&& \qquad  \cdot \sum_{\widetilde{p}=0}^{2^{\frac{m}{2}-k}-1} \int_{I_{2^{\frac{m}{2}-k}p+\widetilde{p}}^{0, m+k}} \int_{I_q^{m, k}} \phi_{p_1, n_1}^k(x-t)\phi_{p_2, n_2}^{2k}(x+t^2) f_3(x) e^{-2i \cdot \frac{q^3 \widetilde{\lambda}(x)}{2^m}}  \cdot e^{3i \cdot \frac{q^2 \widetilde{\lambda}(x)}{2^{\frac{m}{2}-k}} t} dtdx.
\end{eqnarray*}
With a slight notational abuse, we now write 
$$
\widetilde{\lambda}: \left[2^m, 2^{m+1} \right] \cap \Z \mapsto \left[2^\frac{m}{2}, 2^{\frac{m}{2}+1} \right] \cap \Z , \qquad \textnormal{with} \quad \widetilde{\lambda}(p'):=\widetilde{\lambda} \left(\frac{p'}{2^{m+k}} \right).
$$
Making now use of \eqref{20241012eq01}, we deduce  
\begin{eqnarray*}
&& \Lambda_m^k(f_1, f_2, f_3) \simeq 2^k \sum_{\substack{p \sim 2^{\frac{m}{2}+k} \\ q \sim 2^{\frac{m}{2}}}} \sum_{\substack{p_1 \sim 2^{\frac{m}{2}} \\ n_1 \sim 2^{\frac{m}{2}}}} \sum_{\substack{p_2 \sim 2^{\frac{m}{2}+k} \\ n_2 \sim 2^{\frac{m}{2}}}} \left \langle f_1, \phi_{p_1, n_1}^{k} \right \rangle \left \langle f_2, \phi_{p_2, n_2}^{2k} \right \rangle  \sum_{\widetilde{p}=0}^{2^{\frac{m}{2}-k}-1} e^{-2i \cdot \frac{q^3 \widetilde{\lambda} \left(2^{\frac{m}{2}-k}p+\widetilde{p} \right)}{2^{m}}}  \\
&&\qquad \qquad \qquad \qquad \cdot \int_{I_{2^{\frac{m}{2}-k}p+\widetilde{p}}^{0, m+k}} \int_{I_q^{m, k}} \phi_{p_1, n_1}^k(x-t)\phi_{p_2, n_2}^{2k}(x+t^2) f_3(x)  e^{3i \cdot \frac{q^2 \widetilde{\lambda}\left(2^{\frac{m}{2}-k}p+\widetilde{p} \right)}{2^{\frac{m}{2}-k}} t} dtdx \\
&&  \quad  \simeq 2^{2k} \sum_{\substack{p \sim 2^{\frac{m}{2}+k} \\ q \sim 2^{\frac{m}{2}}}} \sum_{\substack{p_1 \sim 2^{\frac{m}{2}} \\ n_1 \sim 2^{\frac{m}{2}}}} \sum_{\substack{p_2 \sim 2^{\frac{m}{2}+k} \\ n_2 \sim 2^{\frac{m}{2}}}} \left \langle f_1, \phi_{p_1, n_1}^{k} \right \rangle \left \langle f_2, \phi_{p_2, n_2}^{2k} \right \rangle\sum_{\widetilde{p}=0}^{2^{\frac{m}{2}-k}-1}e^{-2i \cdot \frac{q^3 \widetilde{\lambda} \left(2^{\frac{m}{2}-k}p+\widetilde{p} \right)}{2^{m}}}  \\
&&\qquad \qquad \cdot \int_{I_{2^{\frac{m}{2}-k}p+\widetilde{p}}^{0, m+k}} \int_{I_q^{m, k}} 2^{\frac{m+k}{2}} \phi \left(2^{\frac{m}{2}+k}(x-t)-p_1 \right) \phi \left(2^{\frac{m}{2}+2k} \left(x+t^2 \right)-p_2 \right) \\
&& \qquad \qquad \qquad \qquad  \qquad \qquad   \cdot e^{i 2^{\frac{m}{2}+k}(x-t)n_1} e^{i2^{\frac{m}{2}+2k}(x+t^2)n_2} f_3(x)  e^{3i \cdot \frac{q^2 \widetilde{\lambda} \left(2^{\frac{m}{2}-k}p+\widetilde{p} \right)}{2^{\frac{m}{2}-k}} t} dtdx.
\end{eqnarray*}
Since $x \in I_{2^{\frac{m}{2}-k}p+\widetilde{p}}^{0, m+k} \subseteq I_p^{m, 2k}$ and $t \in I_q^{m, k}$, we have 
\begin{equation} 
p_1 \cong \frac{p}{2^k}-q \quad \textrm{and} \quad p_2 \cong p+\frac{q^2}{2^{\frac{m}{2}}}. \tag{\textnormal{\textbf{Time Correlation}}}
\end{equation} 
Hence, 
\begin{eqnarray*}
&& \Lambda_m^k(f_1, f_2, f_3)  \simeq 2^{2k}  \sum_{\substack{p \sim 2^{\frac{m}{2}+k} \\ q \sim 2^{\frac{m}{2}}}} \sum_{n_1, n_2 \sim 2^{\frac{m}{2}}} \left \langle f_1, \phi_{\frac{p}{2^k}-q, n_1}^{k} \right \rangle \left \langle f_2, \phi_{p+\frac{q^2}{2^{\frac{m}{2}}}, n_2}^{2k} \right \rangle \sum_{\widetilde{p}=0}^{2^{\frac{m}{2}-k}-1}  e^{-2i \cdot \frac{q^3 \widetilde{\lambda} \left(2^{\frac{m}{2}-k}p+\widetilde{p} \right)}{2^{m}}}  \\
&& \qquad \quad \cdot \int_{I_{2^{\frac{m}{2}-k}p+\widetilde{p}}^{0, m+k}} 2^{\frac{m+k}{2}}  e^{i\left(2^{\frac{m}{2}+k}n_1+2^{\frac{m}{2}+2k} n_2 \right)x} f_3(x) \cdot \left[ \int_{I_q^{m, k}} e^{it \left(-2^{\frac{m}{2}+k}n_1+ \frac{3q^2 \widetilde{\lambda} \left(2^{\frac{m}{2}-k}p+\widetilde{p} \right)}{2^{\frac{m}{2}-k}} \right)+i2^{\frac{m}{2}+2k}n_2t^2} dt \right] dx
\end{eqnarray*}
Note that for $t \in I_q^{m, k}$, $2^{\frac{m}{2}+2k}n_2t^2 \simeq 2^{\frac{m}{2}+2k} n_2 \left(\frac{2tq}{2^{\frac{m}{2}+k}}-\frac{q^2}{2^{m+2k}} \right)+O(1)$, and therefore,
\begin{eqnarray*}
&& \Lambda_m^k(f_1, f_2, f_3) \simeq 2^{2k}  \sum_{\substack{p \sim 2^{\frac{m}{2}+k} \\ q \sim 2^{\frac{m}{2}}}} \sum_{n_1, n_2 \sim 2^{\frac{m}{2}}} \left \langle f_1, \phi_{\frac{p}{2^k}-q, n_1}^{k} \right \rangle \left \langle f_2, \phi_{p+\frac{q^2}{2^{\frac{m}{2}}}, n_2}^{2k} \right \rangle \sum_{\widetilde{p}=0}^{2^{\frac{m}{2}-k}-1}  e^{-2i \cdot \frac{q^3 \widetilde{\lambda} \left(2^{\frac{m}{2}-k}p+\widetilde{p} \right)}{2^{m}}} \\
&& \quad \cdot e^{-i2^{\frac{m}{2}}n_2 \cdot \frac{q^2}{2^m}} \int_{I_{2^{\frac{m}{2}-k}p+\widetilde{p}}^{0, m+k}} 2^{\frac{m+k}{2}} e^{i 2^{m+k} \left(\frac{n_1}{2^{\frac{m}{2}}}+\frac{2^kn_2}{2^{\frac{m}{2}}} \right)} f_3(x) \cdot \left[\int_{I_q^{m, k}} e^{it \left(-2^{\frac{m}{2}+k}n_1+2^{\frac{m}{2}+k} \cdot \frac{2q}{2^{\frac{m}{2}}}n_2+ \frac{3q^2 \widetilde{\lambda} \left(2^{\frac{m}{2}-k}p+\widetilde{p} \right)}{2^{\frac{m}{2}-k}}  \right)} dt \right] dx \\
&& \quad \simeq 2^{2k}  \sum_{\substack{p \sim 2^{\frac{m}{2}+k} \\ q \sim 2^{\frac{m}{2}}}} \sum_{n_1, n_2 \sim 2^{\frac{m}{2}}} \left \langle f_1, \phi_{\frac{p}{2^k}-q, n_1}^{k} \right \rangle \left \langle f_2, \phi_{p+\frac{q^2}{2^{\frac{m}{2}}}, n_2}^{2k} \right \rangle \cdot \sum_{\widetilde{p}=0}^{2^{\frac{m}{2}-k}-1}  e^{-2i \cdot \frac{q^3 \widetilde{\lambda} \left(2^{\frac{m}{2}-k}p+\widetilde{p} \right)}{2^{m}}} e^{-i2^{\frac{m}{2}}n_2 \cdot \frac{q^2}{2^m}} \\
&& \qquad \qquad \cdot \left\langle f_3, \psi_{2^{\frac{m}{2}-k}p+\widetilde{p}, \frac{n_1}{2^{\frac{m}{2}}}+\frac{2^kn_2}{2^{\frac{m}{2}}}}^k \right \rangle  \int_{I_q^{m, k}} e^{it \left(-2^{\frac{m}{2}+k}n_1+2^{\frac{m}{2}+k} \cdot \frac{2q}{2^{\frac{m}{2}}}n_2+ \frac{3q^2 \widetilde{\lambda} \left(2^{\frac{m}{2}-k}p+\widetilde{p} \right)}{2^{\frac{m}{2}-k}}  \right)} dt.
\end{eqnarray*}
Observe that the above equation gives the following:
\begin{enumerate}
    \item \underline{frequency information of $f_3$}: this follows from the Gabor frame corresponding to $f_3$, which implies
    \begin{equation}
        n_3 \cong \frac{n_1}{2^{\frac{m}{2}}}+\frac{2^k n_2}{2^{\frac{m}{2}}} \cong  \frac{2^k n_2}{2^{\frac{m}{2}}}  \tag{{\bf Frequency Correlation}}
    \end{equation}
    \item \underline{time-frequency correlation between $p, q, n_1$ and $n_2$}: this follows from maximizing the $t$-integral. More precisely, we have 
    \begin{eqnarray*}
    && \int_{I_q^{m, k}} e^{it \left(-2^{\frac{m}{2}+k}n_1+2^{\frac{m}{2}+k} \cdot \frac{2q}{2^{\frac{m}{2}}}n_2+\frac{3q^2 \widetilde{\lambda} \left(2^{\frac{m}{2}-k}p+\widetilde{p} \right)}{2^{\frac{m}{2}-k}} \right)} dt \\
    && \quad =\int_{I_0^{m, k}} e^{i \left(t+\frac{q}{2^{\frac{m}{2}+k}} \right) \left(-2^{\frac{m}{2}+k}n_1+2^{\frac{m}{2}+k} \cdot \frac{2q}{2^{\frac{m}{2}}}n_2+\frac{3q^2 \widetilde{\lambda} \left(2^{\frac{m}{2}-k}p+\widetilde{p} \right)}{2^{\frac{m}{2}-k}} \right)} dt \\
    && \quad  =C_{p, \widetilde{p}, q, n_1, n_2} \cdot \int_{I_0^{m, k}} e^{i t \left(-2^{\frac{m}{2}+k}n_1+2^{\frac{m}{2}+k} \cdot \frac{2q}{2^{\frac{m}{2}}}n_2+\frac{3q^2 \widetilde{\lambda} \left(2^{\frac{m}{2}-k}p+\widetilde{p} \right)}{2^{\frac{m}{2}-k}} \right)} dt \\
    && \quad =\frac{C_{p, \widetilde{p}, q, n_1, n_2}}{2^{\frac{m}{2}+k}} \int_0^1 e^{it \left(-n_1+\frac{2qn_2}{2^{\frac{m}{2}}}+\frac{3q^2 \widetilde{\lambda} \left(2^{\frac{m}{2}-k}p+\widetilde{p} \right)}{2^m} \right)}dt,
    \end{eqnarray*}
    where $C_{p, \widetilde{p}, q, n_1, n_2}:=e^{i\frac{q}{2^{\frac{m}{2}+k}} \left(-2^{\frac{m}{2}+k}n_1+2^{\frac{m}{2}+k} \cdot \frac{2q}{2^{\frac{m}{2}}}n_2+\frac{3q^2 \widetilde{\lambda} \left(2^{\frac{m}{2}-k}p+\widetilde{p} \right)}{2^{\frac{m}{2}-k}} \right) }$. As a consequence, we see that, for fixed $p \sim 2^{\frac{m}{2}+k}$, $\widetilde{p} \sim 2^{\frac{m}{2}-k}$,  and $q \sim 2^{\frac{m}{2}}$, the main contribution of the $t$-integral arises when the pair $(n_1, n_2)$ belongs to the \emph{time-frequency correlation set}
    \begin{equation}
        {\textnormal{\textbf{TFC}}}_{p, \widetilde{p}, q}:=\left\{\substack{(n_1, n_2) \\ n_1, n_2 \sim 2^{\frac{m}{2}}}: n_1-\frac{2qn_2}{2^{\frac{m}{2}}}-\frac{3q^2 \widetilde{\lambda} \left(2^{\frac{m}{2}-k}p+\widetilde{p} \right)}{2^m} \cong 0 \right\}. \tag{\bf{Time Frequency Correlation set}}
    \end{equation}
\end{enumerate}
Utilizing all the time-frequency correlations, we see that
\begin{eqnarray*}
&& \Lambda_m^k(f_1, f_2, f_3) \simeq  2^{k-\frac{m}{2}}  \sum_{\substack{p \sim 2^{\frac{m}{2}+k} \\ q \sim 2^{\frac{m}{2}}}} \sum_{\widetilde{p}=0}^{2^{\frac{m}{2}-k}-1} \sum_{(n_1, n_2) \in {\textnormal{\textbf{TFC}}}_{p, \widetilde{p}, q} }  C_{\phi, p, \widetilde{p}, q, n_1, n_2} \nonumber \\
&& \qquad \qquad \qquad  \cdot \left \langle f_1, \phi_{\frac{p}{2^k}-q, n_1}^{k} \right \rangle \left \langle f_2, \phi_{p+\frac{q^2}{2^{\frac{m}{2}}}, n_2}^{2k} \right \rangle \left\langle f_3, \psi_{2^{\frac{m}{2}-k}p+\widetilde{p}, \frac{2^kn_2}{2^{\frac{m}{2}}}}^k \right \rangle , 
\end{eqnarray*}
where $C_{\phi, p, q, n_1, n_2}= e^{-2i \cdot \frac{q^3 \widetilde{\lambda} \left(2^{\frac{m}{2}-k}p+\widetilde{p} \right)}{2^m}}\cdot e^{i2^{\frac{m}{2}}n_2 \cdot \frac{q^2}{2^{m}}} \cdot C_{p, \widetilde{p}, q, n_1, n_2}$. 

This way we obtain the \emph{wave packet discretized (absolute summable) model} of the form $\Lambda_m^k(f_1, f_2, f_3)$ as 
\begin{eqnarray}  \label{20241013eq21}
&& \left(\Lambda_m^k \right)^*(f_1, f_2, f_3):=2^{k-\frac{m}{2}}  \sum_{\substack{p \sim 2^{\frac{m}{2}+k} \\ q \sim 2^{\frac{m}{2}}}} \sum_{\widetilde{p} \sim 2^{\frac{m}{2}-k}} \sum_{n_2 \sim 2^{\frac{m}{2}}}  \left| \left \langle f_1, \phi_{\frac{p}{2^k}-q, \frac{2qn_2}{2^{\frac{m}{2}}}+\frac{3q^2 \widetilde{\lambda} \left(2^{\frac{m}{2}-k}p+\widetilde{p} \right)}{2^m}}^{k} \right \rangle \right| \nonumber \\
&& \qquad \qquad \qquad \qquad \qquad \qquad  \cdot \left| \left \langle f_2, \phi_{p+\frac{q^2}{2^{\frac{m}{2}}}, n_2}^{2k} \right \rangle \right| \left| \left\langle f_3, \psi_{2^{\frac{m}{2}-k}p+\widetilde{p}, \frac{2^kn_2}{2^{\frac{m}{2}}}}^k \right \rangle \right|.
\end{eqnarray}

\subsubsection{A time-frequency sparse-uniform dichotomy.} 
We move now our focus on the term \eqref{20241013eq21}. First, we decompose the regime of $n_2$ into intervals of length $2^{\frac{m}{2}-k}$:
$$
\left\{n_2 \sim 2^{\frac{m}{2}} \right\}=\bigcup_{\widetilde{n_2} \sim 2^k} \left[\widetilde{n_2}2^{\frac{m}{2}-k}, \left(\widetilde{n_2}+1 \right)2^{\frac{m}{2}-k} \right], 
$$
which yields 
\begin{eqnarray} \label{20241013eq50}
&& \left(\Lambda_m^k \right)^*(f_1, f_2, f_3)=2^{k-\frac{m}{2}}  \sum_{\substack{p \sim 2^{\frac{m}{2}+k} \\ q \sim 2^{\frac{m}{2}}}} \sum_{\widetilde{p} \sim 2^{\frac{m}{2}-k}} \sum_{\substack{\widetilde{n_2} \sim 2^k \\ r \sim 2^{\frac{m}{2}-k}}}  \left| \left \langle f_1, \phi_{\frac{p}{2^k}-q, \frac{2q\left(2^{\frac{m}{2}-k}\widetilde{n_2}+r \right)}{2^{\frac{m}{2}}}+\frac{3q^2 \widetilde{\lambda} \left(2^{\frac{m}{2}-k}p+\widetilde{p} \right)}{2^m}}^{k} \right \rangle \right| \nonumber \\
&& \qquad \qquad \qquad \qquad \qquad \qquad  \cdot \left| \left \langle f_2, \phi_{p+\frac{q^2}{2^{\frac{m}{2}}}, 2^{\frac{m}{2}-k}\widetilde{n_2}+r}^{2k} \right \rangle \right| \left| \left\langle f_3, \psi_{2^{\frac{m}{2}-k}p+\widetilde{p}, \widetilde{n_2}}^k \right \rangle \right| \nonumber  \\
&& \quad \lesssim 2^{k-\frac{m}{2}}  \sum_{\substack{p \sim 2^{\frac{m}{2}+k} \\ q \sim 2^{\frac{m}{2}}}} \sum_{\substack{\widetilde{p} \sim 2^{\frac{m}{2}-k} \\ \widetilde{n_2} \sim 2^k}} \left| \left\langle f_3, \psi_{2^{\frac{m}{2}-k}p+\widetilde{p}, \widetilde{n_2}}^k \right \rangle \right| \cdot \left(\sum_{r \sim 2^{\frac{m}{2}-k}} \left| \left \langle f_2, \phi_{p+\frac{q^2}{2^{\frac{m}{2}}}, 2^{\frac{m}{2}-k}\widetilde{n_2}+r}^{2k} \right \rangle \right|^2 \right)^{\frac{1}{2}} \nonumber \\
&& \qquad \qquad \qquad \qquad \qquad \qquad  \cdot \left( \sum_{r \sim 2^{\frac{m}{2}-k}}  \left| \left \langle f_1, \phi_{\frac{p}{2^k}-q, \frac{2q\left(2^{\frac{m}{2}-k}\widetilde{n_2}+r \right)}{2^{\frac{m}{2}}}+\frac{3q^2 \widetilde{\lambda} \left(2^{\frac{m}{2}-k}p+\widetilde{p} \right)}{2^m}}^{k} \right \rangle \right|^2 \right)^{\frac{1}{2}}. 
\end{eqnarray}
For $L_1, L_2 \sim 2^k$, denote 
$$
\Omega_1^{L_1}:=\left[L_1 2^m, (L_1+1)2^m \right] \quad \textrm{and} \quad \Omega_2^{L_2}:=\left[L_2 2^{m+k}, (L_2+1)2^{m+k} \right], 
$$
and also let $f^{\Omega_j^{L_j}}$ be the frequency projection onto $\Omega_j^{L_j}$:
$$
f^{\Omega_j^{L_j}}(x):=\int_\R \widehat{f}(\xi) \chi \left( \frac{\xi-L_j2^{m+(j-1)k}}{|\Omega_j^{L_j}|} \right)e^{i\xi x}d\xi, \quad j=1, 2,
$$
for $\chi$ being a smooth cut-off with $\supp \chi \subset \left[\frac{1}{2}, 2 \right]$. 

Applying Parseval, we further estimate \eqref{20241013eq50} by 
\begin{eqnarray} \label{20241014eq01}
&& \left(\Lambda_m^k \right)^*(f_1, f_2, f_3) \lesssim 2^{k-\frac{m}{2}}  \sum_{\substack{p \sim 2^{\frac{m}{2}+k} \\ q \sim 2^{\frac{m}{2}}}} \sum_{\substack{\widetilde{p} \sim 2^{\frac{m}{2}-k} \\ \widetilde{n_2} \sim 2^k}} \left| \left\langle f_3, \psi_{2^{\frac{m}{2}-k}p+\widetilde{p}, \widetilde{n_2}}^k \right \rangle \right| \left\|f_2^{\Omega_2^{\widetilde{n_2}}} \right\|_{L^2 \left(I^{m, 2k}_{p+\frac{q^2}{2^{\frac{m}{2}}}} \right)} \nonumber \\
&& \qquad \qquad \qquad \qquad \qquad \qquad \qquad  \cdot \left\| f_1^{\Omega_1^{\frac{2q\widetilde{n_2}}{2^{\frac{m}{2}}}+\frac{3q^2}{2^m} \cdot \frac{\widetilde{\lambda} \left(2^{\frac{m}{2}-k}p+\widetilde{p} \right)}{2^{\frac{m}{2}-k}}}} \right\|_{L^2 \left(I^{m, k}_{\frac{p}{2^k}-q} \right)}.
\end{eqnarray}
Next, we decompose the regime of $q$ into intervals of length $2^{\frac{m}{2}-k}$
$$
\left\{q \sim 2^{\frac{m}{2}} \right\}=\bigcup_{\widetilde{q} \sim 2^k} \left[ \widetilde{q} 2^{\frac{m}{2}-k}, \left(\widetilde{q}+1 \right)2^{\frac{m}{2}-k} \right], 
$$
and notice that via Cauchy-Schwarz, \eqref{20241014eq01} becomes 
\begin{eqnarray} \label{20241014eq56}
&& \left(\Lambda_m^k \right)^*(f_1, f_2, f_3) \lesssim 2^{k-\frac{m}{2}}  \sum_{\substack{p \sim 2^{\frac{m}{2}+k} \\ \widetilde{q} \sim 2^k}} \sum_{\substack{\widetilde{p} \sim 2^{\frac{m}{2}-k} \\ \widetilde{n_2} \sim 2^k}} \sum_{s \sim 2^{\frac{m}{2}-k}}  \left| \left\langle f_3, \psi_{2^{\frac{m}{2}-k}p+\widetilde{p}, \widetilde{n_2}}^k \right \rangle \right| \left\|f_2^{\Omega_2^{\widetilde{n_2}}} \right\|_{L^2 \left(I^{m, 2k}_{p+\frac{2^{\frac{m}{2}}\widetilde{q}^2}{2^{2k}}+\frac{2\widetilde{q}s}{2^{k}}+\frac{s^2}{2^{\frac{m}{2}}}} \right)} \nonumber \\
&& \qquad \qquad \qquad \qquad \qquad \qquad \qquad  \cdot \left\| f_1^{\Omega_1^{\frac{2\widetilde{q}}{2^k}\widetilde{n_2}+\frac{3\widetilde{q}^2}{2^{2k}} \cdot \frac{\widetilde{\lambda} \left(2^{\frac{m}{2}-k}p+\widetilde{p} \right)}{2^{\frac{m}{2}-k}}}} \right\|_{L^2 \left(I^{m, k}_{\frac{p}{2^k}-2^{\frac{m}{2}-k}\widetilde{q}-s} \right)} \nonumber \\
&& \lesssim 2^{k-\frac{m}{2}}  \sum_{\substack{p \sim 2^{\frac{m}{2}+k} \\ \widetilde{p} \sim 2^{\frac{m}{2}-k}}} \sum_{\substack{\widetilde{q} \sim 2^k \\ \widetilde{n_2} \sim 2^k}}  \left| \left\langle f_3, \psi_{2^{\frac{m}{2}-k}p+\widetilde{p}, \widetilde{n_2}}^k \right \rangle \right|  \left( \sum_{s \sim 2^{\frac{m}{2}-k}} \left\|f_2^{\Omega_2^{\widetilde{n_2}}} \right\|_{L^2 \left(I^{m, 2k}_{p+\frac{2^{\frac{m}{2}}\widetilde{q}^2}{2^{2k}}+\frac{2\widetilde{q}s}{2^{k}}+\frac{s^2}{2^{\frac{m}{2}}}} \right)}^2 \right)^{\frac{1}{2}} \nonumber \\
&& \qquad \qquad \qquad \qquad \qquad \qquad \qquad  \cdot \left( \sum_{s \sim 2^{\frac{m}{2}-k}} \left\| f_1^{\Omega_1^{\frac{2\widetilde{q}}{2^k}\widetilde{n_2}+\frac{3\widetilde{q}^2}{2^{2k}} \cdot \frac{\widetilde{\lambda} \left(2^{\frac{m}{2}-k}p+\widetilde{p} \right)}{2^{\frac{m}{2}-k}}}} \right\|_{L^2 \left(I^{m, k}_{\frac{p}{2^k}-2^{\frac{m}{2}-k}\widetilde{q}-s} \right)}^2 \right)^{\frac{1}{2}} \nonumber \\
&& \lesssim  2^{k-\frac{m}{2}}   \sum_{\substack{p \sim 2^{\frac{m}{2}+k} \\ \widetilde{p} \sim 2^{\frac{m}{2}-k}}} \sum_{\substack{\widetilde{q} \sim 2^k \\ \widetilde{n_2} \sim 2^k}}  \left\|f_3^{\Omega_2^{\widetilde{n_2}}} \right\|_{L^2 \left(I^{0, m+k}_{2^{\frac{m}{2}-k}p+\widetilde{p}} \right)} \left\|f_2^{\Omega_2^{\widetilde{n_2}}} \right\|_{L^2 \left(I_{\frac{p}{2^{\frac{m}{2}-k}}+\frac{\widetilde{q}^2}{2^k}}^{0, 3k} \right)} \left\| f_1^{\Omega_1^{\frac{2\widetilde{q}}{2^k}\widetilde{n_2}+\frac{3\widetilde{q}^2}{2^{2k}} \cdot \frac{\widetilde{\lambda} \left(2^{\frac{m}{2}-k}p+\widetilde{p} \right)}{2^{\frac{m}{2}-k}}}} \right\|_{L^2 \left(I^{0, 2k}_{\frac{p}{2^{\frac{m}{2}}}-\widetilde{q}} \right)}.
\end{eqnarray}
Guided by \eqref{20241014eq56}, we define the \emph{time-frequency sparse sets of indices} by
$$
\calS_\nu(f_j):=\left\{ \substack{(Q_j, L_j) \\ Q_j \sim 2^k, L_j \sim 2^k}: \int_{I_{Q_j}^{0, (j+1)k}} \left|f_j^{\Omega_j^{L_j}} \right|^2 \gtrsim 2^{-\nu k} \int_{I_{Q_j}^{0, (j+1)k}}|f_j|^2 \right\}, \quad j\in\{1, 2\}\,,
$$
and
$$
\calS_\nu(f_3):=\left\{ \substack{(Q_3, L_3) \\ Q_3 \sim 2^m, L_3 \sim 2^k}: \int_{I_{Q_3}^{0, m+k}} \left|f_3^{\Omega_2^{L_3}} \right|^2 \gtrsim 2^{-\nu k} \int_{I_{Q_3}^{0, m+k}}|f_3|^2 \right\}\,,
$$
where here $\nu>0$ is a small parameter to be chosen later.

We also define the \emph{time-frequency uniform sets of indices} $\calU_\mu(f_1), \calU_\mu(f_2)$ and $\calU_\mu(f_3)$ as the complements of the corresponding sparse sets. 

Now, we decompose our functions into sparse and uniform components as follows:
\begin{enumerate}
    \item [$\bullet$] for $f_j$ with $j=1, 2$, the \emph{uniform} and \emph{sparse} component are defined as
    $$
    f_j^{\calU_\nu}:=\sum_{(Q_j, L_j) \in \calU_\nu(f_j)} \one_{I_{Q_j}^{0, (j+1)k}} f_j^{\Omega_j^{L_j}} \quad \textrm{and} \quad f_j^{\calS_\nu}:=\sum_{(Q_j, L_j) \in \calS_\nu(f_j)} \one_{I_{Q_j}^{0, (j+1)k}} f_j^{\Omega_j^{L_j}} \quad \textrm{respectively}; 
    $$
    \item [$\bullet$] for $f_3$, the \emph{uniform} and \emph{sparse} component are defined as
    \begin{equation} \label{20241015eq01}
    f_3^{\calU_\nu}:=\sum_{(Q_3, L_3) \in \calU_\nu(f_3)} \one_{I_{Q_3}^{0, m+k}} f_3^{\Omega_2^{L_3}} \quad \textrm{and} \quad f_3^{\calS_\nu}:=\sum_{(Q_3, L_3) \in \calS_\nu(f_3)} \one_{I_{Q_3}^{0, m+k}} f_3^{\Omega_2^{L_3}} \quad \textrm{respectively}.
    \end{equation} 
    \end{enumerate}
With these done, we decompose 
$$
\left(\Lambda_m^k \right)^*(f_1, f_2, f_3):=\left(\Lambda_m^k \right)^{*}_{\calU_\nu}(f_1, f_2, f_3)+\left(\Lambda_m^k \right)^{*}_{\calS_\nu}(f_1, f_2, f_3),
$$
with $\left(\Lambda_m^k \right)^{*}_{\calS_\nu}(f_1, f_2, f_3):=\left(\Lambda_m^k \right)^{*}\left(f^{\calS_\nu}_1, f^{\calS_\nu}_2, f^{\calS_\nu}_3 \right)$. 

\subsubsection{Treatment of $\left(\Lambda_m^k \right)^{*}_{\calU_\nu}(f_1, f_2, f_3)$.} \label{20241126subsubsec01} In this case, at least one of the input functions is uniform. Without loss of generality, we assume that $f_3$ is uniform. All the other cases can be treated in a similar way, and hence we will omit their proofs here. 

To begin with, from \eqref{20241014eq56}, we have 
\begin{eqnarray*}
&& \left(\Lambda_m^k \right)^* \left(f_1, f_3, f_3^{\calU_\nu} \right) \lesssim 2^{k-\frac{m}{2}}   \sum_{\substack{p \sim 2^{\frac{m}{2}+k} \\ \widetilde{p} \sim 2^{\frac{m}{2}-k}}} \sum_{\substack{\widetilde{q} \sim 2^k \\ \widetilde{n_2} \sim 2^k}}  \left\|\left(f^{\calU_\nu}_3 \right)^{\Omega_2^{\widetilde{n_2}}} \right\|_{L^2 \left(I^{0, m+k}_{2^{\frac{m}{2}-k}p+\widetilde{p}} \right)} \\
&& \qquad \qquad \qquad \qquad \qquad \qquad \qquad  
 \cdot \left\|f_2^{\Omega_2^{\widetilde{n_2}}} \right\|_{L^2 \left(I_{\frac{p}{2^{\frac{m}{2}-k}}+\frac{\widetilde{q}^2}{2^k}}^{0, 3k} \right)} \left\| f_1^{\Omega_1^{\frac{2\widetilde{q}}{2^k}\widetilde{n_2}+\frac{3\widetilde{q}^2}{2^{2k}} \cdot \frac{\widetilde{\lambda} \left(2^{\frac{m}{2}-k}p+\widetilde{p} \right)}{2^{\frac{m}{2}-k}}}} \right\|_{L^2 \left(I^{0, 2k}_{\frac{p}{2^{\frac{m}{2}}}-\widetilde{q}} \right)}.
 \end{eqnarray*}
Using the assumption that $f_3$ is uniform (in the sense of \eqref{20241015eq01}), the above expression is bounded above by
\begin{eqnarray*}
&& 2^{k-\frac{m}{2}} \cdot 2^{-\frac{\nu k}{2}}  \sum_{\substack{p \sim 2^{\frac{m}{2}+k} \\ \widetilde{p} \sim 2^{\frac{m}{2}-k}}} \sum_{\substack{\widetilde{q} \sim 2^k \\ \widetilde{n_2} \sim 2^k}}  \left\|f_3 \right\|_{L^2 \left(I^{0, m+k}_{2^{\frac{m}{2}-k}p+\widetilde{p}} \right)} \\
&& \qquad \qquad \qquad \qquad \qquad \qquad \qquad  
 \cdot \left\|f_2^{\Omega_2^{\widetilde{n_2}}} \right\|_{L^2 \left(I_{\frac{p}{2^{\frac{m}{2}-k}}+\frac{\widetilde{q}^2}{2^k}}^{0, 3k} \right)} \left\| f_1^{\Omega_1^{\frac{2\widetilde{q}}{2^k}\widetilde{n_2}+\frac{3\widetilde{q}^2}{2^{2k}} \cdot \frac{\widetilde{\lambda} \left(2^{\frac{m}{2}-k}p+\widetilde{p} \right)}{2^{\frac{m}{2}-k}}}} \right\|_{L^2 \left(I^{0, 2k}_{\frac{p}{2^{\frac{m}{2}}}-\widetilde{q}} \right)}.
 \end{eqnarray*}
 Applying Cauchy-Schwarz in the following order: first in $n_2$, then in $\widetilde{q}$, then in $\tilde{p}$ and finally in $p$, we see that the above expression is bounded above by 
 \begin{eqnarray*}
 && 2^{k-\frac{m}{2}} \cdot 2^{-\frac{\nu k}{2}}  \sum_{\tilde{p} \sim 2^{\frac{m}{2}-k}} \sum_{\substack{p \sim 2^{\frac{m}{2}+k} \\ \widetilde{q} \sim 2^k}} \left\|f_3 \right\|_{L^2 \left(I^{0, m+k}_{2^{\frac{m}{2}-k}p+\widetilde{p}} \right)} \left\|f_2 \right\|_{L^2 \left(I_{\frac{p}{2^{\frac{m}{2}-k}}+\frac{\widetilde{q}^2}{2^k}}^{0, 3k} \right)} \left\| f_1 \right\|_{L^2 \left(I^{0, 2k}_{\frac{p}{2^{\frac{m}{2}}}-\widetilde{q}} \right)} \\
 && \lesssim 2^{k-\frac{m}{2}} \cdot 2^{-\frac{\nu k}{2}}  \sum_{\tilde{p} \sim 2^{\frac{m}{2}-k}} \sum_{p \sim 2^{\frac{m}{2}+k}} \left\|f_3 \right\|_{L^2 \left(I^{0, m+k}_{2^{\frac{m}{2}-k}p+\widetilde{p}} \right)}\left( \sum_{\widetilde{q} \sim 2^k}\left\|f_2 \right\|^2_{L^2 \left(I_{\frac{p}{2^{\frac{m}{2}-k}}+\frac{\widetilde{q}^2}{2^k}}^{0, 3k} \right)} \right)^{\frac{1}{2}} \left( \sum_{\widetilde{q} \sim 2^k} \left\| f_1 \right\|^2_{L^2 \left(I^{0, 2k}_{\frac{p}{2^{\frac{m}{2}}}-\widetilde{q}} \right)} \right)^{\frac{1}{2}} \\
 && \lesssim 2^{\frac{k}{2}-\frac{m}{4}} \cdot 2^{-\frac{\nu k}{2}} \left\|f_1 \right\|_{L^2(3I^k)} \sum_{p \sim 2^{\frac{m}{2}+k}} \left\|f_2 \right\|_{L^2 \left(I^{0, 2k}_{\frac{p}{2^{\frac{m}{2}}}} \right)} \left\|f_3 \right\|_{L^2 \left(I^{0, \frac{m}{2}+2k}_{p} \right)} \\
 && \lesssim 2^{\frac{k}{2}} \cdot  2^{-\frac{\nu k}{2}} \left\|f_1 \right\|_{L^2(3I^k)} \sum_{p \sim 2^{k}} \left\|f_2 \right\|_{L^2 \left(I^{0, 2k}_{p} \right)} \left\|f_3 \right\|_{L^2 \left(I^{0, 2k}_{p} \right)}\\
 && \lesssim 2^{\frac{k}{2}} \cdot  2^{-\frac{\nu k}{2}} \left\|f_1 \right\|_{L^2(3I^k)}\left\|f_2 \right\|_{L^2(3I^k)}\left\|f_3 \right\|_{L^2(3I^k)}, 
\end{eqnarray*}
which gives the desired estimate. 

\subsubsection{Treatment of $\left(\Lambda_m^k \right)^{*}_{\calS_\nu}(f_1, f_2, f_3)$.} \label{20241126subsubsec02} In this case, we have 
\begin{enumerate}
    \item [$\bullet$] if $j \in \{1, 2\}$ then up to a $2^{\nu k}$ factor there is a unique measurable function $L_j: [2^{jk}, 2^{jk+1}] \cap \Z \mapsto \left[2^k, 2^{k+1} \right] \cap \Z$ such that for each $Q_j \sim 2^{jk}$ one has
    \begin{equation} \label{20241015eq10}
    2^{-\nu k} \int_{I_{Q_j}^{0, (j+1)k}} |f_j|^2 \lesssim  \int_{I_{Q_j}^{0, (j+1)k}} \left| f_j^{\Omega_j^{L_j(Q_j)}} \right|^2 \lesssim  \int_{I_{Q_j}^{0, (j+1)k}} |f_j|^2, \quad j \in \{1, 2\};
    \end{equation} 
    \item [$\bullet$] up to a $2^{\nu k}$ factor there is a unique measurable function $L_3: [2^m, 2^{m+1}] \cap \Z \mapsto \left[2^k, 2^{k+1} \right] \cap \Z$ such that for each $Q_3 \sim 2^m$ one has
    \begin{equation} \label{20241015eq10}
    2^{-\nu k} \int_{I_{Q_3}^{0, m+k}} |f_3|^2 \lesssim  \int_{I_{Q_3}^{0, m+k}} \left| f_3^{\Omega_3^{L_3(Q_3)}} \right|^2 \lesssim  \int_{I_{Q_3}^{0, m+k}} |f_3|^2.
    \end{equation} 
\end{enumerate}
Combing this with \eqref{20241014eq56}, we see that
\begin{eqnarray*}
&& \left(\Lambda_m^k \right)^{*}_{\calS_\nu}(f_1, f_2, f_3) \lesssim  2^{k-\frac{m}{2}} \cdot 2^{3\nu k} \sum_{\widetilde{p} \sim 2^{\frac{m}{2}-k}} \sum_{(p, \widetilde{q}) \in \TFC_{\widetilde{p}}}  \left\|f_3 \right\|_{L^2 \left(I^{0, m+k}_{2^{\frac{m}{2}-k}p+\widetilde{p}} \right)} \left\|f_2 \right\|_{L^2 \left(I_{\frac{p}{2^{\frac{m}{2}-k}}+\frac{\widetilde{q}^2}{2^k}}^{0, 3k} \right)} \left\| f_1 \right\|_{L^2 \left(I^{0, 2k}_{\frac{p}{2^{\frac{m}{2}}}-\widetilde{q}} \right)} \\
&& \lesssim 2^{k-\frac{m}{2}} \cdot 2^{3\nu k} \left( \sum_{\widetilde{p} \sim 2^{\frac{m}{2}-k}} \left(\# \TFC_{\widetilde{p}} \right) \right)^{\frac{1}{2}} \left(\sum_{\substack{\widetilde{p} \sim 2^{\frac{m}{2}-k} \\ p \sim 2^{\frac{m}{2}+k} \\ \widetilde{q} \sim 2^k}}  \left\|f_3 \right\|^2_{L^2 \left(I^{0, m+k}_{2^{\frac{m}{2}-k}p+\widetilde{p}} \right)} \left\|f_2 \right\|^2_{L^2 \left(I_{\frac{p}{2^{\frac{m}{2}-k}}+\frac{\widetilde{q}^2}{2^k}}^{0, 3k} \right)} \left\| f_1 \right\|^2_{L^2 \left(I^{0, 2k}_{\frac{p}{2^{\frac{m}{2}}}-\widetilde{q}} \right)}  \right)^{\frac{1}{2}},
\end{eqnarray*}
where for each $\widetilde{p} \sim 2^{\frac{m}{2}-k}$, 
$$
\TFC_{\widetilde{p}}:=\left\{\substack{(p, \widetilde{q}) \\ p \sim 2^{\frac{m}{2}+k}, \widetilde{q} \sim 2^k}: L_1\left( \frac{p}{2^{\frac{m}{2}}}-\widetilde{q} \right) \cong \frac{2 \widetilde{q}}{2^k}L_2 \left(\frac{p}{2^{\frac{m}{2}-k}}+\frac{\widetilde{q}^2}{2^k} \right)+ \frac{3\widetilde{q}^2}{2^{2k}} \cdot \frac{\widetilde{\lambda} \left(2^{\frac{m}{2}-k}p+\widetilde{p} \right)}{2^{\frac{m}{2}-k}}  \right\}.
$$
The desired decay comes now from the claim that there exists a suitable $\epsilon_5'>0$ such that for each $\widetilde{p} \sim 2^{\frac{m}{2}-k}$, 
\begin{equation} \label{20241015eq23}
\#\TFC_{\widetilde{p}} \lesssim 2^{\frac{m}{2}+2k-\epsilon_5'k}. 
\end{equation} 
We claim that in this case, one can reduce the estimate \eqref{20241015eq23} to the situation that was treated in \cite[Section 5.3.2]{HL23}. To see this, we further decompose the regime of $p$ as 
$$
\left\{p \sim 2^{\frac{m}{2}+k} \right\}=\bigcup_{p' \sim 2^{2k}} \left[2^{\frac{m}{2}-k}p', 2^{\frac{m}{2}-k}(p'+1)\right].
$$
Using the above, we further decompose the set $\TFC_{\widetilde{p}}$ as 
$$
\TFC_{\widetilde{p}}=\bigcup_{l \sim 2^{\frac{m}{2}-k}} \TFC_{\widetilde{p}, l}, 
$$
where for each $\widetilde{p}, l \sim 2^{\frac{m}{2}-k}$,
$$
\TFC_{\widetilde{p}, l}:=\left\{ \substack{\left(p', \widetilde{q}\right) \\ p' \sim 2^{2k}, \widetilde{q} \sim 2^k}: L_1\left( \frac{p'}{2^k}-\widetilde{q} \right) \cong \frac{2 \widetilde{q}}{2^k}L_2 \left(p'+\frac{\widetilde{q}^2}{2^k} \right)+ \frac{3\widetilde{q}^2}{2^{2k}} \cdot \frac{\widetilde{\lambda} \left(2^{m-2k}p'+2^{\frac{m}{2}-k}l+\widetilde{p} \right)}{2^{\frac{m}{2}-k}}  \right\}.
$$
Since in this case $\widetilde{p}$ and $l$ are fixed, by denoting
$$
L_{3, \widetilde{p}, l}(p'):=\frac{\widetilde{\lambda} \left(2^{m-2k}p'+2^{\frac{m}{2}-k}l+\widetilde{p} \right)}{2^{\frac{m}{2}-k}},
$$
we notice that the estimate for the size of the set $\TFC_{\widetilde{p}, l}$ reduces to the situation treated in \cite[Page 58]{HL23}. 
Thus, it follows that $\#\TFC_{\widetilde{p}, l} \lesssim 2^{(3-\epsilon_5')k}$ and hence our proof is now complete. 

\medskip 

\subsection{Treatment of the uniform component ${\bf L}_{m, k}^{\calU_\mu}(f_1, f_2)$: Case I in Overview} \label{20241028subsec01}

Following the discussion in Section \ref{20241023overview01}, here we only consider \underline{\textsf{Step I.2.1}}, that is, the situation when $\frac{m}{2} \le k \le m$. At the core of our reasoning lies in the \emph{constancy propagation} of $\widetilde{\lambda}$ from intervals of length  $2^{-m-2k}$ to intervals of length $2^{-m-k}$. This strategy resembles the one presented in the proof of Proposition \ref{20241018prop01}, and hence, in this section, we will only focus on the key modifications required for treating our present case.

With these in mind, we follow the reasoning presented in \eqref{20241004eq21}--\eqref{20241011eq40}, and, with the obvious adaptations, as well as notations---see in particular \ref{KH2}---we reduce matters to the following estimate: for each $\widetilde{p} \sim 2^k$, there exists some $\widetilde{\epsilon_1}>0$, such that 
\begin{equation} \label{20241023eq01}
\left|\widetilde{V}_{m, k}^{\widetilde{p}, \frakL}(f_1, f_2) \right| \lesssim 2^{-\widetilde{\epsilon_1} m} \left\|f_1 \right\|_{L^2(3I^k)}^2 \left\|f_2 \right\|^2_{L^2(3I^k)}, 
\end{equation} 
where
$$
\widetilde{V}_{m, k}^{\widetilde{p}, \frakL}(f_1, f_2):=2^{\frac{m}{2}+2k} \sum_{p \in \left[\widetilde{p} 2^{m-k}, (\widetilde{p}+1)2^{m-k} \right]} \sum_{q \sim 2^{\frac{m}{2}}} \int_{\widetilde{\frakL_p^k}} \left| \int_{I_q^{m, k}} f_1(x-t)f_2(x+t^2) e^{3i 2^k \cdot \frac{q^2 \widetilde{\lambda}(x)}{2^{\frac{m}{2}}} t} dt \right|^2  dx.
$$
Here, for $\widetilde{\epsilon}_0>0$ and $p \sim 2^m$, the \emph{light set of $\ell$-indices} associated to the interval $I_p^{0, k+m}$ is defined as
$$
\widetilde{\frakL_p^k}:=\bigcup_{\ell \in \frakL_p^k} \frakM^\ell \left(I_p^{0, k+m} \right) \quad \textrm{with} \quad 
\frakL_p^k:=\left\{\ell \sim 2^{\frac{m}{2}-\widetilde{\epsilon_0} m}: \left| \frakM^\ell \left(I_p^{0, k+m} \right) \right| \lesssim \left|I_p^{0, k+m} \right| 2^{-\widetilde{\epsilon_0} m} \right\}
$$
and
$$
\frakM^\ell \left(I_p^{0, k+m} \right):=\left\{x \in I_p^{0, k+m}: \widetilde{\lambda}(x) \in \left[ \ell 2^{\widetilde{\epsilon_0} m}, (\ell+1)2^{\widetilde{\epsilon_0} m}  \right] \right\}.
$$
Next, following the argument in \textsf{Step I} in Proposition \ref{20241018prop01}, we may assume $\widetilde{\frakL_p^k}=I_p^{0, m+k}$ for each $p \sim 2^m$. Moreover, another application of Proposition \ref{20241006prop01} (with $2^{\del k}$ and $2^{-(\del-2\mu)k}$ replaced by $2^{\del m}$ and $2^{-(\del-2\mu)m}$, respectively) followed by the change of variable\footnote{Here we make again use of Remark \ref{Simplif}.} $t \to t+\frac{q}{2^{\frac{m}{2}+k}}$ reduces our analysis to 
\begin{eqnarray*}
&& {\bf \widetilde{V}}_{m, k}^{\widetilde{p}, \frakL}(f_1, f_2):=2^{\frac{m}{2}+2k+\del m} \sum_{p \in \left[\widetilde{p}2^{m-k}, (\widetilde{p}+1) 2^{m-k}  \right]} \sum_{q \sim 2^\frac{m}{2}}  \int_{I_0^{0, k+m}} \bigg| \int_{I_0^{m, k}} f_1 \left(\frac{p}{2^{k+m}}-\frac{q}{2^{\frac{m}{2}+k}}-t \right) \\
&& \qquad \qquad \qquad \qquad  \cdot f_2 \left(x+\frac{p}{2^{k+m}}+\frac{2tq}{2^{\frac{m}{2}+k}}+\frac{q^2}{2^{m+2k}} \right)  e^{3i 2^k \frac{q^2 \widetilde{\lambda}\left(x+\frac{p}{2^{k+m}} \right)}{2^{\frac{m}{2}}} t} dt \bigg|^2 dx. 
\end{eqnarray*}
In this case, instead of further subdividing the range of $p$ (see, \eqref{20241024eq01}), we decompose the range of $q$:
$$
\left\{q \sim 2^{\frac{m}{2}} \right\}=\bigcup_{q' \sim 2^{m-k}} \left[q'2^{k-\frac{m}{2}}, (q'+1)2^{k-\frac{m}{2}} \right].
$$
Thus, we have
\begin{eqnarray*}
&& {\bf \widetilde{V}}_{m, k}^{\widetilde{p}, \frakL}(f_1, f_2) \lesssim 2^{\frac{m}{2}+2k+\del m} \sum_{p \in \left[\widetilde{p}2^{m-k}, (\widetilde{p}+1) 2^{m-k}  \right]} \sum_{q' \sim 2^{m-k}} \sum_{\ell \sim 2^{k-\frac{m}{2}}} \int_{I_0^{0, k+m}} \bigg| \int_{I_0^{m, k}} f_1 \left(\frac{p}{2^{k+m}}-\frac{q'}{2^m}-\frac{l}{2^{\frac{m}{2}+k}}-t \right) \\
&& \qquad \cdot f_2 \left(x+\frac{p}{2^{k+m}}+\frac{2tq'}{2^m}+\frac{2t\ell}{2^{\frac{m}{2}+k}}+\frac{(q')^2}{2^{2m}}+\frac{2q'\ell}{2^{\frac{3m}{2}+k}}+\frac{\ell^2}{2^{m+2k}} \right)  e^{3i 2^k \frac{\left(q'2^{k-\frac{m}{2}}+\ell \right)^2 \widetilde{\lambda}\left(x+\frac{p}{2^{k+m}} \right)}{2^{\frac{m}{2}}} t} dt \bigg|^2 dx. 
\end{eqnarray*}
Applying the change of variables $p \to p-\left\lfloor \frac{(q')^2}{2^{m-k}} \right \rfloor$ and $x \to x-\frac{2q'\ell}{2^{\frac{3m}{2}+k}}-\frac{\ell^2}{2^{m+2k}}-\frac{\left\{ \frac{\left(q' \right)^2}{2^{m-k}} \right\}}{2^{m+k}}$, we have
\begin{eqnarray*}
&& {\bf \widetilde{V}}_{m, k}^{\widetilde{p}, \frakL}(f_1, f_2) \\
&& \lesssim 2^{\frac{m}{2}+2k+\del m} \sum_{p \in \left[\widetilde{p}2^{m-k}, (\widetilde{p}+2) 2^{m-k}  \right]} \sum_{q' \sim 2^{m-k}} \sum_{\ell \sim 2^{k-\frac{m}{2}}} \int_{9I_0^{0, k+m}} \bigg| \int_{I_0^{m, k}} f_1 \left(\frac{p}{2^{k+m}}-\frac{q'}{2^m}-\frac{(q')^2}{2^{2m}}-\frac{l}{2^{\frac{m}{2}+k}}-t \right) \\
&& \qquad \cdot f_2 \left(x+\frac{p}{2^{k+m}}+\frac{2tq'}{2^m}+\frac{2t\ell}{2^{\frac{m}{2}+k}}\right)  e^{3i 2^k \frac{\left(q'2^{k-\frac{m}{2}}+\ell \right)^2 \widetilde{\lambda}\left(x+\frac{p}{2^{k+m}}-\frac{(q')^2}{2^{2m}}-\frac{2q'\ell}{2^{\frac{3m}{2}+k}}-\frac{\ell^2}{2^{m+2k}}\right)}{2^{\frac{m}{2}}} t} dt \bigg|^2 dx. 
\end{eqnarray*}
Opening the square and then applying Cauchy-Schwarz, we have
$$
{\bf \widetilde{V}}_{m, k}^{\widetilde{p}, \frakL}(f_1, f_2) \lesssim 2^{\frac{m}{2}+2k+\del m} \widetilde{\calA} \,\widetilde{\calB}, 
$$
where 
\begin{eqnarray*}
&& \widetilde{\calA}^2:= \iint_{\left(I_0^{m, k} \right)^2} \sum_{p \in \left[\widetilde{p}2^{m-k}, (\widetilde{p}+2) 2^{m-k}  \right]} \sum_{q' \sim 2^{m-k}} \sum_{\ell \sim 2^{k-\frac{m}{2}}} \left|f_1 \left(\frac{p}{2^{k+m}}-\frac{q'}{2^m}-\frac{(q')^2}{2^{2m}}-\frac{l}{2^{\frac{m}{2}+k}}-t \right) \right|^2\\
&& \qquad \qquad \qquad \qquad  \cdot \left| f_1 \left(\frac{p}{2^{k+m}}-\frac{q'}{2^m}-\frac{(q')^2}{2^{2m}}-\frac{l}{2^{\frac{m}{2}+k}}-s \right) \right|^2 dtds
\end{eqnarray*}
and
\begin{eqnarray*}
&& \widetilde{\calB}^2:= \iint_{\left(I_0^{m, k} \right)^2} \sum_{p \in \left[\widetilde{p}2^{m-k}, (\widetilde{p}+2) 2^{m-k}  \right]} \sum_{q' \sim 2^{m-k}} \sum_{\ell \sim 2^{k-\frac{m}{2}}} \\
&& \qquad \quad \Bigg | \int_{9I_0^{0, k+m}}  f_2 \left(x+\frac{p}{2^{k+m}}+\frac{2tq'}{2^m}+\frac{2t\ell}{2^{\frac{m}{2}+k}}\right) f_2 \left(x+\frac{p}{2^{k+m}}+\frac{2sq'}{2^m}+\frac{2s\ell}{2^{\frac{m}{2}+k}}\right) \\
&& \qquad \qquad \cdot e^{3i \cdot 2^k \frac{\left(q'2^{k-\frac{m}{2}}+\ell \right)^2 \widetilde{\lambda}\left(x+\frac{p}{2^{k+m}}-\frac{(q')^2}{2^{2m}}-\frac{2q'\ell}{2^{\frac{3m}{2}+k}}-\frac{\ell^2}{2^{m+2k}}\right)}{2^{\frac{m}{2}}} (t-s)} dx \Bigg |^2 dtds. 
\end{eqnarray*}
The analysis of the term $\widetilde{\calA}^2$ follows as usual by the assumption that $f_1$ is uniform (see \eqref{uniformest}), and hence we omit it here. We thus move our focus on $\widetilde{\calB}^2$. We first apply the change of variable back $q=q'2^{k-\frac{m}{2}}+\ell$ followed by the change of variables $t \to \frac{2^{\frac{m}{2}}}{q} t$, $s \to \frac{2^{\frac{m}{2}}}{q}s$, and $s \to t-v$ in order to deduce 
\begin{eqnarray*}
&& \widetilde{\calB}^2 \lesssim \iint_{3I_0^{m, k} \times 7I_0^{m, k}} \sum_{p \in \left[ \widetilde{p}2^{m-k}, (\widetilde{p}+2)2^{m-k} \right]} \sum_{q \sim 2^{\frac{m}{2}}} \bigg | \int_{9I_0^{0, k+m}} f_2 \left(x+\frac{p}{2^{k+m}}+\frac{2t}{2^k} \right) \\
&&\qquad \qquad \qquad  \cdot f_2 \left(x+\frac{p}{2^{k+m}}+\frac{2(t-v)}{2^k} \right) e^{3i \cdot 2^k q \widetilde{\lambda} \left(x+\frac{p}{2^{k+m}}-\frac{q^2}{2^{m+2k}} \right)v} dx \bigg|^2 dtdv.
\end{eqnarray*}
For technical reasoning, letting $\widetilde{\epsilon}_2>0$ being sufficiently small, we divide the $t$-interval as follows: 
$$
7I_0^{m, k}:=\left[0, \frac{7}{2^{\frac{m}{2}+k}} \right]=\bigcup_{\iota=0}^{ 7 \cdot 2^{\widetilde{\epsilon}_2 m}-1}  \left[ \frac{\iota}{2^{\frac{m}{2}+k+\widetilde{\epsilon}_2m }}, \frac{\iota+1}{2^{\frac{m}{2}+k+\widetilde{\epsilon}_2m }} \right]=\bigcup_{\iota=0}^{ 7 \cdot 2^{\widetilde{\epsilon}_2 m}-1} I_{\iota}^{m, k+\widetilde{\epsilon_2}m}, 
$$
which gives (after a change of variable)
\begin{eqnarray*}
&& \widetilde{\calB}^2 \lesssim \sum_{\iota \sim 2^{\widetilde{\epsilon_2} m}}  \iint_{3I_0^{m, k} \times I_0^{m, k+\widetilde{\epsilon}_2m }} \sum_{p \in \left[ \widetilde{p}2^{m-k}, (\widetilde{p}+2)2^{m-k} \right]} \sum_{q \sim 2^{\frac{m}{2}}} \bigg | \int_{9I_0^{0, k+m}} f_2 \left(x+\frac{p}{2^{k+m}}+\frac{2\iota}{2^{\frac{m}{2}+2k+\widetilde{\epsilon_2}m}}+\frac{2t}{2^k} \right) \\
&&\qquad \qquad \cdot f_2 \left(x+\frac{p}{2^{k+m}}+\frac{2\iota}{2^{\frac{m}{2}+2k+\widetilde{\epsilon_2}m}}+\frac{2(t-v)}{2^k} \right) e^{3i \cdot 2^k q \widetilde{\lambda} \left(x+\frac{p}{2^{k+m}}-\frac{q^2}{2^{m+2k}} \right)v} dx \bigg|^2 dtdv.
\end{eqnarray*}
Applying now the change of variable $x \to x-\frac{2t}{2^k}$, then, up to an admissible error $O\left(2^{-(\widetilde{\epsilon_2}-\delta)m} \left\|f_1 \right\|^2_{L^2(3I^k)} \left\|f_2 \right\|^2_{L^2(3I^k)}\right)$ in \eqref{20241023eq01}, we further have
\begin{eqnarray*}
&& \widetilde{\calB}^2 \lesssim \sum_{\iota \sim 2^{\widetilde{\epsilon_2} m}} \iint_{3I_0^{m, k} \times I_0^{m, k+\widetilde{\epsilon_2}m}} \sum_{p \in \left[ \widetilde{p}2^{m-k}, (\widetilde{p}+2)2^{m-k} \right]} \sum_{q \sim 2^{\frac{m}{2}}} \bigg | \int_{9I_0^{0, k+m}} f_2 \left(x+\frac{p}{2^{k+m}}+\frac{2\iota}{2^{\frac{m}{2}+2k+\widetilde{\epsilon_2}m}}\right) \\
&&\qquad \qquad \qquad  \cdot f_2 \left(x+\frac{p}{2^{k+m}}+\frac{2\iota}{2^{\frac{m}{2}+2k+\widetilde{\epsilon_2}m}}-\frac{2v}{2^k} \right) e^{3i \cdot 2^k q \widetilde{\lambda} \left(x-\frac{2t}{2^k}+\frac{p}{2^{k+m}}-\frac{q^2}{2^{m+2k}} \right)v} dx \bigg|^2 dtdv.
\end{eqnarray*}
Fix a choice of $\iota \sim 2^{\widetilde{\epsilon_2} m}$ that maximizes the right-hand side of the above estimate and denote $f_2(\cdot):=f_2 \left(\cdot+\frac{2\iota}{2^{\frac{m}{2}+2k+\widetilde{\epsilon_2}m}} \right)$; with these, we see that  
\begin{eqnarray*}
&& \widetilde{\calB}^2 \lesssim 2^{\widetilde{\epsilon_2}m} \sum_{p \in \left[ \widetilde{p} 2^{m-k}, (\widetilde{p}+2)2^{m-k} \right]} \iint_{\left(9I_0^{0, k+m} \right)^2} \vast | \int_{3 I_0^{m, k}} f_2 \left(x+\frac{p}{2^{k+m}} \right) f_2 \left(y+\frac{p}{2^{k+m}} \right) f_2 \left(x+\frac{p}{2^{k+m}}-\frac{2v}{2^k} \right) \\
&& \quad \cdot f_2 \left(y+\frac{p}{2^{k+m}}-\frac{2v}{2^k} \right) \left[\int_{I_0^{m, k}} \left( \sum_{q \sim 2^{\frac{m}{2}}} e^{3i \cdot 2^k v \left[\widetilde{\lambda} \left(x-\frac{2t}{2^k}+\frac{p}{2^{k+m}}-\frac{q^2}{2^{m+2k}} \right)-\widetilde{\lambda} \left(y-\frac{2t}{2^k}+\frac{p}{2^{k+m}}-\frac{q^2}{2^{m+2k}} \right) \right]}\right) dt \right] dv \vast | dxdy \\
&&\lesssim 2^{\widetilde{\epsilon_2}m}  \cdot 
 \widetilde{\calC}\,\widetilde{\calD}, 
\end{eqnarray*}
where 
\begin{eqnarray*}
&& \widetilde{\calC}^2:=\sum_{p \in \left[ \widetilde{p} 2^{m-k}, (\widetilde{p}+2)2^{m-k} \right]} \iiint_{\left(9I_0^{0, k+m} \right)^2 \times 3 I_0^{m, k}}  \Bigg | f_2 \left(x+\frac{p}{2^{k+m}} \right) f_2 \left(y+\frac{p}{2^{k+m}} \right) \\
&& \qquad \quad \cdot  f_2 \left(x+\frac{p}{2^{k+m}}-\frac{2v}{2^k} \right) f_2 \left(x+\frac{p}{2^{k+m}}-\frac{2v}{2^k} \right) f_2 \left(y+\frac{p}{2^{k+m}}-\frac{2v}{2^k} \right) \Bigg|^2 dxdydv
\end{eqnarray*}
and 
\begin{eqnarray*}
&& \widetilde{\calD}^2:=\sum_{p \in \left[ \widetilde{p} 2^{m-k}, (\widetilde{p}+2)2^{m-k} \right]} \iiint_{\left(9I_0^{0, k+m} \right)^2 \times 3I_0^{m, k}} \\
&& \qquad \left|\int_{I_0^{m,k}} \left( \sum_{q \sim 2^{\frac{m}{2}}} e^{3i \cdot 2^k q v \left[\widetilde{\lambda} \left(x-\frac{2t}{2^k}+\frac{p}{2^{k+m}}-\frac{q^2}{2^{m+2k}} \right)-\widetilde{\lambda} \left(y-\frac{2t}{2^k}+\frac{p}{2^{k+m}}-\frac{q^2}{2^{m+2k}} \right) \right]} \right) dt \right|^2 dvdxdy
\end{eqnarray*}
The estimate of $\widetilde{\calC}^2$ follows again easily from the assumption that $f_2$ is uniform. Turning our attention to $\widetilde{D}^2$, we open the square first and then write $q=q'2^{k-\frac{m}{2}}+\ell$ and $q_1=q_1'2^{k-\frac{m}{2}}+\ell_1$:
{\footnotesize \begin{eqnarray*}
&& \widetilde{\calD}^2 \lesssim \sum_{p \in \left[\widetilde{p}2^{m-k}, (\widetilde{p}+2)2^{m-k} \right]} \sum_{q', q_1' \sim 2^{m-k}}  \sum_{\ell, \ell_1 \sim 2^{k-\frac{m}{2}}} \iint_{\left(9I_0^{0, k+m} \right)^2} \iint_{\left(I_0^{m, k} \right)^2} dtdt_1dxdy \\
&& \quad  \left | \int_{3 I_0^{m, k}} e^{3i \cdot 2^k v\left( \substack{(q'2^{k-\frac{m}{2}}+\ell) \left[ \widetilde{\lambda} \left(x-\frac{2t}{2^k}+\frac{p}{2^{k+m}}-\frac{(q')^2}{2^{2m}}-\frac{2q'\ell}{2^{\frac{3m}{2}+k}}-\frac{\ell^2}{2^{m+2k}} \right)-\widetilde{\lambda} \left(y-\frac{2t}{2^k}+\frac{p}{2^{k+m}}-\frac{(q')^2}{2^{2m}}-\frac{2q'\ell}{2^{\frac{3m}{2}+k}}-\frac{\ell^2}{2^{m+2k}} \right)  \right] \qquad \\ \qquad   -(q_1'2^{k-\frac{m}{2}}+\ell_1) \left[ \widetilde{\lambda} \left(x-\frac{2t_1}{2^k}+\frac{p}{2^{k+m}}-\frac{(q_1')^2}{2^{2m}}-\frac{2q_1'\ell_1}{2^{\frac{3m}{2}+k}}-\frac{\ell_1^2}{2^{m+2k}} \right)-\widetilde{\lambda} \left(y-\frac{2t_1}{2^k}+\frac{p}{2^{k+m}}-\frac{(q_1')^2}{2^{2m}}-\frac{2q_1'\ell_1}{2^{\frac{3m}{2}+k}}-\frac{\ell_1^2}{2^{m+2k}} \right) \right] } \right)} dv \right|. 
\end{eqnarray*}} \\
Applying further the change of variables $x \to x+\frac{2q'\ell}{2^{\frac{3m}{2}+k}}+\frac{\ell^2}{2^{m+2k}}+\frac{2t}{2^k}+\frac{\left\{\frac{(q')^2}{2^{m+k}}\right\}}{2^{m+k}}$, $y \to y+\frac{2q'\ell}{2^{\frac{3m}{2}+k}}+\frac{\ell^2}{2^{m+2k}}+\frac{2t}{2^k}+\frac{\left\{\frac{(q')^2}{2^{m+k}}\right\}}{2^{m+k}}$, and $p \to p+\left\lfloor \frac{(q')^2}{2^{m-k}} \right\rfloor$, we see that
{\footnotesize\begin{eqnarray*}
&& \widetilde{\calD}^2 \lesssim \sum_{p \in \left[\widetilde{p} 2^{m-k}, (\widetilde{p}+3)2^{m-k} \right]} \sum_{q', q_1' \sim 2^{m-k}} \sum_{\ell, \ell' \sim 2^{k-\frac{m}{2}}} \iint_{\left(15I_0^{0, k+m} \right)^2} \iint_{\left(I_0^{m, k} \right)^2} dtdt_1dxdy \\
&& \quad \left| \int_{3 I_0^{m, k}} e^{3i \cdot 2^k v \left( \left(q'2^{k-\frac{m}{2}}+\ell \right)\left[\substack{\widetilde{\lambda} \left(x+\frac{p}{2^{k+m}} \right) \quad \\ \quad -\widetilde{\lambda} \left(y+\frac{p}{2^{k+m}} \right)}\right]-\left(q_1'2^{k-\frac{m}{2}}+\ell_1 \right) \left[\substack{\widetilde{\lambda} \left(x+\frac{2(t-t_1)}{2^k}+\frac{p}{2^{k+m}}+\frac{(q')^2-(q_1')^2}{2^{2m}}+\frac{2(q'\ell-q_1'\ell_1)}{2^{\frac{3m}{2}+k}}+\frac{\ell^2-\ell_1^2}{2^{m+2k}} \right) \qquad \\ \qquad - \widetilde{\lambda} \left(y+\frac{2(t-t_1)}{2^k}+\frac{p}{2^{k+m}}+\frac{(q')^2-(q_1')^2}{2^{2m}}+\frac{2(q'\ell-q_1'\ell_1)}{2^{\frac{3m}{2}+k}}+\frac{\ell^2-\ell_1^2}{2^{m+2k}} \right)} \right]   \right)} dv \right|.
\end{eqnarray*}} \\ 
Finally, we apply the change of variable $t \to t_1+s$ and also change back the variables $q=q'2^{k-\frac{m}{2}}+\ell$ and $q_1=q_1'2^{k-\frac{m}{2}}+\ell'$ in order to deduce
\begin{eqnarray} \label{20241031eq01}
&& \widetilde{\calD}^2 \lesssim |I_0^{m, k}|  \sum_{p \in \left[\widetilde{p} 2^{m-k}, (\widetilde{p}+3)2^{m-k} \right]} \sum_{q, q_1 \sim 2^{\frac{m}{2}}}  \iint_{\left(15I_0^{0, k+m} \right)^2} \int_{3I_0^{m, k}} \nonumber  \\
&& \quad \left|\int_{ 3 I_0^{m, k}} e^{3i \cdot 2^{\frac{m}{2}+k} v \left( \frac{q}{2^{\frac{m}{2}}} \left[\substack{\widetilde{\lambda} \left(x+\frac{p}{2^{k+m}} \right) \quad \\ \quad -\widetilde{\lambda} \left(y+\frac{p}{2^{k+m}} \right)}\right]-\frac{q_1}{2^{\frac{m}{2}}} \left[ \substack{ \widetilde{\lambda} \left(x+\frac{2s}{2^k}+\frac{p}{2^{k+m}}-\frac{q_1^2-q^2}{2^{m+2k}} \right) \qquad \\ \qquad - \widetilde{\lambda} \left(y+\frac{2s}{2^k}+\frac{p}{2^{k+m}}-\frac{q_1^2-q^2}{2^{m+2k}} \right) } \right] \right)} dv \right| dsdxdy \nonumber\\
&&\lesssim  |I_0^{m, k}|^2 \sum_{p \in \left[\widetilde{p} 2^{m-k}, (\widetilde{p}+3)2^{m-k} \right]} \sum_{q, q_1 \sim 2^{\frac{m}{2}}}  \iint_{\left(15I_0^{0, k+m} \right)^2} \int_{3I_0^{m, k}} \left\lfloor \widetilde{\Xi}_{x, y, s, p}(q, q_1) \right\rfloor dsdxdy, 
\end{eqnarray}
where 
$$
\widetilde{\Xi}_{x, y, s, p}(q, q_1):=\left[\substack{\widetilde{\lambda} \left(x+\frac{p}{2^{k+m}} \right) \qquad \\ \qquad  -\widetilde{\lambda} \left(y+\frac{p}{2^{k+m}} \right)}\right]-\frac{q_1}{q} \left[ \substack{ \widetilde{\lambda} \left(x+\frac{2s}{2^k}+\frac{p}{2^{k+m}}-\frac{q_1^2-q^2}{2^{m+2k}} \right) \qquad \quad \\ \qquad \quad - \widetilde{\lambda} \left(y+\frac{2s}{2^k}+\frac{p}{2^{k+m}}-\frac{q_1^2-q^2}{2^{m+2k}} \right) } \right]. 
$$
Once here, we notice that the desired control on the term \eqref{20241031eq01} follows using similar reasonings to the ones employed for treating \eqref{20240110eq21} thus concluding our proof.
\bigskip 

\section{Treatment of the stationary off-diagonal component $BHC^{\not \Delta, S}$} \label{20250308sec01}

In this section we focus our attention on the behavior of the stationary off-diagonal component  $BHC^{\not \Delta, S}$. Recall--see \eqref{offdiag}--that $BHC^{\not \Delta, S}$ addresses the regime where one of the following situations is satisfied:
\begin{enumerate}
    \item [(a)] $j \cong l$ and $j>m+100$;
    \item [(b)] $l \cong m$ and $l>j+100$;
    \item [(c)] $m \cong j$ and $m>l+100$.
\end{enumerate}

As before, we only focus on the case $k \ge 0$. We divide our approach into two major cases:
\begin{enumerate}
    \item [({\bf O1})] $j \simeq l, k \ge 0$, and $0 \le m \le \frac{j}{2}$;
    \item [({\bf O2})] all other stationary off-diagonal cases except ({\bf O1}). 
\end{enumerate}
As a brief preview of our strategy in this section we have: 
\begin{itemize}
\item for the Case ({\bf O1}), the phase function $\lambda(x)$ oscillates at the scale $2^{m+3k}$, which implies that the term $\lambda(x)t^3$ is essentially constant on $t$-intervals of length $2^{-m-k}$, which is larger than $2^{-\frac{j}{2}-k}$ -- the linearization scale for the $t$-variable; as a consequence  Rank-I LGC will be effective in the first case;
\item for the Case ({\bf O2}), the phase $\lambda(x) t^3$ oscillates faster, and thus in this situation we will apply Rank-II LGC by extending the earlier argument that we used to treat the main diagonal term.
\end{itemize}    

\subsection{Treatment of the Case ({\bf O1}): $j \simeq l, k \ge 0$, and $0 \le m \le \frac{j}{2}$}

In this context, without loss of generality, we may assume that $j=l$ and hence our main term becomes
\begin{equation} \label{20241124eq01}
\sup_{\lambda \in \R} \left|BHC_{j, j, m, \lambda(x)}^k (f_1, f_2)(x) \right|=\sup_{\lambda \in \R} \left| \one_{I^k}(x) 2^k \phi \left(\frac{\lambda}{2^{m+3k}} \right) \int_{I^k} f_{1, j+k}(x-t) f_{2, j+2k}(x+t^2)e^{i\lambda t^3} dt \right|, 
\end{equation} 
where we recall here that $f_{i, j+ik}$ is the Fourier projection of the function $f_i$ within the frequency interval $[2^{j+ik},\,2^{j+ik+1}]$,  $i \in \{1, 2\}$. Taking again $\lambda(\cdot)$ the measurable function that maximizes the right-hand side of \eqref{20241124eq01} we consider the dual form
\begin{eqnarray} \label{20250308eq03}
&& \Lambda_{j, j, m, \lambda(\cdot)}^k(f_1, f_2, f_3):=\left \langle BHC_{j, j, m, \lambda(x)}^k(f_1, f_2), f_3 \right \rangle  \nonumber \\
&& =2^k\int_{I^k} \int_{I^k} f_{1, j+k}(x-t)f_{2, j+2k}(x+t^2)f_3(x)e^{i \lambda(x) t^3} dtdx \\
&&=2^k \sum_{\widetilde{q} \sim 2^m} \int_{I^k} \int_{I_{\widetilde{q}}^{0, m+k}} f_{1, j+k}(x-t)f_{2, j+2k}(x+t^2)f_3(x)e^{i\lambda(x)t^3} dtdx. \nonumber 
\end{eqnarray}
Preparing the ground for the Rank-I LGC method, we first observe that for $t \in I_{\widetilde{q}}^{0, m+k}$, one has
\begin{equation}\nonumber
\lambda(x)t^3=\frac{3\lambda(x)\,t\,\widetilde{q}^2}{2^{2m+2k}} - \frac{2\lambda(x) \widetilde{q}^3}{2^{3m+3k}}+O(2^{-m})=\frac{\lambda(x) \widetilde{q}^3}{2^{3m+3k}}+O(1). 
\end{equation}
Thus,
$$
\Lambda_{j, j, m, \lambda(\cdot)}^k(f_1, f_2, f_3)\approx 2^k \sum_{\widetilde{q} \sim 2^m} \int_{I^k} e^{\frac{i\lambda(x) \widetilde{q}^3}{2^{3m+3k}}} \left( \int_{I_{\widetilde{q}}^{0, m+k}} f_{1, j+k}(x-t)f_{2, j+2k}(x+t^2)f_3(x) dt \right) dx. 
$$
We further split of the interval $I_{\widetilde{q}}^{0, m+k}$ into intervals of length $2^{-\frac{j}{2}-k}$:
$$
\left[\frac{\widetilde{q}}{2^{m+k}}, \frac{\widetilde{q}+1}{2^{m+k}} \right]=\left[\frac{2^{\frac{j}{2}-m} \widetilde{q}}{2^{\frac{j}{2}+k}}, \frac{2^{\frac{j}{2}-m}(\widetilde{q}+1)}{2^{\frac{j}{2}+k}} \right]=\bigcup_{\ell=0}^{2^{\frac{j}{2}-m}-1} \left[\frac{2^{\frac{j}{2}-m} \widetilde{q}+\ell}{2^{\frac{j}{2}+k}}, \frac{2^{\frac{j}{2}-m} \widetilde{q}+\ell+1}{2^{\frac{j}{2}+k}} \right].
$$
With these one further has
\begin{equation}\nonumber
\Lambda_{j, j, m, \lambda(\cdot)}^k(f_1, f_2, f_3)\approx2^k \sum_{\widetilde{q} \sim 2^m} \sum_{\substack{p \sim 2^{\frac{j}{2}+k} \\ \ell \sim 2^{\frac{j}{2}-m}}} \int_{I_p^{j, 2k}} e^{\frac{i\lambda(x) \widetilde{q}^3}{2^{3m+3k}}} \left( \int_{I_{2^{\frac{j}{2}-m} \widetilde{q}+\ell}^{j, k}} f_{1, j+k}(x-t)f_{2, j+2k}(x+t^2)f_3(x) dt \right) dx. 
\end{equation}
We are now ready to apply Rank-I LGC: at \textsf{Step 1}, for $f_1(x-t)$ we will use spatial intervals of lengths $2^{-\frac{j}{2}-k}$ while for $f_2(x+t^2)$ we will involve spatial intervals of lengths $2^{-\frac{j}{2}-2k}$, respectively. In accordance with this, at \textsf{Step 2}, we apply a \emph{Gabor frame decomposition} as follows:
$$
f_r=\sum_{\substack{p_r \sim 2^{\frac{j}{2}+(r-1)k} \\ n_r \sim 2^{\frac{j}{2}}}} \left \langle f_r, \phi_{p_r, n_r}^{j, rk} \right \rangle \phi_{p_r, n_r}^{j, rk}, \quad r \in \{1, 2\}, 
$$ 
where $\phi_{p_r, n_r}^{j, rk}(x)=2^{\frac{j}{4}+\frac{rk}{2}} \phi \left(2^{\frac{j}{2}+rk}x-p_r \right) e^{i2^{\frac{j}{2}+rk}xn_r}$. This allows one to write 
\begin{eqnarray*}
\Lambda_{j, j, m, \lambda(\cdot)}^k(f_1, f_2, f_3) &\approx & 2^k \sum_{\widetilde{q} \sim 2^m} \sum_{\substack{p \sim 2^{\frac{j}{2}+k} \\ \ell \sim 2^{\frac{j}{2}-m}}} \sum_{\substack{p_1 \sim 2^{\frac{j}{2}} \\ n_1 \sim 2^{\frac{j}{2}}}} \sum_{\substack{p_2 \sim 2^{\frac{j}{2}+k} \\ n_2 \sim 2^{\frac{j}{2}}}} \left \langle f_1, \phi_{p_1, n_1}^{j, k} \right \rangle \left \langle f_2, \phi_{p_2, n_2}^{j, 2k} \right \rangle \\
&& \quad \cdot \int_{I_p^{j, 2k}} f_3(x) e^{\frac{i\lambda(x) \widetilde{q}^3}{2^{3m+3k}}} \left[\int_{I_{2^{\frac{j}{2}-m} \widetilde{q}+\ell}^{j, k}} \phi_{p_1, n_1}^{j, k}(x-t) \phi_{p_2, n_2}^{j, 2k}(x+t^2) dt \right] dx. 
\end{eqnarray*} 
Since $x \in I_p^{j, 2k}$ and $t \in I_{2^{\frac{j}{2}-m} \widetilde{q}+\ell}^{j, k}$, we see that $p_1 \cong \frac{p}{2^k}-\left(2^{\frac{j}{2}-m}\widetilde{q}+\ell \right)$ and $p_2 \cong p+\frac{\left(2^{\frac{j}{2}-m} \widetilde{q}+\ell \right)^2}{2^{\frac{j}{2}}}$. Therefore,  
\begin{eqnarray*}
\Lambda_{j, j, m, \lambda(\cdot)}^k(f_1, f_2, f_3) 
&\approx & 2^{\frac{j}{4}+\frac{3k}{2}} \sum_{\widetilde{q} \sim 2^m} \sum_{\substack{p \sim 2^{\frac{j}{2}+k} \\ \ell \sim 2^{\frac{j}{2}-m}}} \sum_{\substack{n_1 \sim 2^{\frac{j}{2}} \\ n_2 \sim 2^{\frac{j}{2}}}} \left \langle f_1, \phi_{\frac{p}{2^k}-\left(2^{\frac{j}{2}-m}\widetilde{q}+\ell \right), n_1}^{j, k} \right \rangle \left \langle f_2, \phi_{ p+\frac{\left(2^{\frac{j}{2}-m} \widetilde{q}+\ell \right)^2}{2^{\frac{j}{2}}}, n_2}^{j, 2k} \right \rangle \\
&& \quad \cdot \int_{I_p^{j, 2k}} 2^{\frac{j}{4}+k} f_3(x)  e^{\frac{i\lambda(x) \widetilde{q}^3}{2^{3m+3k}}} \cdot e^{i2^{\frac{j}{2}+2k} \left(n_2+\frac{n_1}{2^k} \right)x} \left[ \int_{I_{2^{\frac{j}{2}-m}\widetilde{q}+\ell}^{j, k}} e^{i \left(-2^{\frac{j}{2}+k} n_1 t+2^{\frac{j}{2}+2k}n_2t^2 \right)}dt \right]dx.
\end{eqnarray*}
Observe that on the $t$-interval $I_{2^{\frac{j}{2}-m} \widetilde{q}+\ell}^{j, k}$, one has 
$$
2^{\frac{j}{2}+2k}n_2t^2 \simeq 2^{\frac{j}{2}+2k} n_2\left(\frac{2t \left(2^{\frac{j}{2}-m} \widetilde{q}+\ell \right)}{2^{\frac{j}{2}+k}}-\frac{\left(2^{\frac{j}{2}-m} \widetilde{q}+\ell \right)^2}{2^{j+2k}}\right)+O(1),
$$
which, together with a change of variables, yields that
\begin{eqnarray*} 
\Lambda_{j, j, m, \lambda(\cdot)}^k(f_1, f_2, f_3) & \approx & 2^{\frac{j}{4}+\frac{3k}{2}} \sum_{\widetilde{q} \sim 2^m} \sum_{\substack{p \sim 2^{\frac{j}{2}+k} \\ \ell \sim 2^{\frac{j}{2}-m}}} \sum_{\substack{n_1 \sim 2^{\frac{j}{2}} \\ n_2 \sim 2^{\frac{j}{2}}}} \left \langle f_1, \phi_{\frac{p}{2^k}-\left(2^{\frac{j}{2}-m}\widetilde{q}+\ell \right), n_1}^{j, k} \right \rangle \left \langle f_2, \phi_{ p+\frac{\left(2^{\frac{j}{2}-m} \widetilde{q}+\ell \right)^2}{2^{\frac{j}{2}}}, n_2}^{j, 2k} \right \rangle \nonumber \\
&& \cdot \int_{I_p^{j, 2k}} 2^{\frac{j}{4}+k} f_3(x)  e^{\frac{i\lambda(x) \widetilde{q}^3}{2^{3m+3k}}} \cdot e^{i2^{\frac{j}{2}+2k} \left(n_2+\frac{n_1}{2^k} \right)x} \cdot \frac{C_{\widetilde{q}, \ell, n_2, j}}{2^{\frac{j}{2}+k}} \left[ \int_0^1  e^{i \left(-n_1+n_2 \cdot \frac{2 \left(2^{\frac{j}{2}-m}\widetilde{q}+\ell \right)}{2^{\frac{j}{2}}} \right)t} dt \right] dx.
\end{eqnarray*}
Now, as usual, we maximize the contribution from the $t$-integral, which further leads to 
\begin{eqnarray} \label{20250209eq01}
\left| \Lambda_{j, j, m, \lambda(\cdot)}^k(f_1, f_2, f_3)\right| &\lesssim& 2^{-\frac{j}{4}+\frac{k}{2}} \sum_{\substack{ \widetilde{q} \sim 2^m \\ p \sim 2^{\frac{j}{2}+k}}} \sum_{\substack{\ell \sim 2^{\frac{j}{2}-m} \\ n_2 \sim 2^{\frac{j}{2}}}} \left| \left \langle f_1, \phi_{\frac{p}{2^k}-\left(2^{\frac{j}{2}-m}\widetilde{q}+\ell \right), 2n_2 \left(\frac{\widetilde{q}}{2^m}+\frac{\ell}{2^{\frac{j}{2}}} \right)}^{j, k} \right \rangle \right| \left|  \left \langle f_2, \phi_{ p+\frac{\left(2^{\frac{j}{2}-m} \widetilde{q}+\ell \right)^2}{2^{\frac{j}{2}}}, n_2}^{j, 2k} \right \rangle \right| \nonumber \\
&&\cdot \left| \left \langle f_3^{\widetilde{q}}, \phi_{p, n_2 \left(1+\frac{2\widetilde{q}}{2^{m+k}}+\frac{2\ell}{2^{\frac{j}{2}+k}} \right)}^{j, 2k} \right \rangle \right|\,,
\end{eqnarray}
where here for each $\widetilde{q} \sim 2^m$ we set $f_3^{\widetilde{q}}(x):=f_3(x)e^{\frac{i\lambda(x) \widetilde{q}^3}{2^{3m+3k}}}$.

Once at this point, we define the \emph{sparse index set} as follows: for $\mu>0$ sufficiently small, 
$$
\textsf{S}_\mu(f_i):=\left\{(p_i, n_i): \left|\left \langle f_i, \phi_{p_i, n_i}^{j, ik} \right \rangle \right|^2 \gtrsim 2^{-\mu j} \int_{I_{p_i}^{j, ik}} |f_i|^2,  \quad  p_i \sim 2^{\frac{j}{2}+(i-1)k}, \; n_i \sim 2^{\frac{j}{2}} \right\}, \quad i \in \{1, 2\}.
$$
Define also the \emph{uniform index set} $\textsf{U}_\mu(f_i)$ to be the complement of the set $\textsf{S}_\mu(f_i), i\in\{1, 2\}$. We now decompose the input functions $f_1$ and $f_2$, as usual, into their \emph{uniform} and \emph{sparse} components:  
$$
f_i^{\textsf{U}_\mu}:=\sum_{(p_i, n_i) \in \textsf{U}_\mu(f_i)} \one_{I_{p_i}^{i, ik}} f_i^{\omega_i^{n_i}} \quad \textrm{and} \quad f_i^{\textsf{S}_\mu}:=\sum_{(p_i, n_i) \in \textsf{S}_\mu(f_i)} \one_{I_{p_i}^{i, ik}} f_i^{\omega_i^{n_i}}, \qquad i\in\{1, 2\}, 
$$
where $\omega_i^{n_i}:=\left[n_i 2^{\frac{j}{2}+ik}, (n_i+1)2^{\frac{j}{2}+ik} \right]$.

With these done, we split $ \Lambda_{j, j, m, \lambda(\cdot)}^k(f_1, f_2, f_3)$ into its uniform and sparse components as follows:
$$
\Lambda_{j, j, m, \lambda(\cdot)}^{k, \textsf{S}_\mu}(f_1, f_2, f_3):=\Lambda_{j, j, m, \lambda(\cdot)}^k (f_1^{\textsf{S}_\mu}, f_2^{\textsf{S}_\mu}, f_3) 
$$
and
$$
\Lambda_{j, j, m, \lambda(\cdot)}^{k, \textsf{U}_\mu}(f_1, f_2, f_3):=\Lambda_{j, j, m, \lambda(\cdot)}^k(f_1, f_2, f_3)-\Lambda_{j, j, m, \lambda(\cdot)}^{k, \textsf{S}_\mu}(f_1, f_2, f_3).
$$
In the uniform case -- \emph{i.e.}, if one of the functions $f_1$ or $f_2$ is uniform--assume without loss of generality that, $f_1=f_1^{\textsf{U}_\mu}$. Departing from \eqref{20250209eq01}, we first apply Cauchy-Schwarz in $n_2$, followed by a (backward) change of variable $q=2^{\frac{j}{2}-m}\widetilde{q}+\ell$, and finally by an application of Cauchy-Schwarz first in $q$ and then in $p$:
\begin{eqnarray*}
&& \left| \Lambda_{j, j, m, \lambda(\cdot)}^k(f^{\textsf{U}_\mu}_1, f_2, f_3)\right| \lesssim 2^{-\frac{j}{4}+\frac{k}{2}-\frac{\mu j}{2}} \sum_{\substack{ \widetilde{q} \sim 2^m \\ p \sim 2^{\frac{j}{2}+k}}} \sum_{\substack{\ell \sim 2^{\frac{j}{2}-m} \\ n_2 \sim 2^{\frac{j}{2}}}} \left\|f_1 \right\|_{L^2 \left(I^{j, k}_{\frac{p}{2^k}-\left(2^{\frac{j}{2}-m}\widetilde{q}+\ell \right)} \right)} \\
&& \qquad \qquad \qquad \qquad \qquad \qquad \qquad \cdot \left|  \left \langle f_2, \phi_{ p+\frac{\left(2^{\frac{j}{2}-m} \widetilde{q}+\ell \right)^2}{2^{\frac{j}{2}}}, n_2}^{j, 2k} \right \rangle \right| \left| \left \langle f_3^{\widetilde{q}}, \phi_{p, n_2 \left(1+\frac{2\widetilde{q}}{2^{m+k}}+\frac{2\ell}{2^{\frac{j}{2}+k}} \right)}^{j, 2k} \right \rangle \right| \\
&& \quad \lesssim 2^{-\frac{j}{4}+\frac{k}{2}-\frac{\mu j}{2}} \sum_{\substack{ \widetilde{q} \sim 2^m \\ p \sim 2^{\frac{j}{2}+k} \\ \ell \sim 2^{\frac{j}{2}-m}}} \left\|f_1 \right\|_{L^2 \left(I^{j, k}_{\frac{p}{2^k}-\left(2^{\frac{j}{2}-m}\widetilde{q}+\ell \right)} \right)} \left\|f_2 \right\|_{L^2 \left(I_{ p+\frac{\left(2^{\frac{j}{2}-m} \widetilde{q}+\ell \right)^2}{2^{\frac{j}{2}}}}^{j, 2k} \right)} \left\| f_3^{\widetilde{q}} \right\|_{L^2 \left(I_p^{j, 2k} \right)} \\
&& \quad \lesssim 2^{-\frac{j}{4}+\frac{k}{2}-\frac{\mu j}{2}} \sum_{\substack{p \sim 2^{\frac{j}{2}+k} \\ q \sim 2^{\frac{j}{2}}}}  \left\|f_1 \right\|_{L^2 \left(I^{j, k}_{\frac{p}{2^k}-q} \right)} \left\|f_2 \right\|_{L^2 \left(I_{ p+\frac{q^2}{2^{\frac{j}{2}}}}^{j, 2k} \right)} \left\| f_3 \right\|_{L^2 \left(I_p^{j, 2k} \right)} \\
&& \quad \lesssim 2^{\frac{k}{2}-\frac{\mu j}{2}} \left\|f_1 \right\|_{L^2(I^k)}\left\|f_2 \right\|_{L^2(I^k)}\left\|f_3 \right\|_{L^2(I^k)}.
\end{eqnarray*}
In the sparse case--\emph{i.e.} if both $f_1$ and $f_2$ are sparse--from \eqref{20250209eq01} we have
\begin{align*} 
& \left| \Lambda_{j, j, m, \lambda(\cdot)}^k(f^{\textsf{S}_\mu}_1, f^{\textsf{S}_\mu}_2, f_3)\right| \lesssim 2^{-\frac{j}{4}+\frac{k}{2}} \cdot 2^{2\mu j} \cdot \left(\# \sTFC \right)^{\frac{1}{2}} \\
& \quad \cdot \left[ \sum_{\substack{ \widetilde{q} \sim 2^m \\ p \sim 2^{\frac{j}{2}+k}}} \sum_{\substack{\ell \sim 2^{\frac{j}{2}-m} \\ n_2 \sim 2^{\frac{j}{2}}}}   \left\|f_1 \right\|^2_{L^2 \left(I^{j, k}_{\frac{p}{2^k}-\left(2^{\frac{j}{2}-m}\widetilde{q}+\ell \right)} \right)} \left\|f_2 \right\|^2_{L^2 \left(I_{ p+\frac{\left(2^{\frac{j}{2}-m} \widetilde{q}+\ell \right)^2}{2^{\frac{j}{2}}}}^{j, 2k} \right)} \left| \left \langle f_3^{\widetilde{q}}, \phi_{p, n_2 \left(1+\frac{2\widetilde{q}}{2^{m+k}}+\frac{2\ell}{2^{\frac{j}{2}+k}} \right)}^{j, 2k} \right \rangle \right|^2 \right]^{\frac{1}{2}} \nonumber \\
& \lesssim 2^{-\frac{j}{4}+\frac{k}{2}} \cdot 2^{2\mu j} \cdot \left(\# \sTFC \right)^{\frac{1}{2}} \nonumber \\
& \quad \cdot \left[ \sum_{\substack{ \widetilde{q} \sim 2^m \\ p \sim 2^{\frac{j}{2}+k} \\ \ell \sim 2^{\frac{j}{2}-m}}}  \left\|f_1 \right\|^2_{L^2 \left(I^{j, k}_{\frac{p}{2^k}-\left(2^{\frac{j}{2}-m}\widetilde{q}+\ell \right)} \right)} \left\|f_2 \right\|^2_{L^2 \left(I_{ p+\frac{\left(2^{\frac{j}{2}-m} \widetilde{q}+\ell \right)^2}{2^{\frac{j}{2}}}}^{j, 2k} \right)} \left\| f_3 \right\|_{L^2 \left(I_p^{j, 2k} \right)}^2 \right]^{\frac{1}{2}} 
\end{align*}
where 
$$
\sTFC:=\left\{ \substack{(p, \widetilde{q}, l) \in \Z \times \Z \times \Z  \\ \\ p \sim 2^{\frac{j}{2}+k}, \; \widetilde{q} \sim 2^m \; \ell \sim 2^{\frac{j}{2}-m}}:  L_1 \left(\frac{p}{2^k}-l-2^{\frac{j}{2}-m}\widetilde{q} \right) \cong 2\left( \frac{\widetilde{q}}{2^m}+\frac{\ell}{2^{\frac{j}{2}}} \right)L_2 \left(p+\frac{\left(2^{\frac{j}{2}-m} \widetilde{q}+\ell \right)^2}{2^{\frac{j}{2}}} \right)   \right\},
$$
with $L_i: \left[2^{\frac{j}{2}+(i-1)k}, 2^{\frac{j}{2}+(i-1)k+1} \right] \cap \Z \mapsto \left[2^{\frac{j}{2}}, 2^{\frac{j}{2}+1} \right] \cap \Z, i \in \{1, 2\}$ being some measurable functions. 

Our desired decay bound follows from the control over the size of $\sTFC$; indeed, one can verify that there exists some $\widetilde{\epsilon}>0$ such that 
$$
\# \sTFC \lesssim 2^{j+k-\widetilde{\epsilon} j}. 
$$
Consequently, the estimates for both uniform and sparse components can be summarized as follows: there exists $\widetilde{\epsilon}>0$, such that
$$
\left| \Lambda_{j, j, m, \lambda(\cdot)}^{k, \textsf{U}_\mu}(f_1, f_2, f_3)\right| \lesssim 2^{-\frac{\mu}{2} j} \prod_{r=1}^3 \left\|f_r \right\|_{L^{p_j} \left(3I^k \right)} \quad \textrm{and} \quad \left| \Lambda_{j, j, m, \lambda(\cdot)}^{k, \textsf{S}_\mu}(f_1, f_2, f_3)\right| \lesssim 2^{(2\mu-\widetilde{\epsilon}) j} \prod_{r=1}^3 \left\|f_r \right\|_{L^{p_j} \left(3I^k \right)},
$$
and therefore, we conclude that
\begin{equation} \label{20250210eq01}
\left|\Lambda_{j, j, m, \lambda(\cdot)}^k(f_1, f_2, f_3) \right| \lesssim 2^{-\frac{\widetilde{\epsilon}}{5} j} \prod_{r=1}^3 \left\|f_r \right\|_{L^{p_j} \left(3I^k \right)},
\end{equation} 
whenever $(p_1, p_2, p_3)=(2, 2, \infty), (2, \infty, 2)$ or $(\infty, 2, 2)$. 

\medskip

\subsection{Treatment of Case ({\bf O2}): Other stationary off-diagonal cases}

Our goal in this section is to extend Theorem \ref{20241018mainthm01} to a slightly more general set-up that also covers the other stationary off-diagonal cases described in Case ({\bf O2}). Indeed, the main result of our section is

\begin{thm} \label{generalizedmain}
For any $k, m \in \N$, consider
$$
\widetilde{\Lambda}^k_m(f_1, f_2, f_3):=2^k \int_{I^k} \int_{I^k} f_1(x-t)f_2(x+t^2)f_3(x)e^{i\lambda(x) t^3} dtdx, 
$$
where
\begin{enumerate}
    \item [$\bullet$] $\lambda(x) \in \left[2^{\frac{m}{2}+\kappa m+3k}, 2^{m+3k+1}\right]$ for some $\kappa>0$;
    \item [$\bullet$] one of the following situations holds:
    \begin{enumerate}
        \item [(1)] $\supp \ \widehat{f_1} \subseteq \left[2^k, 2^{m+k+1} \right]$ and $\supp \ \widehat{f_2} \subseteq \left[2^{m+2k}, 2^{m+2k+1} \right]$;
        \item [(2)] $\supp \ \widehat{f_1} \subseteq \left[2^{m+k}, 2^{m+k+1} \right]$ and $\supp \ \widehat{f_2} \subseteq \left[2^{2k}, 2^{m+2k+1} \right]$.
    \end{enumerate}
\end{enumerate}
Moreover, if either $\supp \ \widehat{f_1}$ is away from $[2^{-100} \cdot 2^{m+k}, 2^{m+k+1}]$ or $\supp \ \widehat{f_2}$ is away from $[2^{-100} \cdot 2^{m+2k}, 2^{m+2k+1}]$, then one may further impose $\lambda \in [2^{m+3k}, 2^{m+3k+1}]$.  

Under the above assumptions, there exists some $\epsilon>0$, such that 
$$
\left| \widetilde{\Lambda}^k_m(f_1, f_2, f_3) \right| \lesssim 2^{-\epsilon \min \{2k, m\}} \prod_{j=1}^3 \left\|f_j \right\|_{L^{p_j}(3I^k)}, 
$$
whenever $(p_1, p_2, p_3)=(2, \infty, 2)$, $(\infty, 2, 2)$ or $(2, 2, \infty)$. 
\end{thm}

\subsubsection{Deducing Case ({\bf O2}) from Theorem \ref{generalizedmain}.} In this brief subsection we provide the proof sketch for how to obtain Case ({\bf O2}) if one assumes the veridicity of Theorem \ref{generalizedmain}.

\medskip

\noindent \underline{\textsf{Case 1: $j \cong \ell, k \ge 0$, and $\frac{j}{2} \le m \le j-100$.}} We start by noticing that without loss of generality, we may assume $j=l$ and also that in this situation, the role of the parameter $m$ in Theorem \ref{generalizedmain} is now replaced by $j$. Moreover, using estimate \eqref{20250210eq01} derived from the Case ({\bf O1}) and a standard continuity argument  (see, again \cite[Appendix 13.3]{HL23}), we may also assume that 
$$
\frac{j}{2}+\kappa j \le m \le j-100
$$
for some $\kappa>0$ sufficiently small (note that estimate \eqref{20250210eq01} is independent of $k$). Therefore, the main term we have to control in this case is given by
$$
\Lambda_{j, j, m, \lambda(\cdot)}^k(f_1, f_2, f_3)=2^k \int_{I^k} \int_{I^k} f_{1, j+k}(x-t)f_{2, j+2k}(x+t^2)f_3(x)e^{i \lambda(x) t^3} \phi \left(\frac{\lambda(x)}{2^{m+3k}} \right) dtdx. 
$$
Taking now
$$
\widetilde{\Lambda}_j^{k, 1}(f_1, f_2, f_3):=\sum_{\frac{j}{2}+\kappa j \le m \le j-100} \Lambda_{j, j, m, \lambda(\cdot)}^k(f_1, f_2, f_3),
$$
it is easy to notice that Theorem \ref{generalizedmain} applies. 

\medskip 

\noindent \underline{\textsf{Case 2: $l \cong m, k \ge 0$, and $m>j+100$.}} Without loss of generality, we may assume $l=m$, and hence the main term to understand is given by
$$
\Lambda_{j, m, m, \lambda(\cdot)}^k(f_1, f_2, f_3)=2^k \int_{I^k} \int_{I^k} f_{1, j+k}(x-t)f_{2, m+2k}(x+t^2)f_3(x)e^{i \lambda(x) t^3} \phi \left(\frac{\lambda(x)}{2^{m+3k}} \right) dtdx. 
$$
Note that to apply Theorem \ref{generalizedmain}, it suffices to consider
\begin{eqnarray*}
\widetilde{\Lambda}_m^{k, 2}(f_1, f_2, f_3)%
&:=& \sum_{0 \le j \le m-100} \Lambda_{j, m, m, \lambda(\cdot)}^k(f_1, f_2, f_3) \\
&=& 2^k \int_{I^k} \int_{I^k} f_{1, j, k}(x-t)f_{2, m+2k}(x+t^2)f_3(x)e^{i \lambda(x) t^3} \phi \left(\frac{\lambda(x)}{2^{m+3k}} \right) dtdx,
\end{eqnarray*}
where
$$
f_{1, j, k}:= \calF^{-1}\left( \sum_{0 \le j \le m-100} \widehat{f_{1, j+k}} \right).
$$

\medskip 

\noindent \underline{\textsf{Case 3: $m \cong j, k \ge 0$, and $m>l+100$.}} The treatment of this case resembles \textsf{Case 2} above. In this situation, without loss of generality, one considers
\begin{eqnarray} \label{20250210eq20}
\widetilde{\Lambda}_m^{k, 3}(f_1, f_2, f_3)%
&:=& \sum_{0 \le l \le m-100} \Lambda_{m, l, m, \lambda(\cdot)}^k(f_1, f_2, f_3) \nonumber \\
&=& \sum_{0 \le l \le m-100} 2^k \int_{I^k} \int_{I^k} f_{1, m+k}(x-t)f_{2, l+2k}(x+t^2)f_3(x)e^{i \lambda(x) t^3} \phi \left(\frac{\lambda(x)}{2^{m+3k}} \right) dtdx \nonumber \\
&=& 2^k \int_{I^k} \int_{I^k} f_{1, m+k}(x-t)f_{2, m, 2k}(x+t^2)f_3(x)e^{i \lambda(x) t^3} \phi \left(\frac{\lambda(x)}{2^{m+3k}} \right) dtdx,
\end{eqnarray}
where
$$
f_{2, m, 2k}:= \calF^{-1} \left( \sum_{0 \le l \le m-100} \widehat{f_{2, l+2k}} \right).
$$
It is clear that \eqref{20250210eq20} satisfies all the assumptions of Theorem \ref{generalizedmain}.

\subsubsection{Proof of Theorem \ref{generalizedmain}} To this end, we notice that the proof of Theorem \ref{generalizedmain} follows by a straightforward modification of the argument used in the proof of Theorem \ref{20241018mainthm01}. Indeed, carefully  examining the proof of Theorem \ref{20241018mainthm01}, one employs the following three \emph{key} ingredients:
\begin{enumerate}
    \item [(1)] the essential spatial length of constancy for $f_1$ is $2^{-m-k}$;
    \item [(2)] the essential spatial length of constancy of $f_2$ is $2^{-m-2k}$;
    \item [(3)] the height of the re-normalized phase $\widetilde{\lambda}(x):=\frac{\lambda(x)}{2^{\frac{m}{2}+3k}}$ lies in the range $\left[2^{\kappa m}, 2^{\frac{m}{2}} \right]$. 
\end{enumerate}
With these in mind, one can follow the same strategy outlined in Section \ref{20241023overview01}: first, apply the Rank-II LGC to propagate the length of constancy of $\lambda$ to a larger scale; then, use the argument presented in \cite[Theorem 4.3]{HL23} (when $k \ge \frac{m}{2})$) or Rank-I LGC (when $0 \le k \le \frac{m}{2}$) to complete the proof. We leave further details to the interested reader.

\section{Treatment of the low oscillatory component $BHC^{Lo}$} \label{20250228sec02}

In this section, we treat the low oscillatory component $BHC^{Lo}(f_1, f_2):=\sup\limits_{\lambda \in \R} \left| BHC_\lambda^{Lo}(f_1, f_2) \right|$, where we recall that 
 $$
BHC_{\lambda}^{Lo}(f_1, f_2):=\sum_{k \in \Z} \sum_{(j, l, m) \in \Z^3 \backslash \N^3}BHC_{j, l, m, \lambda}^k(f_1, f_2).
$$
We start with the routine linearization procedure by letting $\lambda(\cdot)$ be the measurable function that maximizes the supremum $\sup\limits_{\lambda \in \R} \left| BHC_\lambda^{Lo}(f_1, f_2) \right|$, and therefore, it suffices to bound the operator 
\begin{equation} \label{multdec}
BHC_{\lambda(\cdot)}^{Lo}(f_1, f_2):=\sum_{k \in \Z} \sum_{(j, l, m) \in \Z^3 \backslash \N^3}BHC_{j, l, m, \lambda(\cdot)}^k(f_1, f_2),
\end{equation} 
where we recall that 
\begin{equation} \label{20241104eq01}
BHC_{j, l, m, \lambda(x)}^k(f_1, f_2)(x):=2^k\phi \left(\frac{\lambda(x)}{2^{m+3k}} \right) \int_{\R} f_{1, j+k}(x-t)f_{2, l+2k}(x+t^2)e^{i\lambda(x)t^3} \rho(2^k t) dt.
\end{equation}
Here, $\rho$ is an odd smooth cut-off and $\phi$ is an even and positive smooth cut-off both supported in $\left[-4, 4 \right] \backslash \left[-\frac{1}{4}, \frac{1}{4} \right]$, $f_{1, j+k}=\calF^{-1} \left( \widehat{f_1} \phi \left(\frac{\cdot}{2^{j+k}} \right)\right)$ and $f_{2, l+2k}=\calF^{-1} \left(\widehat{f_2} \phi \left(\frac{\cdot}{2^{l+2k}} \right) \right)$. The main result of this section is given by
\begin{thm}\label{thmin} 
For any $1<p,\,q<\infty$ and $\frac{1}{2}< r<\infty$ with $\frac{1}{p}+\frac{1}{q}=\frac{1}{r}$ the following holds:
\begin{equation}\label{20241106thm01}
\left\|BHC^{Lo}(f_1, f_2)\right\|_r\lesssim_{p,q} \|f_1\|_p\,\|f_2\|_q\,.
\end{equation}
\end{thm}
Our analysis will be split into seven distinct cases depending on the sign of each of the parameters $j,\,l,\,m\in\Z$ appearing in \eqref{multdec}. Indeed, for $*_1,\,*_2,\,*_3\in\{-,\,+\}$, we set
\begin{equation}\label{decompnegat}
BHC_{*_1,\,*_2,\,*_3, \lambda(\cdot)}(f_1, f_2):=\sum_{k\in\Z}\sum_{j\in \Z_{*_1}}\sum_{l\in \Z_{*_2}}\sum_{m\in \Z_{*_3}} BHC_{j, l, m, \lambda(\cdot)}^k(f_1, f_2),
\end{equation}
where $\Omega:=\left\{\left(*_1,\,*_2,\,*_3\right)\,|\, \left(*_1,\,*_2,\,*_3 \right)^3 \in \{-,\,+\}^3\setminus\{\left(+,+,+\right)\right\}$ and list below the seven cases that are of interest for us: 
\begin{itemize}
\item \textbf{Case A}: \textsf{$j<0$, $l<0$, and $m \le 0$;}
\item \textbf{Case B}: \textsf{$j<0$, $l<0$, and $m>0$;}
\item \textbf{Case C}: \textsf{$j<0$, $l\geq 0$, and $m \le 0$;}
\item \textbf{Case D}: \textsf{$j< 0$, $l\geq 0$, and $m> 0$;}
\item \textbf{Case E}: \textsf{$j\geq 0$, $l<0$, and $m \le 0$;}
\item \textbf{Case F}: \textsf{$j\geq 0$, $l<0$, and $m>0$;}
\item \textbf{Case G}: \textsf{$j \geq 0, l \ge 0$, and $m \le 0$.}
\end{itemize}
To this end, we introduce the notation $\psi(x):=\sum\limits_{\ell \in \Z \backslash \N} \phi \left(\frac{x}{2^{\ell}} \right)$, and $\check{\psi}_k(x):=2^k \check{\psi}(2^k x)$. 

\subsection{Treatment of the Cases A, C, E, and G: $m \le 0$}  We notice that in all of these cases $m \le 0$ which together with the support conditions on $\phi$ and $\rho$ yields $\lambda(x) t^3 \lesssim 1$. Therefore, one can put together all the cases {\bf A, C, E}, and {\bf G} and encapsulate them into the study of the following operator:
$$
BHC^{Lo, \le}_{\lambda (\cdot)}(f_1, f_2)(x):=\int_{\R} f_1(x-t)f_2(x+t^2)e^{i\lambda(x)t^3} \chi(\lambda(x), t) \frac{dt}{t},
$$
where 
$$
\chi(\lambda, t):=\sum_{\substack{\ell, k \in \Z \\ \ell+k \ge 0}} \rho(2^{3\ell} \lambda) \rho(2^k t).
$$
Note that $\chi$ enjoys the following properties: 1) $\chi(\lambda, t) \lesssim 1$ for al $\lambda, t \neq 0$, and 2) if $\chi(\lambda, t) \neq 0$, then $|\lambda t^3| \lesssim 1$. Our goal is to show that $BHC^{Lo, \le}_{\lambda (\cdot)}(f_1, f_2)$ satisfies estimate \eqref{20241106thm01}, which is a direct consequence of 
\begin{equation} \label{20250216eq21}
\left|BHC^{Lo, \le}_{\lambda (\cdot)}(f_1, f_2) \right| \lesssim \calH^*_{\calC}(f_1, f_2)+\calM_{\calC}(f_1, f_2)\,,
\end{equation}
where
\begin{enumerate}
    \item [$\bullet$] $\calH^*_{\calC}(f_1, f_2)$ is the \emph{maximal (smooth) truncation of the bilinear Hilbert transform along the curve $(-t, t^2)$}
    $$
    \calH^*_{\calC}(f_1, f_2)(x):=\sup_{r>0} \left| \int_\R f_1(x-t)f_2(x+t^2) \theta \left(\frac{t}{r} \right) \frac{dt}{t} \right|,
    $$
    with $\theta$ being a suitable smooth bump function supported near the origin, and
    \item [$\bullet$] $\calM_{\calC}(f_1, f_2)$ is the \emph{bilinear maximal function along the curve $(t, t^2)$}
    \begin{equation} \label{bilinearmax}
    \calM_{\calC}(f_1, f_2)(x):=\sup_{r>0} \frac{1}{2r} \int_{-r}^r \left|f_1(x-t) \right| \left|f_2(x+t^2) \right|dt.
    \end{equation}
\end{enumerate}
Indeed, notice that if \eqref{20250216eq21} holds, then the desired estimate \eqref{20241106thm01} follows from the existing estimates for these operators (see, e.g., \cite{{GL20}}). Now, to see that \eqref{20250216eq21} holds, we split $BHC^{Lo, \le}_{\lambda (\cdot)}(f_1, f_2)(x)$ as follows:
\begin{align*}
BHC^{Lo, \le}_{\lambda (\cdot)}(f_1, f_2)(x) 
&= \int_{\R} f_1(x-t)f_2(x+t^2)\chi(\lambda(x), t) \frac{dt}{t} \\
& \quad +\int_{\R} f_1(x-t)f_2(x+t^2) \left(e^{i\lambda(x)t^3}-1 \right) \chi(\lambda(x), t) \frac{dt}{t} \\
&:= I+II.
\end{align*}
For the first term $I$, we have 
\begin{align*}
|I| 
&= \left| \int_{\R} f_1(x-t)f_2(x+t^2)\chi(\lambda(x), t) \frac{dt}{t} \right| \\
&=\left| \sum_{\ell \in \Z} \rho (2^{3\ell} \lambda(x)) \int_\R f_1(x-t)f_2(x+t^2)  \theta\left(\frac{t}{2^\ell} \right) \frac{dt}{t} \right| \\
& \le \left( \sum_{\ell \in \Z} \rho \left(2^{3\ell} \lambda(x) \right) \right) \calH_{\calC}^*(f_1, f_2)(x) \lesssim \calH_{\calC}^*(f_1, f_2)(x),
\end{align*}
where here we denote $\theta(t):=\sum_{k \ge 0} \rho(2^k t)$ (a smooth bump function supported near the origin). 

For the second term, we use the basic relations $|e^{is}-1| \le |s|$ and $\chi(\lambda, t)\lesssim \one_{[-C, C]} (\lambda t^3)$ in order to deduce 
\begin{align*}
|II|
& \lesssim \int_{t: |\lambda(x)t^3| \lesssim 1} |f_1(x-t)||f_2(x+t^2)||\lambda(x)t^3| \frac{dt}{|t|} \\
& \lesssim  \sum_{\ell \in \N} 2^{-\ell} \cdot \frac{1}{2^{-\frac{\ell}{3}} |\lambda(x)|^{-\frac{1}{3}}} \int_{\{t: |t| \lesssim 2^{-\frac{\ell}{3}}|\lambda(x)|^{-\frac{1}{3}}\}   } |f_1(x+t)||f_2(x+t^2)|dt \\
& \lesssim \calM_{\calC}(f_1, f_2)(x). 
\end{align*}
This concludes the proof of \eqref{20250216eq21}. 

\subsection{Treatment of Cases B:  $j<0$, $l<0$, and $m>0$}

In this situation, for a given $k\in\Z$, the multiplier associated to $BHC_{-,-,+, \lambda(\cdot)}^k$ is given by
$$
\mathfrak{m}_{-, -, +, \lambda(x)}^k(\xi,\eta)=\sum_{m\in\N} \mathfrak{m}_{-, -, m, \lambda(x)}^k(\xi,\eta), 
$$
where 
$$
\mathfrak{m}_{-, -, m, \lambda(x)}^k(\xi,\eta):=\left[\int_\R e^{i \left(-\frac{\xi}{2^k}t+\frac{\eta}{2^{2k}}t^2+\frac{\lambda(x)}{2^{3k}} t^3 \right)}\rho(t)\,dt\right]\,\psi\left(\frac{\xi}{2^{k}}\right)\,
\psi\left(\frac{\eta}{2^{2k}}\right)\,\phi\left(\frac{\lambda(x)}{2^{m+3k}}\right)\,,
$$
Applying a Taylor series argument followed by an integration by parts, we further have
\begin{eqnarray*}\label{multdecB1}
&&\mathfrak{m}_{-, -, m, \lambda(x)}^k(\xi,\eta) \\
&&=\sum_{\ell_1, \ell_2 \ge 0} \frac{C_{\ell_1, \ell_2}}{\ell_1! \ell_2!} \left(\frac{\xi}{2^k} \right)^{\ell_1} \psi \left(\frac{\xi}{2^k} \right) \left(\frac{\eta}{2^{2k}} \right)^{\ell_2} \psi \left(\frac{\eta}{2^{2k}} \right) \phi \left(\frac{\lambda(x)}{2^{m+3k}} \right) \int_{\R} e^{i \frac{\lambda(x)}{2^{3k}}t^3} t^{\ell_1+2\ell_2} \rho(t)dt \nonumber \\
&& \approx \frac{1}{2^m} \sum_{\ell_1, \ell_2 \ge 0} \frac{\tilde{C}_{\ell_1, \ell_2}}{\ell_1! \ell_2!}\,\left(\frac{\xi}{2^k} \right)^{\ell_1} \psi \left(\frac{\xi}{2^k} \right) \left(\frac{\eta}{2^{2k}} \right)^{\ell_2} \psi \left(\frac{\eta}{2^{2k}} \right) \tilde{\phi} \left(\frac{\lambda(x)}{2^{m+3k}} \right)    \int_{\R} e^{i \frac{\lambda(x)}{2^{3k}}t^3} \left(t^{\ell_1+2\ell_2-2} \rho(t)\right)' dt\,,
\end{eqnarray*}
where $\tilde{C}_{\ell_1, \ell_2}$ and $\tilde{\phi}$ are some objects with features similar to the original $C_{\ell_1, \ell_2}$ and $\phi$, respectively. 

This implies
$$
\left| \sum_{k \in \Z} BHC_{-, -, +, \lambda(\cdot)}^k(f_1, f_2) (x)\right| \lesssim Mf_1(x)Mf_2(x),
$$
which gives the desired estimate \eqref{20241106thm01}. Here $Mf$ refers to the standard Hardy-Littlewood maximal operator.

\subsection{Treatment of Case D: $j< 0$, $l\ge 0$, and $m>0$}  
As before, we start with a Taylor series development for the multiplier: 
\begin{equation} \label{multdecE}
\frakm_{-, l, m, \lambda(x)}^k(\xi, \eta)=\sum_{\ell_1 \ge 0} \frac{C_{\ell_1}}{\ell_1!} \left(\frac{\xi}{2^k} \right)^{\ell_1} \psi \left(\frac{\xi}{2^k} \right) \phi \left(\frac{\eta}{2^{2k+l}} \right) \phi \left(\frac{\lambda(x)}{2^{3k+m}} \right) \int_{\R} e^{i \left(\frac{\eta}{2^{2k}}t^2+\frac{\lambda(x)}{2^{3k}}t^3 \right)} t^{\ell_1}\rho(t)dt
\end{equation} 
Due to the fast decay in the parameters $\ell_1$ it suffices to consider $\ell_1=0$.
\medskip

\noindent \textsf{Subcase E.1: $l\ncong m$.} \quad  In this situation the phase in the integral term of \eqref{multdecE} has no stationary points. Thus, applying an integration by parts argument, we have
\begin{equation}\label{multdecE1}
\frakm_{-, l, m, \lambda(x)}^k(\xi, \eta) \approx  \frac{1}{2^{\max\{l, m\}}} \Big(\int_\R e^{i\,\frac{\eta}{2^{2k}}t^2}\, e^{i\,\frac{\lambda(x)}{2^{3k}}t^3}\,\rho(t)\,dt\Big)\, \psi\Big(\frac{\xi}{2^{k}}\Big) \phi\Big(\frac{\eta}{2^{2k+l}}\Big) \phi\Big(\frac{\lambda(x)}{2^{3k+m}}\Big)\,.
\end{equation}
As a consequence by Cauchy-Schwarz, 
\begin{eqnarray*}
&& \left|BHC_{-,\,+,\,+, \lambda(x)}(f_1,f_2)(x)\right|_{\textnormal{\textsf{E.1}}} \\
&& \approx \left|\sum_{k\in\Z}\sum_{{l,m\in\N}\atop{|l-m|\gtrsim 1}}\,
\frac{1}{2^{\max\{l, m\}}}\,(f_1*\check{\psi}_k)(x)\,
\left(\int_{\R} (f_2*\check{\phi}_{2k+l})\left(x+\frac{t^2}{2^{2k}}\right)\,e^{i\,\frac{\lambda(x)}{2^{3k}}t^3}\,\rho(t)\,dt\right)\,
\phi\Big(\frac{\lambda(x)}{2^{3k+m}}\Big)\right| \\
&&\lesssim\sum_{k\in\Z}\sum_{{l,m\in\N}\atop{|l-m|\gtrsim 1}}\,
\frac{1}{2^{\max\{l, m\}}}\,|f_1*\check{\psi}_k|(x)\,
\left(\int_{\R} \left|(f_2*\check{\phi}_{2k+l})\left(x+\frac{t^2}{2^{2k}}\right)\,\rho(t)\right|\,dt\right)\,
\phi\Big(\frac{\lambda(x)}{2^{3k+m}}\Big)\\
&&\lesssim \sum_{{l,m\in\N}\atop{|l-m|\gtrsim 1}}\,
\frac{1}{2^{\max \{l, m\}}}\,Mf_1(x)\,\left(\int_{\R} \calS^{\phi}_{2, 2^{\frac{l}{2}} t} f_2(x)\,|\rho(t)|\,dt\right), 
\end{eqnarray*}
where $\calS_{\iota, t}^{\phi}f$ is the $t$-shifted square maximal operator defined as 
$$
\calS_{\iota, t}^\phi f(x):=\left( \sum_{k \in \Z} \left| \left(f*\check{\phi}_{\iota k} \right) \left(x+\frac{t^{\iota}}{2^{\iota k}} \right) \right|^2 \right)^{\frac{1}{2}}. 
$$
The desired estimate \eqref{20241106thm01} is then a consequence of the following estimate: 
$$
\left\|\calS_{2, 2^{\frac{l}{2}}t}^{\phi}f \right\|_p \lesssim_{\phi, p} (1+l)^{\frac{2}{p^*}-1} \left\|f\right\|_p,
$$ 
where $p^*=\min\{p, p'\}$ (see, \cite[Proposition 42]{Lie18} or \cite[Lemma 3.3]{GL20}).
\medskip

\noindent \textsf{Subcase E.2: $l\cong m$.} \quad In this situation, for each given $k\in\Z$ and $m\in\N$, the phase in the integral term of \eqref{multdecE} may have a stationary point of magnitude $\simeq 1$. This is the context that naturally connects with the topic of the curved $\gamma$-Carleson operators as discussed in \cite{Lie2024}. More precisely, we have 
\begin{eqnarray*} 
&& |BHC_{-,\,+,\,+, \lambda(x)}(f_1,f_2)(x)|_{\textnormal{\textsf{E.2}}} \nonumber \\
&& \approx \left|\sum_{k\in\Z}\sum_{m\in\N}\,
(f_1*\check{\psi}_k)(x)\,
\left(\int_{\R} (f_2*\check{\phi}_{2k+m}) \left(x+\frac{t^2}{2^{2k}} \right)\,e^{i\,\frac{\lambda(x)}{2^{3k}}t^3}\,\rho(t)\,dt\right)\,
\phi\Big(\frac{\lambda(x)}{2^{3k+m}}\Big)\right|  \nonumber \\
&& \approx \left|\sum_{k\in\Z}\sum_{m\in\N}\,
(f_1*\check{\psi}_k)(x)\,
\left(\int_{\R} (f_2*\check{\phi}_{2k+m})(x+t)\,e^{i\,\lambda(x)\,t^{\frac{3}{2}}}\,2^{2k}\,\tilde{\rho}(2^{2k}\,t)\,dt\right)\,
\phi\Big(\frac{\lambda(x)}{2^{3k+m}}\Big)\right|  \nonumber \\
&& \lesssim \sum_{m\in\N} \left(\sum_{k\in\Z} \left|(f_1*\check{\psi}_k)(x)\,\phi\Big(\frac{\lambda(x)}{2^{3k+m}}\Big)\right|^2\right)^{\frac{1}{2}}
\,\left(\sum_{k\in\Z} \left|\calC_{\gamma, 2k, m}f_2(x) \right|^2\right)^{\frac{1}{2}}  \nonumber \\
&& \lesssim Mf_1(x) \cdot \sum_{m\in\N}  \left(\sum_{k\in\Z} \left|\calC_{\gamma, 2k, m}f_2(x) \right|^2\right)^{\frac{1}{2}}\:,
\end{eqnarray*}
where $\widetilde{\rho}(t):=t^{-1/2}\rho\left(t^{1/2}\right)$, and $\calC_{\gamma, 2k, m} f$ is the $\gamma$-Carleson operator for $\gamma(x, t):=\lambda(x) t^{\frac{3}{2}}$: 
\begin{align} \label{20250218eq01}
\calC_{2k,m}f(x) 
&=\left( \int_{\R} (f*\check{\phi}_{2k+m}) \left(x+\frac{t}{2^{2k}} \right) e^{i \gamma\left(x, \frac{t}{2^{2k}} \right)} \tilde{\rho}(t)\,dt \right) \phi \left(\frac{\lambda(x)}{2^{m+3k}}\right) \nonumber \\
&:=\left( \int_{\R} (f*\check{\phi}_{2k+m}) \left(x+\frac{t}{2^{2k}} \right) e^{i \frac{\lambda(x) t^{\frac{3}{2}}}{2^{3k}}} \tilde{\rho}(t)\,dt \right) \phi \left(\frac{\lambda(x)}{2^{m+3k}}\right).
\end{align} 
Finally, estimate \eqref{20241106thm01} can be now achieved by combining  Minkowski, Fefferman-Stein maximal inequality, interpolation and the following features of the operator \eqref{20250218eq01}: 
\begin{enumerate}
    \item [$\bullet$] (\emph{$L^2$--estimate}) Appealing to \cite[Proposition 48]{Lie2024}, there exists some $\epsilon>0$, such that
$$
\|\calC_{2k,m} f\|_2 \lesssim 2^{-\epsilon m}\,\|f*\check{\phi}_{2k+m}\|_2;
$$
\item [$\bullet$] (\emph{$L^p$--estimate}) For fixed $m \in \N$, one has the trivial bound
$$
\left| \calC_{\gamma, 2k, m} f(x) \right| \lesssim \left(\int_{\R} \calM^{(-2^m t)} f(x) \rho(t) dt\right) \cdot \phi \left(\frac{\lambda(x)}{2^{m+3k}} \right),
$$
where 
$$
\calM^{(2^m t)}f(x)=\sup_{k \in \Z} \left|\left(f*\check{\phi}_k \right) \left(x-\frac{2^m t}{2^k} \right) \right|
$$
enjoys---see, e.g., \cite[Lemma 3.4]{GL20}---the $L^p$ estimate
$$
\left\|\calM^{(2^m t)}f \right\|_p \lesssim (1+m)^{\frac{1}{p}} \left\|f\right\|_p\qquad \textrm{for any}\:\:|t| \sim 1\,. 
$$
\end{enumerate}

We leave further details to the interested reader.

\subsection{Treatment of Case F: $j\geq 0$, $l<0$, and $m>0$}  
The approach in this case is similar to the one in Case D above with just some straightforward changes: the roles of $f_1$ and $f_2$ are interchanged, the term $\frac{t^2}{2^{2k}}$ is replaced in the new context by $\frac{t}{2^{k}}$ and one appeals to the $\gamma$-Carleson operator for $\gamma(x,t)=\lambda(x)\,t^{3}$. Thus, one can immediately verify that \eqref{20241106thm01} holds.
\bigskip 

\section{The maximal quasi-Banach boundedness range of $BHC$} \label{20250309sec01}

In this last section we provide the desired maximal\footnote{Up to end-points.} quasi-Banach range for our main operator $BHC$ as stated in {\bf Main Theorem} \ref{mainresult}.

We start by recalling the main decomposition of our operator: 
\begin{eqnarray*}
BHC(f_1, f_2) &\le&  BHC^{Hi}(f_1, f_2)\,+\,BHC^{Lo}(f_1, f_2) \\
&\le&  \begin{matrix} \underbrace{BHC^{\Delta}(f_1, f_2)} \\ \textnormal{{\bf main diagonal}} \end{matrix}\quad+\quad\begin{matrix} \underbrace{BHC^{\not\Delta, S}(f_1, f_2)} \\ \textnormal{{\bf stationary off-diagonal}}  \end{matrix}\\
&+&\begin{matrix} \underbrace{BHC^{\not\Delta, NS}(f_1, f_2)} \\ \textnormal{{\bf non-stationary off-diagonal}} \end{matrix}\quad + \quad \begin{matrix} \underbrace{BHC^{Lo}(f_1, f_2)} \\  \textnormal{{\bf low oscillatory}} \end{matrix}.
\end{eqnarray*}
Since in Section \ref{20250228sec01} and Section \ref{20250228sec02}, we have established the full quasi-Banach range for both the non-stationary off-diagonal and the low-oscillatory components our main focus in this section will be on the other two remaining components.\footnote{Throughout this whole section, we assume that $k \ge 0$. A sketch of the case $k\le 0$ can be found in the Appendix.}

\subsection{Treatment of the main diagonal component $BHC^{\Delta}$: Tame bounds}\label{20250308sec10}

In this first subsection, we provide slowly growing\footnote{\emph{I.e}, at most polynomially in the parameters $k,\,m\in\N$.} restricted weak type bounds for the sub-components of $BHC^{\Delta}$. We start by recalling that in this case, the main term we have to understand takes the form
\begin{equation}\nonumber
\Lambda_m^k(f_1, f_2, f_3)= \int_{\R^2} (f_1*\check{\phi}_{m+k}) \left(x-\frac{t}{2^k} \right) (f_2* \check{\phi}_{m+2k}) \left(x+\frac{t^2}{2^{2k}} \right) f_3(x) e^{i \lambda(x) \frac{t^3}{2^{3k}}} \rho(t) \phi \left(\frac{\lambda(x)}{2^{m+3k}} \right) dtdx,  
\end{equation}
where here  $f_i* \check{\phi}_{m+ik}$ refers to the Fourier projection of $f_i$ on the frequency interval $[2^{m+ik},\,2^{m+ik+1}]$, $\lambda (\cdot)$ is some real measurable function taking values of the size $\sim 2^{m+3k}$ and $\rho$ is an odd function with $\supp \ \rho \subseteq \left[-4, 4 \right] \backslash \left[-\frac{1}{4}, \frac{1}{4} \right]$. Throughout this section we let $|f_i| \le \one_{F_i}$, $i\in\{1,2,3\}$, where $F_1$ and $F_2$ are some arbitrary measurable sets, and $F_3$ is any measurable set with positive finite measure. 

One technical subtlety in providing restricted weak type estimates for the above expression is that this task needs to be treated differently based on both the spatial and frequency information of the input functions. This becomes especially important\footnote{See the proof of Theorem \ref{refinedrefined}.}  when dealing with the stationary off-diagonal component $BHC^{\not\Delta, S}$ in Section \ref{20250309sec02}.

\subsubsection{The cases $0 \le k \le 2m$ ($m$ fixed) and $0 \le m \le 2k$ ($k$ fixed): a single-scale argument} \label{20250415subsec01}  In this situation, we prove restricted weak-type estimate directly for the single term $\Lambda_m^k(f_1, f_2, f_3)$ appealing \emph{only} to the spatial information of the input functions $f_1, f_2$, and $f_3$. Indeed, it is enough to notice that
\begin{align} \label{20250430eq02}
\left| \Lambda_m^k(f_1, f_2, f_3) \right|
& \lesssim \int_{\R} \left(\int_{\R} Mf_1 \left(x-\frac{t}{2^k} \right) Mf_2 \left(x+\frac{t^2}{2^{2k}} \right) |\rho(t)|dt \right) |f_3(x)|dx 
\nonumber \\
& \lesssim \int_{\R} \calM_{\calC} (Mf_1, Mf_2)(x) |f_3(x)|dx, 
\end{align}
where we recall that $\calM_{\calC}(f_1, f_2)$ is the \emph{bilinear maximal function along the curve} $(t, t^2)$ (see \eqref{bilinearmax}). Since $\calM_{\calC}(f_1, f_2)$ maps $L^{p_1} \times L^{p_2} \to L^r$ with $\frac{1}{r}=\frac{1}{p_1}+\frac{1}{p_2}$ satisfying $1 \le p_1, p_2 \le \infty$ and $\frac{1}{2}<r \le \infty$, one further has that for any $\theta>0$, there exists some $F_3' \subseteq F_3$ with $|F_3'| \ge \frac{|F_3|}{2}$ such that for $|f_3| \le \one_{F_3'}$, 
\begin{align} \label{20250415eq30Z}
\left| \Lambda_m^k(f_1, f_2, f_3) \right | 
& \lesssim_{\theta} \left\|\calM_{\calC}(Mf_1, M f_2) \right\|_{L^{\frac{1}{2}+\theta}} |F_3|^{\frac{2\theta-1}{2\theta+1}} \lesssim \left\|Mf_1 \right\|_{L^{1+2\theta}} \left\|M f_2 \right\|_{L^{1+2\theta}} |F_3|^{\frac{2\theta-1}{2\theta+1}}  \nonumber \\
& \lesssim_{\theta} \left\|f_1 \right\|_{L^{1+2\theta}} \left\| f_2 \right\|_{L^{1+2\theta}} |F_3|^{\frac{2\theta-1}{2\theta+1}} \le |F_1|^{\frac{1}{1+2\theta}}|F_2|^{\frac{1}{1+2\theta}}|F_3|^{\frac{2\theta-1}{2\theta+1}},
\end{align}
which proves the desired restricted weak-type estimate in this case.

\subsubsection{The case $ k \ge 2m$ ($m$ fixed): a multi-scale argument} \label{20250415subsec02}  In this situation we provide restricted weak-type estimates for the term
    $$
    \Lambda_m^{\ge}(f_1, f_2, f_3):=\sum_{k \ge 2m} \Lambda_m^k(f_1, f_2, f_3)\,,
    $$
by allowing the implicit bounds to depend polynomially on $m$. First of all, via some standard Fourier analytic arguments (see for example the by-product of Rank I LGC in Section \ref{20250228sec03} and also \cite[Section 4.3]{GL25}), we may, without loss of generality, assume that
\begin{align*}
\Lambda_m^{\ge} (f_1, f_2, f_3)
&=\sum_{k \ge 2m} \int_{\R^2} \left(f_1*\check{\phi}_{m+k} \right) \left(x-\frac{t}{2^k}\right)\left(f_2*\check{\phi}_{m+2k}\right)\left(x+\frac{t^2}{2^{2k}}\right) \\
& \qquad \qquad \qquad \qquad  \cdot \left(f_3*\check{\phi}_{m+2k} \right)(x) e^{i\lambda(x)\frac{t^3}{2^{3k}}} \rho(t) \phi \left(\frac{\lambda(x)}{2^{m+3k}} \right)dtdx. 
\end{align*}
Discretizing the form above first for the $x$ variable and then for the $t$ variable, we have 
\begin{eqnarray} \label{20250301eq01}
&& \left|\Lambda_m^{\ge}(f_1, f_2, f_3) \right| \lesssim \sum_{\substack{ k \ge 2m \\ z \in \Z}}  \frac{1}{2^{m+2k}} \int_{\frac{1}{4}}^4 \left| \left(f_1*\check{\phi}_{m+k} \right) \left(\frac{z}{2^{m+2k}}-\frac{t}{2^k} \right) \right| \left| \left(f_2*\check{\phi}_{m+2k} \right) \left(\frac{z}{2^{m+2k}}+\frac{t^2}{2^{2k}} \right) \right| \nonumber \\
&& \qquad \quad \qquad \qquad \quad \qquad \cdot \left| \left(f_3*\check{\phi}_{m+2k} \right) \left(\frac{z}{2^{m+2k}}  \right) \right|dt \nonumber  \\
&&\simeq \frac{1}{2^m} \sum_{k \ge 2m} \sum_{\substack{r \in \Z \\ p \sim 2^m}} \frac{1}{\left|I_r^{m+k} \right|^{\frac{1}{2}}} \left| \left \langle f_1, \Phi_{P_{m+k}(r-p)} \right \rangle \right| \left[ \sum_{I_u^{m+2k} \subseteq I_r^{m+k}} \left| \left \langle f_2, \Phi_{P_{m+2k}\left(u+\frac{p^2}{2^m} \right)} \right \rangle \right| \left| \left \langle f_3, \Phi_{P_{m+2k}(u)} \right \rangle \right| \right] \nonumber \\
&&\simeq \frac{1}{2^m} \sum_{k \ge 2m} \sum_{\substack{r \in \Z \\ p \sim 2^m}} \frac{1}{\left|I_r^{m+k} \right|^{\frac{1}{2}}} \left| \left \langle f_1, \Phi_{P_{m+k}(r-p)} \right \rangle \right| \left[ \sum_{I_u^{m+2k} \subseteq I_r^{m+k}} \left| \left \langle f_2, \Phi_{P_{m+2k}(u)} \right \rangle \right| \left| \left \langle f_3, \Phi_{P_{m+2k}\left(u-\frac{p^2}{2^m}\right)} \right \rangle \right| \right] \nonumber \\
&&\simeq \frac{1}{2^m} \sum_{k \ge 2m} \sum_{\substack{r \in \Z \\ p \sim 2^m}} \frac{1}{\left|I_r^{m+k} \right|^{\frac{1}{2}}} \left| \left \langle f_1, \Phi_{P_{m+k} \left(r-2^{\frac{m}{2}}p^{\frac{1}{2}} \right)} \right \rangle \right| \nonumber   \\
&& \qquad \quad \qquad \qquad \quad \qquad \cdot \left[ \sum_{I_u^{m+2k} \subseteq I_r^{m+k}} \left| \left \langle f_2, \Phi_{P_{m+2k}(u)} \right \rangle \right| \left| \left \langle f_3, \Phi_{P_{m+2k}\left(u-p\right)} \right \rangle \right| \right], 
\end{eqnarray}
where $I_r^{m+k}:=\left[\frac{r}{2^{m+k}}, \frac{r+1}{2^{m+k}} \right]$, $P_{m+k}(v)$ is the time-frequency tile adapted to $\left[\frac{v}{2^{m+k}}, \frac{v+1}{2^{m+k}} \right] \times \left[2^{m+k}, 2^{m+k+1} \right]$, and $\Phi_{P_{m+k}(v)}$ is the $L^2$-normalized Gabor wave packet adapted to it. Moreover, in the last step of the above estimate, we have used the change of variable $p \to 2^{\frac{m}{2}}p^{\frac{1}{2}}$. 

Now by an $\ell^1$--$\ell^\infty$   H\"older in the $p$ parameter, we further have
\begin{eqnarray*}
&& \left| \Lambda_m^{\ge}(f_1, f_2, f_3) \right| \lesssim \sum_{\substack{r \in \Z \\ k \ge 2m}} \frac{1}{\left|I_r^{m+k}\right|^{\frac{1}{2}}} \left(\frac{1}{2^m} \sum_{p \sim 2^m} \left| \left \langle f_1, \Phi_{P_{m+k} \left(r-2^{\frac{m}{2}}p^{\frac{1}{2}} \right)} \right\rangle \right| \right) \\
&& \qquad \quad \qquad \qquad \quad \qquad \cdot \left[ \sum_{I_u^{m+2k} \subseteq I_r^{m+k}} \left| \left \langle f_2, \Phi_{P_{m+2k}(u)} \right \rangle \right| \left| \left \langle f_3, \Phi_{P_{m+2k}\left(u-p \left(I_r^{m+k} \right)\right)} \right \rangle \right| \right], 
\end{eqnarray*}
where 
$$
p(\cdot): \left\{I_r^{m+k} \right\}_{r \in \Z, \; k \ge 2m} \mapsto [2^m, 2^{m+1}] \cap \N \quad \textrm{is any measurable function}.
$$
Assume now $f_1, f_2$, and $f_3$ verify the standard assumptions involved in a restricted weak-type estimate: 
$$
|f_1| \le \one_{F_1}, \quad |f_2| \le \one_{F_2}, \quad \textrm{and} \quad |f_3| \le \one_{F_3'}, 
$$
where
\begin{equation} \label{20250626eq10}
F_3':=F_3 \backslash \Omega\qquad \textrm{with}\qquad\Omega:=\left\{M\one_{F_1} \ge C \frac{|F_1|}{|F_3|}  \right\}\cup \left\{M\one_{F_3} \ge C\right\}
\end{equation}
for some $C>0$ sufficiently large such that $|F_3'| \ge \frac{|F_3|}{2}$. Note here that $\Omega$ is \emph{independent} of the choice of both $m$ and the measurable function $p(\cdot)$. 

Next, for $\beta \in \N$, set 
\begin{equation} \label{20250626eq01}
\calI_\beta:=\left\{I_{\widetilde{r}}^k: 1+\frac{\dist(I_{\widetilde{r}}^k, \; \Omega^c)}{\left|I_{\widetilde{r}}^k\right|} \simeq 2^\beta, \; \widetilde{r} \in \Z \right\}. 
\end{equation} 
This allows us to write 
\begin{equation} \label{20250626eq02}
\left| \Lambda_m^{\ge} (f_1, f_2, f_3) \right| \lesssim \sum_{\beta \in \N} \Lambda_m^{\ge, \beta} (f_1, f_2, f_3), 
\end{equation} 
where
$$
\Lambda_m^{\ge, \beta} (f_1, f_2, f_3):=\sum_{k \ge 2m}  \sum_{I_{\widetilde{r}}^k \in \calI_\beta} \left(\frac{1}{\left| I_{\widetilde{r}}^k \right|} \int_{I_{\widetilde{r}}^k}|f_1| \right) \left[ \sum_{\substack{I_u^{m+2k} \subseteq I_r^{m+k} \\ I_r^{m+k} \subseteq I_{\widetilde{r}}^k}} \left| \left \langle f_2, \Phi_{P_{m+2k}(u)} \right \rangle \right| \left| \left \langle f_3, \Phi_{P_{m+2k}\left(u-p \left(I_r^{m+k} \right)\right)} \right \rangle \right| \right]. 
$$
Using now \eqref{20250626eq01} and Cauchy-Schwarz, we have for each $\beta \in \N$, 
\begin{align} \label{20250626eq03} 
\Lambda_m^{\ge , \beta}(f_1, f_2, f_3) 
& \lesssim  \frac{2^\beta |F_1|}{|F_3|} \cdot \sum_{k \ge 2m}  \sum_{I_{\widetilde{r}}^k \in \calI_\beta} \sum_{\substack{I_u^{m+2k} \subseteq I_r^{m+k} \\ I_r^{m+k} \subseteq I_{\widetilde{r}}^k}} \left| \left \langle f_2, \Phi_{P_{m+2k}(u)} \right \rangle \right| \left| \left \langle f_3, \Phi_{P_{m+2k}\left(u-p \left(I_r^{m+k} \right)\right)} \right \rangle \right| \nonumber \\
& \lesssim  \frac{2^\beta |F_1|}{|F_3|} \cdot  \int_{\R} \sum_{k \ge 2m} \sum_{\substack{I_u^{m+2k} \subseteq I_r^{m+k} \\ I_r^{m+k} \subseteq I_{\widetilde{r}}^k \\ I_{\widetilde{r}}^k \in \calI_\beta}} \frac{ \left| \left \langle f_2, \Phi_{P_{m+2k}(u)} \right \rangle \right|}{\left|I_u^{m+2k} \right|^{\frac{1}{2}}} \cdot \frac{ \left| \left \langle f_3, \Phi_{P_{m+2k}\left(u-p \left(I_r^{m+k} \right)\right)} \right \rangle \right|}{\left|I_u^{m+2k} \right|^{\frac{1}{2}}} \one_{I_u^{m+2k}}(x)dx \nonumber \\
& \lesssim  \frac{2^\beta |F_1|}{|F_3|} \cdot \int_{\R} \calS f_2(x) \calS_{\calI_\beta, p(\cdot)}^{\max} f_3(x) dx, 
\end{align}
where 
$$
\calS_{\calI_\beta, p(\cdot)}^{\max} f_3(x):=\left( \sum_{k \ge 2m} \sum_{\substack{I_u^{m+2k} \subseteq I_r^{m+k} \\ I_r^{m+k} \subseteq I_{\widetilde{r}}^k \\ I_{\widetilde{r}}^k \in \calI_\beta}} \frac{ \left| \left \langle f_3, \Phi_{P_{m+2k}\left(u-p \left(I_r^{m+k} \right)\right)} \right \rangle \right|^2}{\left|I_u^{m+2k} \right|}  \one_{I_u^{m+2k}}(x) \right)^{\frac{1}{2}}.
$$
is the \emph{maximal $m$-shifted square function} relative to the pair $(\calI_\beta, p(\cdot))$. By \cite[(4.166)]{GL25}, 
\begin{equation} \label{20250626eq04}
\left\| \calS_{\calI_\beta, p(\cdot)}^{\max} f_3 \right\|_{L^p(\R)} \lesssim_n m2^{-n\beta} |F_3|^{\frac{1}{p}} 
\end{equation} 
for any $n \in \N$ and $1 < p<+\infty$. Putting \eqref{20250626eq02} -- \eqref{20250626eq04} together, we deduce that for any $\mu \in (0, 1)$,  
$$
\left|\Lambda_m^{\ge}(f_1, f_2, f_3) \right | \lesssim m|F_1||F_2|^\mu |F_3|^{-\mu}, 
$$
which concludes the desired restricted weak-type estimate in this case. 

\subsubsection{The case $ m \ge 2k$ ($k$ fixed): a multi-scale argument}\label{20250430sec01}  In this situation we provide restricted weak-type estimates for the term
$$
\Lambda_{\ge}^k(f_1, f_2, f_3):=\sum_{m \ge 2k} \Lambda_m^k(f_1, f_2, f_3)\,,
$$
by allowing (again) the implicit bounds to depend polynomially on $k$. Since we want to make use of this argumentation for the treatment of the stationary off-diagonal component we employ a three-case treatment depending on the spatial and frequency localization of each of the input functions.

We start by stating the form of the expression that is of interest for us
\begin{equation} \label{20250415eq01}
\Lambda_{\ge}^k (f_1, f_2, f_3)= \sum_{m \ge 2k} \int_{\R^2} \left(f_1*\check{\phi}_{m+k} \right) \left(x-\frac{t}{2^k}\right)\left(f_2*\check{\phi}_{m+2k}\right)\left(x+\frac{t^2}{2^{2k}}\right) f_3(x) e^{i\lambda(x)\frac{t^3}{2^{3k}}} \rho(t) \phi \left(\frac{\lambda(x)}{2^{m+3k}} \right) dtdx.
\end{equation}

\noindent \underline{{\textsf A. Treatment of $\Lambda_{\ge}^k(f_1, f_2, f_3)$ involving only the spatial information of $f_1$ and $f_3$.}}

\vspace{0.2cm}

 By Cauchy-Schwarz, we estimate \eqref{20250415eq01} by 
\begin{align*}
\left|\Lambda_{\ge}^k(f_1, f_2, f_3) \right| 
& \lesssim  \int_{\R} \int_{\R} M f_1 \left(x-\frac{t}{2^k} \right) \calS f_2 \left(x+\frac{t^2}{2^{2k}} \right) |f_3(x)| \left( \sum_{m \ge 2k} \phi^2 \left(\frac{\lambda(x)}{2^{m+3k}} \right) \right)^{\frac{1}{2}} |\rho(t)| dtdx \\
& \lesssim \int_{\R} \left(\int_{\R} Mf_1 \left(x-\frac{t}{2^k} \right) \calS f_2 \left(x+\frac{t^2}{2^{2k}} \right) |\rho(t)| dt \right) |f_3(x)| dx \\
& \lesssim \int_{\R} \calM_{\calC} \left(Mf_1, \calS f_2 \right)(x) |f_3(x)|dx.
\end{align*}
It now suffices to follow the same argument from Section \ref{20250415subsec01} in order to deduce that for any $\theta>0$
\begin{equation} \label{20250430eq01}
\left| \Lambda_{\ge}^k(f_1, f_2, f_3) \right | \lesssim |F_1|^{\frac{1}{1+2\theta}}|F_2|^{\frac{1}{1+2\theta}}|F_3|^{\frac{2\theta-1}{2\theta+1}}. 
\end{equation}

\vspace{0.2cm}

\noindent \underline{\textsf{B. Treatment of $\Lambda_{\ge}^k(f_1, f_2, f_3)$ involving only the spatial information of $f_2$ and $f_3$.}}

\vspace{0.2cm}

The proof of this case follows from a straightforward modification of the previous case, by interchanging the role of $f_1$ and $f_2$. This yields
$$
\left| \Lambda_{\ge}^k(f_1, f_2, f_3) \right| \lesssim \int_{\R} \calM_{\calC}(\calS f_1, Mf_2)(x) |f_3(x)|dx, 
$$
which further gives the same restricted weak-type estimate \eqref{20250430eq01}. 

\vspace{0.2cm}

\noindent  \underline{\textsf{C. Treatment of $\Lambda_{\ge}^k(f_1, f_2, f_3)$ involving only the spatial information of $f_3$.}}

\vspace{0.2cm}
In this case, we apply Cauchy-Schwarz in \eqref{20250415eq01} to the terms involving $f_1$ and $f_2$ to deduce that 
\begin{align*}
\left| \Lambda_{\ge}^k (f_1, f_2, f_3) \right| %
& \lesssim \int_{\R} \int_{\R} \calS f_1 \left(x-\frac{t}{2^k} \right) \calS f_2 \left(x+\frac{t^2}{2^{2k}} \right) |f_3(x)| |\rho(t)| dtdx \\
& \lesssim \int_{\R} \left(\int_{\R} \calS f_1 \left(x-\frac{t}{2^k} \right) \calS f_2 \left(x+\frac{t^2}{2^{2k}} \right) |\rho(t)| dt \right) |f_3(x)| dx \\
& \lesssim \int_{\R} \calM_{\calC} (\calS f_1, \calS f_2)(x) |f_3(x)| dx. 
\end{align*}
The rest of the argument is similar to the previous cases, so we omit further details here.

\subsection{Treatment of the stationary off-diagonal component $BHC^{\not\Delta, S}$: Tame bounds}\label{20250309sec02}

In this second (sub)section, we provide slowly growing restricted weak-type bounds for the sub-components of $BHC^{\not\Delta, S}$ by properly adapting the approach described in the previous section.

We start by recalling the discretization used in Section \ref{20250308sec01} which reduces the analysis of the stationary off-diagonal component to two main cases: 
\begin{enumerate}
    \item [({\bf O1})] $j \cong l, k \ge 0$ and $0 \le m \le \frac{j}{2}$;
    \item [({\bf O2})] all other remaining stationary off-diagonal cases. 
\end{enumerate}

\vspace{0.1cm}

Assume now that $f_1, f_2$, and $f_3$ obey the standard restricted-weak type assumptions---see the introductory paragraph in Section \ref{20250308sec10}. For the case ({\bf O1}), it is not hard to see that the arguments in Sections \ref {20250415subsec01} and \ref{20250415subsec02} still apply, and hence we have
\begin{enumerate}
    \item [$\bullet$] \textsf{multi-scale estimate---summation in $k$:}
    \begin{equation} \label{20250310eq00}
    \sum_{k \ge 2j} \left| \Lambda_{j, l, m, \lambda(\cdot)}^k (f_1, f_2, f_3) \right| \lesssim (j+1) |F_1||F_2|^{\mu}|F_3|^{-\mu}  \quad \textnormal{for any} \ \mu \in (0, 1);
    \end{equation}
\item [$\bullet$] \textsf{single-scale estimate} for any $k \ge 0$ and $j, l, m$ satisfying assumption ({\bf O1}), we have
\begin{equation} \label{20250310eq01}
\left|\Lambda_{j, l, m, \lambda(\cdot)}^k (f_1, f_2, f_3) \right| \lesssim |F_1|^{\frac{1}{1+2\theta}}|F_2|^{\frac{1}{1+2\theta}}|F_3|^{\frac{2\theta-1}{2\theta+1}},  \quad \textnormal{for any} \ \theta>0,
\end{equation}
\end{enumerate}
where $\Lambda_{j, j, m, \lambda(\cdot)}^k (f_1, f_2, f_3)$ is the \emph{dual form} corresponding to the case ({\bf O1}) (see \emph{e.g.} \eqref{20250308eq03}).  

\vspace{0.1cm}

Next, for the case ({\bf O2}), it is sufficient to prove the desired restricted weak-type estimate for the term\footnote{Here, we may slightly relax the assumption of $\lambda$ in Theorem \ref{generalizedmain} by allowing $\lambda(x) \in \left[2^{\frac{m}{2}+3k}, 2^{m+3k} \right]$.} $\widetilde{\Lambda}_m^k(f_1, f_2, f_3)$ defined as in Theorem \ref{generalizedmain}. More precisely, we have
\begin{enumerate}
    \item [$\bullet$] \textsf{multi-scale estimate---summation in $k$}:
    \begin{equation} \label{20250310eq02}
    \left| \sum_{k \ge 2m} \widetilde{\Lambda}_m^k(f_1, f_2, f_3)\right| \lesssim (m+1)^3 |F_1||F_2|^{\mu}|F_3|^{-\mu}  \quad \textnormal{for any} \ \mu \in (0, 1); 
    \end{equation} 
    \item [$\bullet$] \textsf{multi-scale estimate---summation in $m$}:
    \begin{equation} \label{20250415eq40}
    \left| \sum_{m \ge 2k} \widetilde{\Lambda}_m^k(f_1, f_2, f_3)\right| \lesssim (|k|+1)^3 |F_1|^{\frac{1}{1+2\theta}}|F_2|^{\frac{1}{1+2\theta}}|F_3|^{\frac{2\theta-1}{2\theta+1}}, \quad \textnormal{for any} \ \theta>0, 
    \end{equation}
    \item [$\bullet$] \textsf{single scale estimate}: for any $m, k \ge 0$, one has
    \begin{equation} \label{20250310eq03}
    \left| \widetilde{\Lambda}_m^k(f_1, f_2, f_3) \right| \lesssim |F_1|^{\frac{1}{1+2\theta}}|F_2|^{\frac{1}{1+2\theta}}|F_3|^{\frac{2\theta-1}{2\theta+1}}, \quad \textnormal{for any} \ \theta>0. 
    \end{equation} 
\end{enumerate}

\noindent \textit{Sketch of the proof.} We first note that the argument in Section \ref{20250415subsec01} works for both single-scale estimates \eqref{20250310eq01} and \eqref{20250310eq03}. For the estimate \eqref{20250310eq00}, it is enough to observe that the argument in Section \ref{20250415subsec02} still applies if we switch the roles of $j$ and $m$. So, it remains to handle the two multi-scale estimates in case ({\bf O2}).

For the estimate \eqref{20250310eq02}, recall that $m$ is fixed and the sum is taken over all $k \ge 2m$. We consider two sub-cases: \begin{enumerate} \item[\textit{(a-1)}.] $\supp \ \widehat{f_1} \subseteq [2^k, 2^{m+k+1}]$ and $\supp \ \widehat{f_2} \subseteq [2^{m+2k}, 2^{m+2k+1}]$; \item[\textit{(a-2)}.] $\supp \ \widehat{f_1} \subseteq [2^{m+k}, 2^{m+k+1}]$ and $\supp \ \widehat{f_2} \subseteq [2^{2k}, 2^{m+2k+1}]$. \end{enumerate} In both of these sub-cases, the argument from Section \ref{20250415subsec02} still works. More precisely, in case \textit{(a-1)}, the only change to make is to observe that the input function $f_1$ is supported on about $m$ dyadic frequency intervals, and hence in the case the magnitude of the constant in \eqref{20250310eq02} is $ \lesssim (m+1)^2$. Next, in the case \textit{(a-2)}, since $k \ge 2m$, we observe that $\supp \ \widehat{f_2} \subseteq [2^{2k}, 2^{m+2k+1}]$, and hence both $f_2$ and $f_3$ are supported on about $m$ dyadic frequency intervals, and hence the magnitude of the constant of $ \lesssim (m+1)^3$. 

Next, for the estimate \eqref{20250415eq40}, note that $k$ is fixed and the sum is taken over $m \ge 2k$. Let us break the analysis into four sub-cases: \begin{enumerate} \item[\textit{(b-1)}.] $\supp \ \widehat{f_1} \subseteq [2^k, 2^{-100} \cdot 2^{m+k}]$ and $\supp \ \widehat{f_2} \subseteq [2^{m+2k}, 2^{m+2k+1}]$; \item[\textit{(b-2)}.] $\supp \ \widehat{f_1} \subseteq [2^{m+k}, 2^{m+k+1}]$ and $\supp \ \widehat{f_2} \subseteq [2^{2k}, 2^{-100} \cdot 2^{m+2k}]$; \item[\textit{(b-3)}.] $\supp \ \widehat{f_1} \subseteq [2^{-100} \cdot 2^{m+k}, 2^{m+k+1}]$ and $\supp \ \widehat{f_2} \subseteq [2^{m+2k}, 2^{m+2k+1}]$; \item[\textit{(b-4)}.] $\supp \ \widehat{f_1} \subseteq [2^{m+k}, 2^{m+k+1}]$ and $\supp \ \widehat{f_2} \subseteq [2^{-100} \cdot 2^{m+2k}, 2^{m+2k+1}]$. \end{enumerate}

In case \textit{(b-1)}, the assumption ({\bf O2}) allows us to assume that $\textnormal{range}(\lambda) \subseteq [2^{m+3k}, 2^{m+3k}]$, so the argument from Case \textsf{A} in Section \ref{20250430sec01} applies. For case \textit{(b-2)}, the reasoning in Case \textsf{B} from Section \ref{20250430sec01} is valid. Finally, in cases \textit{(b-3)} and \textit{(b-4)}, we can use the argument from Case \textsf{C} in the same section. \hfill $\square$

\subsection{Putting together all the estimates: A unified treatment of $BHC^{\Delta}$ and $BHC^{\not\Delta, S}$}

In this section, our goal is to prove the full desired quasi-Banach range for both $BHC^{\Delta}$ and $BHC^{\not\Delta, S}$. These bounds will follow from an interpolation argument between the $L^2$ estimates established in Theorem \ref{generalizedmain} and the restricted weak-type estimates proved in Sections
\ref{20250308sec10} and \ref{20250309sec02}.

\subsubsection{Preliminaries} With these settled, we start by noticing that based on the above and with the previous notations the dual forms that correspond to the two major cases to treat are given by 
$$
\Lambda_{\textnormal{(\bf O-1})}(f_1, f_2, f_3):=\left \langle BHC_{\lambda(\cdot)}^{\textnormal{({\bf O1})}}(f_1, f_2), f_3 \right \rangle:=\sum_{j, k \ge 0} \sum_{\substack{l \cong j \\ 0 \le m \le \frac{j}{2}}} \Lambda_{j, l, m, \lambda(\cdot)}^k(f_1, f_2, f_3),
$$
where $\Lambda_{j, l, m, \lambda(\cdot)}^k(f_1, f_2, f_3)$ is a single piece of the dual form with $j, l, m$ and $k$ fixed (see \eqref{20250308eq03}) and
$$
\Lambda_{\textnormal{(\bf O-2})}(f_1, f_2, f_3):=\left \langle BHC_{\lambda(\cdot)}^{\textnormal{({\bf O2})}}(f_1, f_2), f_3 \right \rangle :=\sum_{m, k \ge 0}  \widetilde{\Lambda}_m^k(f_1, f_2, f_3),
$$
where $ \widetilde{\Lambda}_m^k(f_1, f_2, f_3)$ is a single piece of the ``extended" diagonal component  defined in Theorem \ref{generalizedmain}. 

Our final \emph{goal} is to establish the full quasi-Banach range bounds for both these forms. In order to do so, we first collect all the key estimates we have already obtained: 
\begin{enumerate}
    \item [(1)] {\bf $L^2$-estimates}: there exists some absolute $\epsilon>0$ sufficiently small, such that for any $(p_1, p_2, p_3)=(2, \infty, 2), (\infty, 2, 2)$, and $(2, 2, \infty)$, one has
    \begin{enumerate}
        \item [(a)] under assumption ({\bf O1})---see \eqref{20250210eq01}:
        $$
        \left|\Lambda_{j, l, m, \lambda(\cdot)}^k(f_1, f_2, f_3) \right| \lesssim 2^{-\epsilon j} \prod_{i=1}^3 \left\|f_i \right\|_{L^{p_j} \left(3I^k \right)}\,;
        $$
      
        \item [(b)] under assumption ({\bf O2})---see Theorem \ref{generalizedmain}: 
        \begin{equation} \label{20250312eq01}
        \left| \widetilde{\Lambda}_m^k(f_1, f_2, f_3) \right| \lesssim 2^{-\epsilon \min\{2k, m\}} \prod_{i=1}^3 \left\|f_i \right\|_{L^{p_j} \left(3I^k \right)}\,;
        \end{equation} 
       
    \end{enumerate}
    \item [(2)] {\bf Restricted weak-type estimates}: for $F_1, F_2$, and $F_3$ measurable sets with finite (non-zero) Lebesgue measure, there exists some 
    $$
    F_3' \subseteq F_3 \quad \textnormal{measurable with} \quad |F_3'| \ge \frac{|F_3|}{2}\,,
    $$
    such that for any functions $f_1, f_2$, and $f_3$ obeying
    \begin{equation} \label{20250314eq01}
    |f_1| \le \one_{F_1}, \quad |f_2| \le \one_{F_2}, \quad \textrm{and} \quad |f_3| \le \one_{F_3'}, 
    \end{equation} 
    one has 
    \begin{enumerate}
        \item [(c)] under assumption ({\bf O1})
       
            \noindent $\bullet$ \textsf{multi-scale estimate---summation in $k$}, see \eqref{20250310eq00}): 
            \begin{equation} \label{summary001}
            \sum_{k \ge 2j} \left| \Lambda_{j, l, m, \lambda(\cdot)}^k(f_1, f_2, f_3) \right| \lesssim (j+1)^2|F_1||F_2|^{\mu}|F_3|^{-\mu}\, \quad \textrm{for any} \ \mu \in(0,1 )\, ; 
            \end{equation}

            \noindent $\bullet$ \textsf{single-scale estimate}---for any $k \ge 0$, see \eqref{20250310eq01}:
              \begin{equation} \label{summary002}
        \left| \Lambda_{j, l, m, \lambda(\cdot)}^k(f_1, f_2, f_3) \right| \lesssim |F_1|^{\frac{1}{1+2\theta}}|F_2|^{\frac{1}{1+2\theta}}|F_3|^{\frac{2\theta-1}{2\theta+1}},  \quad \textnormal{for any} \ \theta>0\,;
        \end{equation}
        
       \item [(d)] under assumption ({\bf O2}): 
        
         \noindent $\bullet$ \textsf{multi-scale estimate---summation in $k$},  see \eqref{20250310eq02}:
            \begin{equation} \label{summary003}
           \sum_{k \ge 2m}  \left| \widetilde{\Lambda}_m^k(f_1, f_2, f_3) \right| \lesssim (m+1)^3|F_1||F_2|^{\mu}|F_3|^{-\mu}\, \quad \textrm{for any} \ \mu \in (0, 1);
            \end{equation}

             \noindent $\bullet$ \textsf{multi-scale estimate---summation in $m$}, see \eqref{20250415eq40}:
            \begin{equation} \label{summary004}
             \sum_{m \ge 2k} \left|\widetilde{\Lambda}_m^k(f_1, f_2, f_3) \right| \lesssim (|k|+1)^3 |F_1|^{\frac{1}{1+2\theta}}|F_2|^{\frac{1}{1+2\theta}}|F_3|^{\frac{2\theta-1}{2\theta+1}},  \qquad \textnormal{for any} \ \theta>0\,; 
            \end{equation}
            
             \noindent $\bullet$ \textsf{single-scale estimate}---for any $k \ge 0$, see \eqref{20250310eq03}:
            \begin{equation} \label{summary005}
            \left|  \widetilde{\Lambda}_m^k(f_1, f_2, f_3) \right|   \lesssim |F_1|^{\frac{1}{1+2\theta}}|F_2|^{\frac{1}{1+2\theta}}|F_3|^{\frac{2\theta-1}{2\theta+1}},  \quad \textnormal{for any} \ \theta>0\,.
            \end{equation}
           \end{enumerate}
\end{enumerate}

\subsubsection{A refinement of \eqref{20250312eq01}} In this brief section we solve the following issue: estimate \eqref{20250312eq01} is \emph{not} sufficient for summing in both the $k$ and $m$ parameters. The main reason is that Theorem \ref{generalizedmain}, as stated, lacks frequency localization in $f_3$. This inconvenience is addressed in the following refinement of Theorem \ref{generalizedmain}:

\begin{thm} \label{refinedrefined}
For any $k, m \in \N$, let
$$
BHC_{m, \lambda, \textnormal{ref}}^k(f_1, f_2)(x):=\sup_{\lambda \in \left[2^{\frac{m}{2}+\kappa m+3k}, 2^{m+3k} \right]} \left|  \int_{\R} f_1(x-t)f_2(x+t^2)e^{i\lambda t^3} 2^k \rho(2^k t) dt \right|
$$
where $\kappa>0$ is the same constant that we picked in Theorem \ref{generalizedmain}, and further denote its dual form by
$$
\widetilde{\Lambda}_{m, \lambda, \textnormal{ref}}^k(f_1, f_2, f_3):= \left\langle BHC_{m, \lambda, \textnormal{ref}}^k(f_1, f_2), f_3 \right \rangle. 
$$
Assume one of the following situations holds:
    \begin{enumerate}
        \item [(1)] $\supp \ \widehat{f_1} \subseteq \left[2^k, 2^{m+k+1} \right]$ and $\supp \ \widehat{f_2} \subseteq \left[2^{m+2k}, 2^{m+2k+1} \right]$;
        \item [(2)] or $\supp \ \widehat{f_1} \subseteq \left[2^{m+k}, 2^{m+k+1} \right]$ and $\supp \ \widehat{f_2} \subseteq \left[2^{2k}, 2^{m+2k+1} \right]$.
    \end{enumerate}
Moreover, if either $\widehat{f_1}$ is supported away from $[2^{m+k-3}, 2^{m+k+1}]$ or $\widehat{f_2}$ is supported away from $[2^{m+2k-3}, 2^{m+2k+1}]$ then one may further assume that the linearizing supremum phase function obeys $\lambda (x) \in [2^{m+3k}, 2^{m+3k+1}]$.

Under the above assumptions, the following hold:
\begin{enumerate}
\item [(a)] assuming $m \le k+20$, then
\begin{enumerate}
   \item [(a-1)] if $\supp \ \widehat{f_1} \subseteq 
[2^k, 2^{m+k+1} ]$ and $\supp \ \widehat{f_2} \subseteq [2^{m+2k}, 2^{m+2k+1}]$ then one may assume the frequency localization $\supp \ \widehat{f_3} \subseteq [2^{m+2k}, 2^{m+2k+1}]$;
\item [(a-2)] if $\supp \ \widehat{f_1} \subseteq [2^{m+k}, 2^{m+k+1}]$ and $\supp \ \widehat{f_2} \subseteq \left[2^{2k}, 2^{m+2k+1} \right]$ then one may assume the frequency localization $\supp \ \widehat{f_3} \subseteq [0, 2^{m+2k+1}]$;
\end{enumerate}
\item [(b)] assuming $m \ge k+20$, then  
\begin{enumerate}
    \item [(b-1)] if $\supp \ \widehat{f_1} \subseteq 
[2^k, 2^{m+k+1} ]$ and $\supp \ \widehat{f_2} \subseteq [2^{m+2k}, 2^{m+2k+1}]$ then one may assume the frequency localization $\supp \ \widehat{f_3} \subseteq [2^{m+2k}, 2^{m+2k+1}]$;
\item [(b-2)] if $\supp \ \widehat{f_1} \subseteq [2^{m+k}, 2^{m+k+1}]$ and $\supp \ \widehat{f_2} \subseteq \left[2^{2k}, 2^{m+k-10} \right]$ then one may assume the frequency localization $\supp \ \widehat{f_3} \subseteq [2^{m+k-1}, 2^{m+k+1}]$;
\item [(b-3)] if $\supp \ \widehat{f_1} \subseteq [2^{m+k}, 2^{m+k+1}]$ and $\supp \ \widehat{f_2} \subseteq [2^{m+k-10}, 2^{m+2k+1}]$ then one may assume the frequency localization\footnote{The frequency localization of $f_3$ will \emph{not} be relevant here, as we will prove the $L^2$ estimates for $(p_1, p_2, p_3)=(2, 2, \infty)$ in this case.} $\supp \ \widehat{f_3} \subseteq [0, 2^{m+2k+1}]$. 
\end{enumerate}
\end{enumerate}
Moreover, \eqref{20250312eq01} remains valid for $\widetilde{\Lambda}_{m, \lambda, \textnormal{ref}}^k(f_1, f_2, f_3)$ in all the above cases. 
\end{thm}

\begin{proof}
The fact that \eqref{20250312eq01} holds in all of the aforementioned cases follows from a careful inspection of the proof of Theorem \ref{generalizedmain}. Regarding the claims on the frequency localization of $f_3$ we notice that these can be verified using a standard Rank-I LGC argument as outlined in Section \ref{20250228sec03}. We leave the further details to the interested reader.
\end{proof}

\subsubsection{$L^r$ estimates for $\Lambda_{\textnormal{(\bf O-1})}$ and $\Lambda_{\textnormal{(\bf O-2)}}$ within the quasi-Banach range}

In the first part of this section we address the Hilbert case, \emph{i.e.}, we establish $L^p$ bounds for $\Lambda_{\textnormal{(\bf O-1)}}(f_1, f_2, f_3)$ and $\Lambda_{\textnormal{(\bf O-2)}}(f_1, f_2, f_3)$ when $(p_1, p_2, p_3)$ is restricted to the triples $(2, \infty, 2), (\infty, 2, 2)$, and $(2, 2, \infty)$. 

First, for case ({\bf O1}), it suffices to consider 
$$
\Lambda_{\textnormal{(\bf O-1)}}(f_1, f_2, f_3) \simeq \sum_{j\geq 0} \Lambda_{\textnormal{(\bf O-1)}}^{j}(f_1, f_2, f_3):=\sum_{j\geq 0} \sum_{k\ge 0} \sum_{0 \le m \le \frac{j}{2}} \Lambda_{j, j, m, \lambda(\cdot)}^k(f_1, f_2, f_3), 
$$
and then, using \eqref{20250210eq01} and Cauchy-Schwarz, we have 
\begin{align} \label{20250415eq50}
\left| \Lambda_{\textnormal{(\bf O-1)}}^{j}(f_1, f_2, f_3) \right| 
& \le\sum_{0 \le m \le \frac{j}{2}} \left| \sum_{k \ge 0} \Lambda_{j, j, m, \lambda(\cdot)}^k(f_1, f_2, f_3) \right| \lesssim j\,2^{-\epsilon j} \prod_{i=1}^3 \left\|f_i \right\|_{L^{p_i}(\R)}, 
\end{align}
where $(p_1, p_2, p_3)=(2, 2, \infty), (2, \infty, 2)$, and $(\infty, 2, 2)$. 

Next, for case ({\bf O2}), guided by the refined Theorem \ref{refinedrefined}, we write
\begin{align} \label{20250503eq43}
\Lambda_{\textnormal{(\bf O-2)}}(f_1, f_2, f_3) 
& = \Lambda_{\textnormal{(\bf O-2-a-1)}}(f_1, f_2, f_3)+\Lambda_{\textnormal{(\bf O-2-a-2)}}(f_1, f_2, f_3)+\Lambda_{\textnormal{(\bf O-2-b-1)}}(f_1, f_2, f_3) \nonumber  \\
& \quad +\Lambda_{\textnormal{(\bf O-2-b-2)}}(f_1, f_2, f_3)+\Lambda_{\textnormal{(\bf O-2-b-3)}}(f_1, f_2, f_3),
\end{align}
where
$$
\Lambda_{\textnormal{(\bf O-2-a-1)}}(f_1, f_2, f_3):=\sum_{m \ge 0} \Lambda^{m}_{\textnormal{(\bf O-2-a-1)}}(f_1, f_2, f_3):= \sum_{m \ge 0} 
\sum_{k \ge m-20} \widetilde{\Lambda}_{m, \lambda, \textnormal{ref}}^k \left( f_1^{[2^k, 2^{m+k+1}]}, f_2^{[2^{m+2k}, 2^{m+2k+1}]}, f_3 \right),
$$
$$
\Lambda_{\textnormal{(\bf O-2-a-2)}}(f_1, f_2, f_3):=\sum_{m \ge 0} \Lambda^{m}_{\textnormal{(\bf O-2-a-2)}}(f_1, f_2, f_3):=\sum_{m\geq 0}\sum_{k \ge m-20} \widetilde{\Lambda}_{m, \lambda, \textnormal{ref}}^k \left( f_1^{[2^{m+k}, 2^{m+k+1}]}, f_2^{[2^{2k}, 2^{m+2k-10}]}, f_3 \right),
$$
$$
    \Lambda_{\textnormal{(\bf O-2-b-1)}}(f_1, f_2, f_3):=\sum_{k \ge 0}  \Lambda_{\textnormal{(\bf O-2-b-1)}}^{k}(f_1, f_2, f_3):=\sum_{k\geq 0}\sum_{m \ge k+20} \widetilde{\Lambda}_{m, \lambda, \textnormal{ref}}^k \left( f_1^{[2^k, 2^{m+k+1}]}, f_2^{[2^{m+2k}, 2^{m+2k+1}]}, f_3 \right),
    $$
$$
    \Lambda_{\textnormal{(\bf O-2-b-2)}}(f_1, f_2, f_3):=\sum_{k \ge 0}  \Lambda_{\textnormal{(\bf O-2-b-2)}}^{k}(f_1, f_2, f_3):=\sum_{k\geq 0}\sum_{m \ge k+20} \widetilde{\Lambda}_{m, \lambda, \textnormal{ref}}^k \left( f_1^{[2^{m+k}, 2^{m+k+1}]}, f_2^{[2^{2k}, 2^{m+k-10}]}, f_3 \right),
    $$
and
$$
\Lambda_{\textnormal{(\bf O-2-b-3)}}(f_1, f_2, f_3):=\sum_{k \ge 0}\Lambda_{\textnormal{(\bf O-2-b-3)}}^{k}(f_1, f_2, f_3):=\sum_{k\geq 0} \sum_{m \ge k+20} \widetilde{\Lambda}_{m, \lambda, \textnormal{ref}}^k \left( f_1^{[2^{m+k}, 2^{m+k+1}]}, f_2^{[2^{m+k-10}, 2^{m+2k+1}]}, f_3 \right).
$$
Here, we recall that $f^{I}$ refers to the Fourier projection of $f$ on the frequency interval $I$. 

For the case ({\bf O-2-a-1}), we have
\begin{equation*}
\left| \Lambda^{m}_{\textnormal{(\bf O-2-a-1)}}(f_1, f_2, f_3)\right| 
\lesssim \sum_{k \ge m-20} \left| \widetilde{\Lambda}_{m, \lambda, \textnormal{ref}}^k \left( f_1^{[2^k, 2^{m+k+1}]}, f_2^{[2^{m+2k}, 2^{m+2k+1}]}, f_3 \right) \right| \lesssim  2^{-\epsilon m}\,\prod_{i=1}^3 \left\|f_i \right\|_{L^{p_i}(\R)}\,,  
\end{equation*} 
where $(p_1, p_2, p_3)=(2, 2, \infty), (2, \infty, 2)$, and $(\infty, 2, 2)$. Note that the same argument applied to the term $\Lambda^{m}_{\textnormal{(\bf O-2-a-2)}}(f_1, f_2, f_3)$, yields the same smoothing estimate. Next, for $\Lambda_{\textnormal{(\bf O-2-b-1)}}^{k}(f_1, f_2, f_3)$, from case (b-1) in Theorem \ref{refinedrefined}, we have 
\begin{align} \label{20250415eq34}
& \left| \Lambda_{\textnormal{(\bf O-2-b-1)}}^{k}(f_1, f_2, f_3) \right|  \le  \sum_{m \ge k+20} \left| \widetilde{\Lambda}_{m, \lambda, \textnormal{ref}}^k \left( f_1^{[2^k, 2^{m+k+1}]}, f_2^{[2^{m+2k}, 2^{m+2k+1}]}, f_3 \right) \right| \nonumber \\
& \lesssim 2^{-2 \epsilon k}  \left(\sum_{m>k+20} \left\|f_1^{[2^k, 2^{m+k+1}]} \right\|_{L^\infty(\R)} \left\|f_2^{[2^{m+2k}, 2^{m+2k+1}]} \right\|_{L^2(\R)} \left\|f^{[2^{m+2k}, 2^{m+2k+1}]}_3 \right\|_{L^2(\R)} \right)  \nonumber \\
& \lesssim  2^{-2\epsilon k} \left\|f_1 \right\|_{L^\infty} \left\|f_2 \right\|_{L^2(\R)} \left\|f_3 \right\|_{L^2(\R)}\,.
\end{align}
Note that in this case, the $L^2$ estimate is \emph{only} valid for $(p_1, p_2, p_3)=(\infty, 2, 2)$. Finally, for the terms $\Lambda_{\textnormal{(\bf O-2-b-2)}}^{k}(f_1, f_2, f_3)$ and $\Lambda_{\textnormal{(\bf O-2-b-3)}}^{k}(f_1, f_2, f_3)$, one can apply a similar argument as in \eqref{20250415eq34}, now utilizing the frequency localizations specified in cases (b-2) and (b-3), respectively, of Theorem~\ref{refinedrefined}.  The corresponding $L^p$ smoothing estimate indices in these cases are $(p_1, p_2, p_3)=(2, \infty, 2)$, and $(2, 2, \infty)$, respectively. 

These conclude the $L^2$-base decay estimates that now can be uses as a base-point for multilinear interpolation with the tamer bounds described below (the proof of the latter require only standard adaptations of the reasonings involved for proving \eqref{summary001}---\eqref{summary005}):

\begin{lem} \label{20250317lem01}
Let $f_1$, $f_2$, $f_3$ satisfy the assumptions \eqref{20250314eq01} with the major set $F_3'$ defined by \eqref{20250626eq10}. Then, for any 
$(p_1, p_2, p_3)$ satisfying $\frac{1}{p_1}+\frac{1}{p_2}=\frac{1}{p_3'}$ with $1 < p_1, p_2< +\infty$ and $\frac{1}{2}<p_3'<\infty$
the following estimates hold:
$$
 \sum_{k \ge 0}  \left|\Lambda_{j, j, m, \lambda(\cdot)}^k(f_1, f_2, f_3) \right| \lesssim (j+1)^2 \prod_{i=1}^3 \left\|f_i \right\|_{L^{p_i}(\R)}, \quad \sum_{k \ge 0}\left| \widetilde{\Lambda}_m^k(f_1, f_2, f_3) \right| \lesssim (m+1)^3 \prod_{i=1}^3 \left\|f_i \right\|_{L^{p_i}(\R)}, 
$$
and
$$
\sum_{m \ge 0} \left| \widetilde{\Lambda}_m^k(f_1, f_2, f_3) \right| \lesssim (|k|+1)^3 \prod_{i=1}^3 \left\|f_i \right\|_{L^{p_i}(\R)}.
$$
\end{lem}

Combining now the main $L^2$ decay estimates with Lemma \ref{20250317lem01} and using multi-linear interpolation, we conclude the desired (maximal) quasi Banach range estimates: 
$$
\left|\Lambda_{\textnormal{(\bf O-1})}(f_1, f_2, f_3) \right|, \; \left| \Lambda_{\textnormal{(\bf O-2})}(f_1, f_2, f_3) \right| \lesssim \prod_{i=1}^3 \left\|f_i \right\|_{L^{p_i}(\R)}\,
$$
for any triple $(p_1, p_2, p_3)$ satisfying the hypothesis in the statement of Lemma \ref{20250317lem01}.

\section{Appendix: Treatment of the case $k<0$: an outline}

In this section, we provide a brief outline for the treatment of the main diagonal term $BHC_{-}^{\Delta}$ representing, of course, the crux of the entire $k<0$ analysis and which further crucially relies on the Rank-II LGC methodology. To begin with, we recall that the operator we wish to understand is given by 
$$
 BHC_{-}^{\Delta}(f_1, f_2)\approx\sum_{k \in \Z_{-}} \sum_{m \in \N} \sup_{\lambda} \left|BHC_{m, \lambda}^k(f_1, f_2)(x) \right|,
$$
with
\begin{equation} \label{20241215eq10}
BHC_{m, \lambda}^k(f_1, f_2)(x) \approx 2^k \phi \left(\frac{\lambda}{2^{m+3k}} \right) \int_{\R} f_{1, m+k}(x-t) f_{2, m+2k}(x+t^2)e^{i\lambda t^3} \rho(2^k t) dt,
\end{equation} 
where $\phi$ is a smooth even function supported in $[-4, 4] \backslash \left[-\frac{1}{4}, \frac{1}{4} \right]$, and, wlog $\rho\in C_0^{\infty}(\R)$ with $\textrm{supp}\,\rho\subset[\frac{1}{10},10]$. 

In what follows, we focus entirely only on the key differences relative to the case $k \ge 0$:
\begin{enumerate}
    \item [$\bullet$] the dominant frequency regime corresponds to $f_1$, having the magnitude $2^{m+k}$;
    \item [$\bullet$] the dominant spatial interval in the $x$-variable becomes $I^{2k}:=I_0^{0, 2k}$ versus $I^k$ in the $k>0$ regime;
    \item [$\bullet$] the roles of $f_1$ and $f_2$ are interchanged. 
\end{enumerate}

With these being said, we first apply the change of variable $t^2 \to t$ in order to rewrite \eqref{20241215eq10} as
$$
BHC_{m, \lambda}^k(f_1, f_2)(x) \approx 2^{2k} \phi \left(\frac{\lambda}{2^{m+3k}} \right) \int_{\R} f_{2, m+2k}(x+t) f_{1, m+k}\left(x-t^{\frac{1}{2}} \right) e^{i\lambda t^{\frac{3}{2}}} \widetilde{\rho} \left(2^{2k} t \right)dt,
$$
where $\widetilde{\rho}(t):=\frac{t^{-1/2}}{2} \rho(t^{1/2})$. Observe that the above operator is concentrated locally on intervals of length $2^{-2k}$, and hence it suffices to consider 
$$
BHC_{m, \lambda}^{k, I^{2k}}(f_1, f_2)(x) \approx 2^{2k} \one_{I^{2k}}(x) \phi_2 \left(\frac{\lambda}{2^{m+3k}} \right) \int_{\R} f_{2, m+2k}(x+t) f_{1, m+k}\left(x-t^{\frac{1}{2}} \right) e^{i\lambda t^{\frac{3}{2}}} \widetilde{\rho} \left(2^{2k} t \right)dt.
$$
Let again $\lambda (\cdot): I^{2k} \to \left[2^{m+3k}, 2^{m+3k+1} \right]$ be a measurable function such that 
$$
\sup_{\lambda \in \R} \left| BHC_{m, \lambda}^{k, I^{2k}}(f_1, f_2)(x)\right| \le 2 \left|BHC_{m, \lambda(x)}^{k, I^{2k}}(f_1, f_2)(x) \right|.
$$
Next, by using Heisenberg's uncertainty principle, the analogue of {\bf Key Heuristic 1} becomes the following: 

\bigskip 

\noindent\textbf{Key Heuristic 1$^\prime$.}

\begin{center}
\texttt{$\lambda(\cdot)$ takes constant values on intervals of length $2^{-m-k}$ when $k<0$}. 
\end{center}

\bigskip 

It now becomes transparent that the natural analogue of Theorem \ref{20241018mainthm01} becomes:
\bigskip 

\begin{thm} \label{20241217thm01A}
Let $k \in \Z_{<0}$ and $m \in \N$. Let further, 
\begin{equation} \label{20241215eq02}
\underline{\Lambda}_m^k(f_1, f_2, f_3):=2^{2k} \int_{I^{2k}} \int_{I^{2k}} f_2(x+t)f_1(x-t^{\frac{1}{2}}) f_3(x) e^{i\lambda(x) t^{\frac{3}{2}}} \widetilde{\rho} \left(2^{2k} t \right) dt, 
\end{equation}
be the dual form of $BHC_{m, \lambda(\cdot)}^{k, I^{2k}}(f_1, f_2)$, where $\lambda(\cdot)$ takes constant value on intervals of length $2^{-m-k}$.

Then, there exists some $\underline{\epsilon}>0$, such that 
$$
\left |\underline{\Lambda}_m^k(f_1, f_2, f_3) \right| \lesssim 2^{-\underline{\epsilon} \min\{-2k, m\}} \prod_{j=1}^3 \left\|f_j \right\|_{L^{p_j}(3I^{2k})},
$$
where $(p_1, p_2, p_3)=(2, 2, \infty), (2, \infty, 2)$ or $(\infty, 2, 2)$. 
\end{thm}
\bigskip 

Below we only provide a sketch of the proof for the above statement and leave further details to the interested reader. 
\medskip 

Consider now the dual form of $BHC_{m, \lambda(x)}^{k, I^{2k}}(f_1, f_2)(x)$, that is, 
\begin{equation} \label{20241215eq02A}
\underline{\Lambda}_m^k(f_1, f_2, f_3):=2^{2k} \int_{I^{2k}} \int_{I^{2k}} f_2(x+t)f_1(x-t^{\frac{1}{2}}) f_3(x) e^{i\lambda(x) t^{\frac{3}{2}}} \widetilde{\rho} \left(2^{2k} t \right) dt, 
\end{equation}
 By Cauchy-Schwarz and a Taylor series argument, we see that 
\begin{equation} 
\left| \underline{\Lambda}_m^k(f_1, f_2, f_3) \right|^2 \lesssim 2^{2k} \left\|f_3 \right\|_{L^2(3I^{2k})}^2 \left[ \underline{\bf L}_{m, 2k}(f_1, f_2) \right]^2,  \label{20240429eq01A}
\end{equation} 
where 
\begin{equation} \label{20241216eq01A}
\left[ \underline{\bf L}_{m, 2k}(f_1, f_2) \right]^2:=2^{2k} \int_{I^{2k}}\left|\sum_{q \sim 2^{\frac{m}{2}}} \int_{I_q^{m, 2k}} f_2(x+t)f_1(x-t^{\frac{1}{2}})e^{\frac{3i \lambda(x)}{2} \cdot \frac{q^{\frac{1}{2}}}{2^{\frac{m}{4}+k}}t} dt \right|^2 dx. 
\end{equation} 
Following the strategy outlined in Section \ref{20241216subsec01}, one employs a sparse-uniform dichotomy procedure that reduces the analysis of \eqref{20241004eq01} to the following expression
\begin{eqnarray} \label{20241216eq02}
\left[\underline{\calL}_{m, 2k}(f_1, f_2) \right]^2%
&=& 2^{\frac{m}{2}+2k} \int_{I^{2k}} \sum_{q \sim 2^{\frac{m}{2}}} \left| \int_{I_q^{m, 2k}} f_2(x+t)f_1(x-t^{\frac{1}{2}})e^{\frac{3i \cdot 2^{\frac{m}{2}+2k}}{2} \cdot \left(\frac{q}{2^{\frac{m}{2}}}\right)^{\frac{1}{2}} \widetilde{\lambda}(x) t} dt \right|^2 dx
\end{eqnarray}
where both $f_1$ and $f_2$ in \eqref{20241216eq02} are uniform, in the sense that
$$
\left\|f_i \right\|_{L^\infty(3I^{2k})} \lesssim 2^{\frac{\mu \min\{m, 2|k|\}}{2}} \cdot 2^k \left\|f_i \right\|_{L^2(3I^{2k})}, \quad i\in\{1, 2\}, 
$$
for some $\mu>0$ sufficiently small and $\widetilde{\lambda}(x):=\frac{\lambda(x)}{2^{\frac{m}{2}+3k}}$. Note that $\widetilde{\lambda}(\cdot): I^{2k} \mapsto \left[2^{\frac{m}{2}}, 2^{\frac{m}{2}+1} \right]$ takes constant value on intervals of length $2^{-m-k}$. The analysis of \eqref{20241216eq02} is similar to the treatment of the term \eqref{20241004eq01}. More precisely, we again divide the proof into two parts:

\medskip 

 \noindent   \underline{\textsf{Part I: $k \le -\frac{m}{2}$.}} In this case, following the strategy presented in \cite[Theorem 4.3]{HL23}:  

\vspace{0.1cm}

\begin{enumerate}

\item [$\bullet$] If $k \le -m$, then the proof of the uniform component follows from the same argument in \cite[Theorem 11.1]{HL23}. 

\item [$\bullet$] If $-m \le k \le -\frac{m}{2}$, then the whole argument can be divided into two steps.

\noindent \underline{\textsf{Step I.1}:} we apply a bootstrap argument (see, \cite[Proposition 4.7]{HL23}) in order to extend the length of constancy of $\widetilde{\lambda}$ to the level $2^{-m-2k}$.

\vspace{0.1cm}

\noindent \underline{\textsf{Step I.2}:} once we reach the scale $2^{-m-2k}$, note that in this case $2^{-m-2k} \ge 2^{-\frac{m}{2}-k}$, and hence we apply another bootstrap argument to extend the length of constancy of $\widetilde{\lambda}$ to the level $2^{-\frac{m}{2}-2k}$, which matches the scale of $t$-intervals (that is, the maximal scale that $t^{\frac{1}{2}}$ behaves like a linearized function), and hence the argument in \cite[Proposition 4.9]{HL23} applies. 

\end{enumerate}

\vspace{0.1cm}

\noindent $\bullet$ \textit{Part II: $-\frac{m}{2} \le k \le 0$.} Again, this is the case that is not covered by \cite[Theorem 4.3]{HL23} and hence one has to apply a new idea in the spirit of the proof developed in Sections \ref{20241217subsec01} and \ref{20250228sec03}. 

\vspace{0.1cm}

\noindent \underline{\textsf{Step II.1}:} mirroring \underline{\textsf{Step I.1}} above, we need to apply a bootstrap argument in the spirit of Proposition \ref{20241018prop01} in otder to extend the length of constancy of $\widetilde{\lambda}$ to the level $2^{-m-2k}$.

\vspace{0.1cm}

\noindent \underline{\textsf{Step II.2}:} once we reach the scale $2^{-m-2k}$, instead of applying another bootstrap argument, we implement Rank-I LGC methodology to \eqref{20241215eq02}.  

\medskip 

We end our brief outlook with the following remark: in the case $-m \le k \le 0$ in order to fully make use of the hidden cancellation encoded in the curvature, one needs to subdivide the integral expression in $2^{m}$ many $x$-intervals of the length $2^{-m-2k}$ and then group together $2^{m+k}$ such consecutive intervals. More precisely, the bootstrap argument is applied to the operator
\begin{equation} 
\calV_{m, k}^{\widetilde{p}}(f_1, f_2):=2^{-k} \sum_{p \in \left[\widetilde{p} 2^{m+k}, (\widetilde{p}+1)2^{m+k} \right]} \left[ \frac{1}{2^m} \calV_{m, k}^p (f_1, f_2) \right] \label{20241217eq10A},
\end{equation} 
where here $\widetilde{p} \sim 2^{-k}$ and for each $p \sim 2^m$ one has
\begin{equation}
\calV_{m, k}^p (f_1, f_2):=2^{\frac{3m}{2}+2k} \sum_{q \sim 2^{\frac{m}{2}}} \int_{I_p^{0, m+2k}} \left| \int_{I_q^{m, 2k}} f_2(x+t)f_1(x-t^{\frac{1}{2}})e^{\frac{3i \cdot 2^{\frac{m}{2}+2k}}{2} \cdot \left(\frac{q}{2^{\frac{m}{2}}}\right)^{\frac{1}{2}} \widetilde{\lambda}(x) t} dt \right|^2 dx. \label{20241004eq21A}
\end{equation}

\end{document}